\documentclass[a4paper,11pt,twoside]{article}
\usepackage[left=2cm,right=2cm,top=2cm,bottom=3cm]{geometry}

\usepackage[T1]{fontenc}
\usepackage[utf8]{inputenc} 
\usepackage[english]{babel} 

\usepackage[super]{nth}  

\usepackage{comment}
\usepackage{xcolor}

\usepackage[usestackEOL]{stackengine}
\usepackage{amsmath} 
\usepackage{amssymb} 
\usepackage{amsthm}
\usepackage{amsfonts}
\usepackage{mathtools}
\usepackage{mathrsfs}
\usepackage{braket} 
\usepackage{graphicx}
\usepackage{graphics}
\usepackage{stmaryrd}
\usepackage{textcomp}
\usepackage[export]{adjustbox}
\usepackage{cancel}
\usepackage{mathabx,epsfig}
\usepackage{cases}
\usepackage[overload]{empheq}
\usepackage{esvect}
\usepackage{nicefrac}
\usepackage{mathtools}
\mathtoolsset{showonlyrefs=true,showmanualtags=true}  

\usepackage[all]{xy}

\newcommand{\mycomment}[1]{}

\newcommand{\numberset}{\mathbb}
\newcommand{\N}{\numberset{N}}

\newcommand{\R}{\numberset{R}}
\newcommand{\T}{\numberset{T}}
\newcommand{\Z}{\numberset{Z}}
\newcommand{\C}{\numberset{C}}
\newcommand{\E}{\numberset{E}}

\newcommand{\F}{\mathcal{F}}  

\newcommand{\de}{\text{d}}
\DeclareMathOperator{\p}{\numberset{P}} 
\DeclareMathOperator{\low}{\ell}
\DeclareMathOperator{\roeta}{\rho^{\epsilon,\eta}}

\DeclareMathOperator*{\essinf}{ess\,inf}

\newcommand{\esssup}{\text{ess sup}}

\renewcommand{\and}{\quad\textrm{ and }\quad}

\renewcommand{\P}{\mathbb{P}}
\renewcommand{\a}{\alpha}

\newcommand{\TT}{\mathbb{T}}

\newcommand{\norm}[1]{\left\| #1 \right\|}

\def\XXint#1#2#3{{\setbox0=\hbox{$#1{#2#3}{\int}$ }
\vcenter{\hbox{$#2#3$ }}\kern-.6\wd0}}

\theoremstyle{definition}
\newtheorem{definition}{Definition}[section]
\newtheorem{theorem}[definition]{Theorem}
\newtheorem{lemma}[definition]{Lemma}
\newtheorem{corollary}[definition]{Corollary}
\newtheorem{proposition}[definition]{Proposition}

\newtheorem{example}[definition]{Example}

\newtheorem{remark}[definition]{Remark}

\newtheorem{assumption}[definition]{Assumption}
\newtheorem*{notation*}{Notation}
\newtheorem*{definition*}{Definition}
\newtheorem*{theorem*}{Theorem}
\newtheorem*{lemma*}{Lemma}
\newtheorem*{corollary*}{Corollary}
\newtheorem*{proposition*}{Proposition}
\newtheorem*{fact*}{Fact}
\newtheorem*{example*}{Example}
\newtheorem*{claim*}{Claim}
\newtheorem*{remark*}{Remark}
\newtheorem*{conjecture*}{Conjecture}

\numberwithin{equation}{section}

\allowdisplaybreaks[4]

\title{A central limit theorem for nonlinear conservative SPDEs}
\author{Andrea Clini \thanks{Mathematical Institute, University of Oxford, Oxford OX2 6GG, UK (andrea.clini@maths.ox.ac.uk)} \and Benjamin Fehrman \thanks{Department of Mathematics, Louisiana State University, Baton Rouge 70803, USA (fehrman@math.lsu.edu)}}
\date{\today} 

\begin{document}
\maketitle
\begin{abstract}
    We prove a central limit theorem characterizing the small noise fluctuations of stochastic PDEs of fluctuating hydrodynamics type.  The results apply to the case of nonlinear and potentially degenerate diffusions and irregular noise coefficients including the square root.
    In several cases, the fluctuations of the solutions agree to first order with the fluctuations of certain interacting particle systems about their hydrodynamic limits.  
\end{abstract} 



\section{Introduction}
\label{section/Introduction}

In this paper we study the small noise fluctuations of nonnegative solutions $\rho^{\epsilon}$ to conservative, stochastic PDEs of the type
\begin{equation}\label{equation/generalized dean--kawasaki equation 1}
    \partial_t\rho^{\epsilon}=\Delta\phi(\rho^\epsilon)\textcolor{black}{-\nabla\cdot\nu(\rho^\epsilon)}-\sqrt{\epsilon}\nabla\cdot\left(\sigma(\rho^\epsilon)\circ \dot{\xi}^{\epsilon}\right)\,\,\,\text{in } \mathbb{T}^d\times(0,T),\quad \rho^\epsilon(\cdot,0)=\rho_0\,\,\,\text{in } \mathbb{T}^d\times\{0\},
\end{equation}
about their zero noise limit $\Bar{\rho}$ solving the equation
\begin{equation}\label{equation/hydrodynamic limit equation 1}
    \partial_t\Bar{\rho}=\Delta\phi(\bar{\rho})\textcolor{black}{-\nabla\cdot\nu(\bar{\rho})}\,\,\,\text{in } \mathbb{T}^d\times(0,T),\quad \bar{\rho}(\cdot,0)=\rho_0\,\,\,\text{in } \mathbb{T}^d\times\{0\}.
\end{equation}
We prove that the fluctuations $\epsilon^{-\frac{1}{2}}(\rho^\epsilon-\Bar{\rho})$ converge in probability, in a space of distributions, to the solution of the linearized Langevin equation
\begin{equation}\label{equation/OU equation 1}
    \partial_tv=\Delta\big(\dot{\phi}(\bar{\rho})\,v\big)-\nabla\cdot\left(\dot{\nu}(\bar{\rho})\,v+\sigma(\bar{\rho})\circ \dot{\xi}\right)\,\,\,\text{in } \mathbb{T}^d\times(0,T),\quad v(\cdot,0)=0\,\,\,\text{in } \mathbb{T}^d\times\{0\},\end{equation}
which is the linearization of \eqref{equation/generalized dean--kawasaki equation 1} around $\bar{\rho}$.

The assumptions on the noise sequence $\xi^\epsilon$, which can converge to a space-time white noise or some other cylindrical noise $\xi$ as $\epsilon\to0$, on the initial data $\rho_0$ and on the nonlinearities $\phi$, $\nu$ and $\sigma$ are given in Section \ref{section/hypotheses and notations}, and include a range of relevant stochastic PDEs (see Example \ref{example/applicability of coefficient assumptions}).
In particular, they apply to the full range of porous media diffusions, that is $\phi(z)=z^m$ for every $m\in(0,\infty)$, and to degenerate convective terms $\sigma$ and $\nu$, including the square root $\sigma(z)=\sqrt{z}$, and hence to the nonlinear version of the Dean--Kawasaki equation with correlated noise
\begin{equation}
    \partial_t\rho^{\epsilon}=\Delta(\rho^\epsilon)^{m}-\sqrt{\epsilon}\,\nabla\cdot\left(\sqrt{\rho^\epsilon}\circ \dot{\xi}^{\epsilon}\right).
\end{equation}
%
Equations of type \eqref{equation/hydrodynamic limit equation 1} and \eqref{equation/generalized dean--kawasaki equation 1} arise in the fluctuating hydrodynamics of particle systems, to describe respectively the limiting behavior and the nonequilibrium fluctuations of particle systems, in the context of mean field theory with common noise and in the area of stochastic geometric PDEs.
Applications to particle systems are discussed in Section \ref{section/comments on the literature} and we refer to \cite[Section 1.2]{fehrman-gess-Well-posedness-of-the-Dean-Kawasaki-and-the-nonlinear-Dawson-Watanabe-equation-with-correlated-noise} for a detailed overview of other applications.

The wellposedness of \eqref{equation/generalized dean--kawasaki equation 1} has been a long-standing open problem.
In the case the noise $\xi^\epsilon$ has sufficient space correlation, existence and uniqueness for \eqref{equation/generalized dean--kawasaki equation 1} have been established  only recently in \cite{fehrman-gess-Well-posedness-of-the-Dean-Kawasaki-and-the-nonlinear-Dawson-Watanabe-equation-with-correlated-noise}, through the concept of stochastic kinetic solution (cf. Definition \ref{definition/stochastic kinetic solution} and Theorem \ref{theorem/existence and uniquess of kinetic solutions to dean-kawasaki} below).

The study of the small noise behaviour of \eqref{equation/generalized dean--kawasaki equation 1} has been initiated in \cite{fehrman-gess-large-deviations-for-conservative}.
The authors proved that, in the $\epsilon\to0$ limit, solutions of \eqref{equation/generalized dean--kawasaki equation 1} converge to solutions of the zero-noise limit \eqref{equation/hydrodynamic limit equation 1}:
\begin{equation}\label{equation/small noise limit}
    \lim_{\epsilon\to0}\|\rho^\epsilon-\bar{\rho}\|_{L^1([0,T];L^1(\T^d))}=0\,\,\,\text{in probability}.
\end{equation}
Furthermore they satisfy a large deviation principle in $L^1([0,T];L^1(\T^d))$ with rate function
\begin{equation}
\label{equation/rate function 1}
    I_{\rho_0}(\rho)=\inf\left\{\|g\|^2_{L^2(\TT^d\times[0,T])^d} \,:\, \partial_t\rho=\Delta\phi(\rho)-\nabla\cdot\nu(\rho)-\nabla\cdot\left(\sigma(\rho)g\right)\,\,\text{and}\,\, \rho(\cdot,0)=\rho_0\right\}.
\end{equation}

\noindent
Finally the authors rigorously identified the rate function  \eqref{equation/rate function 1} with that governing the large deviations of the zero range particle process (cf. \cite[Theorem 6.8 and 8.6]{fehrman-gess-large-deviations-for-conservative}), first studied in \cite{benois-kipnis-landim-large-deviations-from-the-hydrodynamical-limit-of-mean-zero-asymmetric}, formalizing the connection between the particle system and the above SPDEs.
\medskip

The contribution of this paper is to complete the picture above by proving a central limit theorem characterizing the fluctuations of \eqref{equation/generalized dean--kawasaki equation 1} around its deterministic limit \eqref{equation/hydrodynamic limit equation 1}.
See Section \ref{section/comments on the literature} for an overview of the relevant literature and related results.
Precisely, consider the normalized fluctuations $v^\epsilon\coloneqq\epsilon^{-\nicefrac{1}{2}}\left(\rho^\epsilon-\bar{\rho}\right)$, which are readily seen to solve
{
\begin{align}[left ={\empheqlbrace}]\label{equation/fluctuation equation 1}
\begin{split}
    &\partial_tv^{\epsilon}=\Delta\left(\epsilon^{-\nicefrac{1}{2}}\Big(\phi(\rho^\epsilon)-\phi(\bar{\rho})\Big)\right)-\nabla\cdot\left(\epsilon^{-\nicefrac{1}{2}}\left(\nu(\rho^\epsilon)-\nu(\bar{\rho})\right)+\sigma(\rho^\epsilon)\circ \dot{\xi}^{\epsilon}\right)\,\,\,\text{in } \mathbb{T}^d\times(0,T),
    \\
    &v^\epsilon(\cdot,0)=0\,\,\,\text{on } \mathbb{T}^d\times\{0\}.
\end{split}
\end{align}
}
Since $\rho^\epsilon\to\bar{\rho}$ and $\xi^\epsilon\to\xi$ (cf. Assumption \ref{assumption/assumption N2}) as $\epsilon\to0$, one formally deduces that $v^\epsilon$ should converge to the solution of the Langevin equation \eqref{equation/OU equation 1}.  The main result of the present work is to make this ansatz rigorous.

\begin{theorem*}[Theorem \ref{theorem/clt for fluctuations 1} and \ref{theorem/clt for fluctuations 3} below]
Let $(\xi^\epsilon)_{\epsilon>0}$ satisfy Assumption \ref{assumption/assumption N2}.
Let $\phi, \nu, \sigma$ satisfy Assumption \ref{assumption/assumption C1} and \ref{assumption/assumption C2 weak}.
Let $\rho_0\in L^{\infty}(\Omega;(0,\infty))$ be an $\F_0$-measurable random constant bounded away from zero, so that $\bar{\rho}(x,t)\equiv \rho_0$ solves \eqref{equation/hydrodynamic limit equation 1}.
For each $\epsilon>0$, let $\rho^\epsilon$ be the stochastic kinetic solution to \eqref{equation/generalized dean--kawasaki equation 1} with initial data $\rho_0$, in the sense of Definition \ref{definition/stochastic kinetic solution}.
Let $v\in L^2(\Omega\times[0,T];H^{-\alpha}(\T^d))$ be the solution of \eqref{equation/OU equation 1} with noise $\xi=\lim_{\epsilon\to0}\xi^\epsilon$, for any $\alpha>\frac{d}{2}$.
For any $T>0$, for $\tau=2$ or $\tau=\infty$, for any $\beta>\frac{d}{2}$ or $\beta>1+\frac{d}{2}$ respectively, we have the following results.

\begin{enumerate}
    \item [i)]Along a suitable scaling regime where $\epsilon\to0$ and $\xi^\epsilon\to\xi$, explicitly given in \eqref{equation/scaling regime for noise}, the nonequilibrium fluctuations satisfy
        \begin{equation}\label{equation/clt in probability 1}
            v^\epsilon:=\epsilon^{-\nicefrac{1}{2}}(\rho^\epsilon-\Bar{\rho})\to v \quad\text{in $L^{\tau}([0,T];H^{-\beta}(\T^d))$ in probability},
        \end{equation}
    with an explicit rate of convergence, given in \eqref{formula/ rate of convergence clt in probability}, which depends on $\beta$, $T$, the coefficients $\phi,\nu,\sigma$, the initial data $\rho_0$, and the space regularity of the noise sequence $\xi^\epsilon$.
     
    \item[ii)] In addition, if the coefficients $\phi,\nu,\sigma$ also satisfy Assumption \ref{assumption/assumption C2}, along the same scaling regime \eqref{equation/scaling regime for noise}, we have
        \begin{align}\label{equation/rate of convergence 1}
            v^\epsilon\to v \quad\text{in $L^2(\Omega;L^{\tau}([0,T];H^{-\beta}(\T^d)))$},
    \end{align}
    with an explicit rate of convergence, given in \eqref{equation/rate of convergence 2}, which depends on $\beta$, $T$, the coefficients $\phi,\nu,\sigma$, the initial data $\rho_0$, and the space regularity of the noise sequence $\xi^\epsilon$.
\end{enumerate}

\end{theorem*}

The results of this paper and \cite{fehrman-gess-large-deviations-for-conservative} have immediate consequences on the simulation of several particle systems.
Indeed, upon choosing the right coefficients $\phi$, $\nu$ and $\sigma$ according to the hydrodynamics of the particle process considered, equation \eqref{equation/generalized dean--kawasaki equation 1} provides a continuum model which correctly captures the fluctuations of the particle system up to order one and exhibits the same nonequilibrium large deviations.
This is discussed in detail in Section \ref{section/comments on the literature}.

\subsection{Structure of the paper}
\label{section/structure of the paper}

The paper is organized as follows.
We end this section with a discussion of the applications to particle systems, followed by an overview of our methods and of the relevant literature.
In Section \ref{section/Preliminaries and solution theory} we first lay out our notations and assumptions; then we present the solution theory for the equations involved: the zero noise limit \eqref{equation/hydrodynamic limit equation 1}, the stochastic PDE \eqref{equation/OU equation 1} for the asymptotic fluctuations and the original SPDE \eqref{equation/generalized dean--kawasaki equation 1}.
In Section \ref{section/The central limit theorem for fluctuations} we finally prove our central limit theorem.
First we prove the stronger version \eqref{equation/rate of convergence 1} for coefficients satisfying the extra Assumption \ref{assumption/assumption C2}; then we extend this to rougher coefficients only satisfying Assumption \ref{assumption/assumption C2 weak} and prove the central limit theorem in probability \eqref{equation/clt in probability 1}.


\subsection{Applications to particle systems, methods and relevant literature}
\label{section/comments on the literature}


\noindent
\emph{\textbf{Applications to particle systems.}}

\noindent
A byproduct of the above results is that the SPDE \eqref{equation/generalized dean--kawasaki equation 1} provides a continuum model correctly simulating the fluctuations of several conservative particle systems up to order one and exhibiting the same nonequilibrium large deviations.
Consider for example a symmetric zero range process on the torus with mean local jump rate $\phi$ and slowly varying initial state $\rho_0$.
As the number $n$ of particles increases, the parabolically rescaled empirical measure $\mu^n$ converges in probability, weakly in the sense of measure, to its hydrodinamic limit: the deterministic measure $\bar{\rho}\,\de x$ solving
\begin{equation}\label{equation/application to particles/hydrodynamic limit}
    \partial_t\Bar{\rho}=\Delta\phi(\bar{\rho})\,\,\,\text{in } \mathbb{T}^d\times(0,T),\quad \bar{\rho}(\cdot,0)=\rho_0\,\,\,\text{in } \mathbb{T}^d\times\{0\}.
\end{equation}
In this way, solutions to the porous media equation \eqref{equation/application to particles/hydrodynamic limit} describe the dynamics of the particle process up to order zero (see e.g. \cite[Chapter 5]{kipnis-landim-scaling-limits}).

In order to reach higher order continuum approximations, it is necessary to incorporate the fluctuations present in $\mu^n$.
A detailed study of the nonequilibrium central limit fluctuations is presented in \cite[Chapter 11]{kipnis-landim-scaling-limits}, where it is shown that the measures
\begin{equation}
    \label{equation/application to particles/fluctuations}
    \text{m}^n\coloneqq\sqrt{n}\left(\mu^n-\bar{\rho}\,dx\right)
\end{equation}
converge as $n\to\infty$ to the solution of the linear stochastic PDE
\begin{equation}\label{equation/application to particles/OU equation}
    \partial_t\text{m}=\Delta\big(\dot{\phi}(\bar{\rho})\,\text{m}\big)-\nabla\cdot\Big(\phi^{\nicefrac{1}{2}}(\bar{\rho})\circ \dot{\xi}\Big)\,\,\,\text{in } \mathbb{T}^d\times(0,T),\quad \text{m}(\cdot,0)=0\,\,\,\text{in } \mathbb{T}^d\times\{0\},
\end{equation}
for $\xi$ a space-time white noise and $\bar{\rho}$ solving \eqref{equation/application to particles/hydrodynamic limit}.
Since solutions to \eqref{equation/application to particles/OU equation} are not function-valued the convergence is in the sense of distributions.
Furthermore, the results of \cite{benois-kipnis-landim-large-deviations-from-the-hydrodynamical-limit-of-mean-zero-asymmetric} and \cite{fehrman-gess-large-deviations-for-conservative} prove that the non-equilibrium large deviations of $\mu^n$ are described in terms of the rate function 
\begin{equation}
\label{equation/application to particles/rate function }
    I_{\rho_0}(\rho)=\inf\left\{\|g\|^2_{L^2(\TT^d\times[0,T])^d} \,:\, \partial_t\rho=\Delta\phi(\rho)-\nabla\cdot\left(\phi^{\nicefrac{1}{2}}(\rho)\,g\right)\,\,\text{and}\,\, \rho(\cdot,0)=\rho_0\right\}.
\end{equation}

Consider now the the nonlinear stochastic PDE
\begin{equation}\label{equation/application to particles/continuum model 1}
    \partial_t \rho^n=\Delta\phi(\rho^n)-\frac{1}{\sqrt{n}}\nabla\cdot\left(\phi^{\nicefrac{1}{2}}(\rho^n)\circ\dot{\xi}^n\right),\quad\rho^n(\cdot,0)=\rho_0,
\end{equation}
for suitable approximations $\{\xi^n\}_{n\in\N}$ of an $\R^d$-valued space-time white noise (cf. Assumption \ref{assumption/assumption N2}). 
Our main result Theorem \ref{theorem/clt for fluctuations 3} ensures that, as $n\to\infty$, the small noise fluctuations $\sqrt{n}(\rho^n-\bar{\rho})$ converge to the solution $\text{m}$ of \eqref{equation/application to particles/OU equation}.
Therefore, this and \eqref{equation/application to particles/fluctuations}-\eqref{equation/application to particles/OU equation} give us the expansion
\begin{equation}
    \label{equation/application to particles/first order expansion}
    \mu^n=\rho^n\,dx+\text{o}\left(\frac{1}{\sqrt{n}}\right),
\end{equation}
which now correctly captures the fluctuations of $\mu^n$ up to order one.
Furthermore $\rho^n$ also features the same large deviation rate function \eqref{equation/application to particles/rate function } of $\mu^n$, as demonstrated by \eqref{equation/rate function 1} from \cite[Theorem 6.8]{fehrman-gess-large-deviations-for-conservative}.

Equation \eqref{equation/application to particles/continuum model 1} is an approximation for the formal nonlinear SPDE
\begin{equation}\label{equation/application to particles/continuum model 2}
    \partial_t \rho^n=\Delta\phi(\rho^n)-\frac{1}{\sqrt{n}}\nabla\cdot\left(\phi^{\nicefrac{1}{2}}(\rho^n)\circ\dot{\xi}\right).
\end{equation}
The approximation is needed because of the irregularity of the space-time white noise $\xi$.
Indeed, equation \eqref{equation/application to particles/continuum model 2} is supercritical in the language of regularity structures \cite{Hairer-regulairity-structures}; its intrinsic ill-posedness and negative results have been discussed in the seminal work \cite{lehmann_konarovsky_vonrenesse_On_deankawasaki_dynamics_smooth_drift}.
In fact, it can also be argued (see e.g. \cite[Section 3]{giacomin-lebowitz-presutti-Deterministic-and-stochastic-hydrodynamic-equations-arising-from-simple-microscopic-model-systems}) that the microscopic particle system comes with a typical correlation length for the noise, like the grid size, which leads to consider equation \eqref{equation/application to particles/continuum model 1} for some space correlated noise $\xi^n$ instead of equation \eqref{equation/application to particles/continuum model 2}.

Finally, it is worth pointing out that one could also consider the simpler first order expansion
\begin{equation}
    \label{equation/application to particles/first order expansion fake}
    \mu^n=\bar{\rho}\,\de x+\frac{1}{\sqrt{n}}\text{m}+\text{o}\left(\frac{1}{\sqrt{n}}\right),
\end{equation}
which follows immediately from \eqref{equation/application to particles/fluctuations}-\eqref{equation/application to particles/OU equation}.
However $\bar{\rho}^n\coloneqq\bar{\rho}\,\de x+\frac{1}{\sqrt{n}}\text{m}$ exhibits a rate function different from \eqref{equation/application to particles/rate function }, given by
\begin{equation}
\label{equation/wrong rate function for first order expansion}
    \bar{I}_{\rho_0}(\rho)=\inf\left\{\|g\|^2_{L^2(\TT^d\times[0,T])^d} \,:\, \partial_t\left(\rho-\bar{\rho}\right)=\Delta\left(\dot{\phi}(\bar{\rho})(\rho-\bar{\rho})\right)-\nabla\cdot\left(\phi^{\nicefrac{1}{2}}(\bar{\rho})\,g\right),\quad \rho(\cdot,0)=\rho_0\right\},
\end{equation}
as follows from Schilder's theorem and a formal application of the contraction principle.
Therefore this expansion does not capture the large deviations of $\mu^n$ and yields an imprecise approximation.

In conclusion we remark that completely analogous results hold for more general zero range processes and exclusion processes.
Essentially the results hold for all the particle processes whose fluctuations can be described through SPDEs falling in the framework above. 
We refer the reader to the monographs \cite{kipnis-landim-scaling-limits,spohn-large-scale-dynamics} and the surveys \cite{giacomin-lebowitz-presutti-Deterministic-and-stochastic-hydrodynamic-equations-arising-from-simple-microscopic-model-systems} for results in this sense.


\bigskip

\noindent
\emph{\textbf{Overview of the methods.}}

\noindent
As mentioned above, a concept of solution for \eqref{equation/generalized dean--kawasaki equation 1} with an actual white noise seems to be out of hand.
If the noise $\xi^\epsilon$ has some space correlation (cf. Assumption \ref{assumption/assumption N1}), the well-posedness of equation \eqref{equation/generalized dean--kawasaki equation 1} is proved in \cite{fehrman-gess-Well-posedness-of-the-Dean-Kawasaki-and-the-nonlinear-Dawson-Watanabe-equation-with-correlated-noise}.
The main difficulties in applying a classical concept of weak solutions are due to nonlinearities that are possibly only $\nicefrac{1}{2}$-H\"older continuous and singular terms which are not even known to be locally integrable.
These issues are addressed by the the notion of \emph{stochastic kinetic solution} (cf. Definition \ref{definition/stochastic kinetic solution}).
After rewriting the Stratonovich equation \eqref{equation/generalized dean--kawasaki equation 1} in the equivalent It\^o form, the equation is recast in its \emph{kinetic formulation}: an equation in the original space and time variables and in a new additional \emph{velocity} variable, corresponding to the magnitude of the solution.
Then a renormalization away from zero and infinity is introduced: solutions are required to satisfy the PDE only after cutting out small and large values in order to enforce the local integrability and further regularity of the nonlinear terms. 
In fact, when the equation coefficients are nice enough, this renormalization is not even needed and stochastic kinetic solutions satisfy the equation in a classical weak sense (cf. Definition \ref{definition/weak solution to generalized dean kawasaki} and Proposition \ref{proposition/kinetic solutions implies weak solution to deak-kawasaki}).
Furthermore, stochastic kinetic solutions depend continuously on the equation coefficients in a suitable sense (cf. Proposition \ref{proposition/kinetic solutions of dean-kawasaki depend continuously on the coefficients}).
These two properties will be crucial in our arguments.
We refer the reader to \cite[Section 1.1]{fehrman-gess-Well-posedness-of-the-Dean-Kawasaki-and-the-nonlinear-Dawson-Watanabe-equation-with-correlated-noise} for more details on their methods.

The small noise large deviations of \eqref{equation/generalized dean--kawasaki equation 1} are analyzed in \cite{fehrman-gess-large-deviations-for-conservative}.
Schilder's Theorem and a formal application of the contraction principle lead to guess the rate function \eqref{equation/rate function 1}.
After a careful analysis of the \emph{energy critical} PDE featuring in the definition of $I_{\rho_0}$, this ansatz is made rigourous via the so-called \emph{weak approach to large deviations} \cite{DupEll1997,budhiraja-dupuis-maroula-large-deviations,BudDup2019,BudDupSal}.

We finally come to the criticalities and the methods of this paper.
A first difficulty is that the compactness of the noise-to-solution map of \eqref{equation/generalized dean--kawasaki equation 1} is not enough to show convergence of the fluctuations $v^\epsilon=\epsilon^{-\nicefrac{1}{2}}(\rho^\epsilon-\bar{\rho})$. 
Indeed, although it is true that $\rho^\epsilon\to\bar{\rho}$ in a suitable sense, the previous arguments yield no information about the convergence rate; whereas we need to quantify this to compensate for the blowing up factor $\epsilon^{-\nicefrac{1}{2}}$. 
Moreover, the high nonlinearity and degeneracy of the equation hinder the application of Fourier analysis.
Finally, the renormalization away from small and large values in the kinetic formulation of the equation and the possible lack of regularity coming from the degeneracy of the diffusion make it difficult to exploit the equation for $\rho^\epsilon$ in our computations.
The picture is further complicated by the irregularity and unboundedness of the noise coefficient $\sigma$.

To address these issues, we consider approximating versions 
of equation \eqref{equation/generalized dean--kawasaki equation 1} with smoothed coefficients $\phi^\eta,\nu^\eta,\sigma^\eta$ indexed by a parameter $\eta\in(0,1)$, for which a stronger notion of solution is available (cf. Definition \ref{definition/weak solution to generalized dean kawasaki}).
We always work with the corresponding solutions $\rho^{\epsilon,\eta}$, which have enough regularity to justify our computations, and then pass the results to the true solutions $\rho^\epsilon$ either with Fatou's Lemma or with probabilistic arguments based on analytical estimates.

We begin with the stronger version \eqref{equation/rate of convergence 1} of the CLT for coefficients satisfying the additional Assumption \ref{assumption/assumption C2}.
This already applies to the model case $\phi(z)=z^{m}$, $\sigma(z)=z^{\frac{m}{2}}$ in the regime $m\in[2,\infty)$, or to the linear case $\phi(z)=z$ for suitable noise coefficients $\sigma$.

First we obtain estimates on the $L^h$-norm of the fluctuations $v^\epsilon$, for some $h$ big enough depending on the nonlinearities involved.
Since $v^\epsilon$ is converging to $v$, which is only distribution valued, we expect these estimates to blow up as $\epsilon\to0$ and we want to quantify the explosion rate in terms of the small parameter $\epsilon$.
Formally, this is achieved by applying It\^o formula to the power $|v^\epsilon|^h$ and exploiting the equation satisfied by $v^\epsilon$.
Equipped with these moment estimates, we estimate the Fourier coefficients of $v^{\epsilon}-v$ and show that $\|v^\epsilon-v\|_{L^{\infty}([0,T];H^{-\beta}(\T^d))}$ vanishes in $L^2(\Omega)$ as $\epsilon\to0$.

To consider even rougher coefficients, including $\sigma(z)=\sqrt{z}$ and $\phi(z)=z^m$ for every $m\in(0,\infty)$, some extra work is needed.
With a Moser iteration argument, adapted from \cite{dirr-fehrman-gess-conservative-stochastic-pde-and-fluctuations-of-the-symmetric}, we show that the solutions $\rho^\epsilon$ of \eqref{equation/generalized dean--kawasaki equation 1} are bounded from below by the minimum of the \emph{positive} initial data $\rho_0$ with increasing probability as $\epsilon\to0$ (cf. Proposition \ref{proposition/moser iteration} and Corollary \ref{corollary/solutions are bounded by initial data with high probability}).
In particular, solutions $\rho^\epsilon$ stay away from the irregularities at zero of $\phi$, $\nu$, $\sigma$ on the events $\omega$ of a certain subset $\Omega^\epsilon\subseteq\Omega$ where they satisfy $\rho^\epsilon(x,t,\omega)\in[\inf\rho_0-\delta,\infty)$ for a.e. $x$ and $t$.
A refinement of the pathwise uniqueness result for \eqref{equation/generalized dean--kawasaki equation 1} shows that on $\Omega^\epsilon$ the solution $\rho^\epsilon$ must coincide with the solution $\rho^{\epsilon,\eta}$ of the smoothed version of equation \eqref{equation/generalized dean--kawasaki equation 1}
with coefficients $\phi^\eta,\nu^\eta,\sigma^\eta$ that match the true coefficients $\phi,\sigma,\eta$ sufficiently away from the irregularity points (cf. Propositition \ref{proposition/enhanced pathwise uniqueness}).
For such solutions $\rho^{\epsilon,\eta}$ we have the above CLT in $L^2(\Omega)$ at our disposal and with standard probabilistic arguments we pass this to a CLT in probability for the true solutions $\rho^\epsilon$.

{\color{red}

}



\bigskip

\noindent
\emph{\textbf{Overview of the literature.}}

\noindent

Linear and nonlinear diffusion equation with different kinds of noise terms have received a lot of attention, from several point of view.  We refer to \cite{fehrman-gess-Well-posedness-of-the-Dean-Kawasaki-and-the-nonlinear-Dawson-Watanabe-equation-with-correlated-noise} for a detailed overview of the literature and we mention only the most recent results on the equation considered here \cite{cornalba_fischer_ingmanns_raithel_density_fluctuations_in_weakly_interacting, fehrman_gess_gvalani_ergodicity,djurdjevac-kremp-perkowski-weak-error-analysis,clini-porous-media-equations,wang_wu_zhang_dean_kawasaki_equation_with_singular_nonlocal_interactions}.  Similarly, we refer to the introduction of \cite{fehrman-gess-large-deviations-for-conservative} for the literature on large deviations principles on related SPDEs.

We collect here instead some results concerning central limit theorems for SPDEs.  Namely, central limit theorems for stochastic heat equations or stochastic wave equations with Lipschitz continuous noise coefficients have been obtained in \cite{Chen_nualart_pu_poincare_inequality_and_central_limit,huang-nualart-viitasaari-a-central-limit-theorem,huang-nualart-viitasaari-gaussian-fluctuations} and in \cite{vences-nualart-zheng-central-limit-theorem-for-the-stochastic-wave} respectively.
A central limit theorem for the heat equation driven by nonlinear gradient noise with H\"older continuous coefficient has been proved by Dirr, Gess and the second author in \cite{dirr-fehrman-gess-conservative-stochastic-pde-and-fluctuations-of-the-symmetric}.
Finally, central limit theorems for primitive equations, a specific type of semilinear evolution equation, in low dimension have been established in \cite{Slavik-large-and-moderate-deviations,zhang-zhou-guo-stochastic-2d-primitive-equations-central-limit,Hu-li-wang-central-limit}.


\section{Preliminaries and solution theory}
\label{section/Preliminaries and solution theory}


\subsection{Hypotheses and notations}
\label{section/hypotheses and notations}

In this section we present our assumptions and notations, and we remark that they apply to a range of relevant cases (see Example \ref{remark/ on assumption N2} and \ref{example/applicability of coefficient assumptions}).
We first take care of the randomness in the equation.
We fix a probability space $\left(\Omega,\left(\F_t\right)_{t\geq0},\p\right)$ with a complete right-continuous filtration $\F_t$ and supporting a countable sequence of $\R^d$-valued independent Brownian motions $\left(B_t^k\right)_{k\in\N}$ and the random initial condition $\rho_0$.

As regards the initial condition we require the following.
\begin{assumption}[Assumptions on the initial data]\label{assumption/assumption I1}
The initial data $\rho_0\in L^1(\Omega;L^1(\T^d))$ is nonnegative, $\F_0$-measurable and satisfies one of the following hypotheses, for $p\geq2$ and $m\geq1$ given in Assumption \ref{assumption/assumption C1} below:
\begin{itemize}
    \item [(i)] $\rho_0\in L^p(\Omega;L^p(\T^d))\cap L^{p+m-1}(\Omega;L^1(\T^d))$;
    \item[(ii)] $\rho_0\in L^{p+m-1}(\Omega;(0,\infty))$ with $\rho\geq r>0$ a.e. for some $r\in(0,\infty)$, i.e. the initial data is a random positive constant.
\end{itemize}
\end{assumption}

We now define the noise sequence $\xi^\epsilon$.
For each $\epsilon>0$, let $F^\epsilon=\left(f_k^{\epsilon}\right)_{k\in\N}$ be a family of continuously differentiable functions on $\T^d$ and define the noise
\begin{equation}
    \label{equation/noise xi^epsilon}
    \xi^\epsilon=\sum_{k\in\N}f_{k}^\epsilon B_t^k.
\end{equation}
It then follows that the Stratonovich equation \eqref{equation/generalized dean--kawasaki equation 1} is formally equivalent to the It\^o equation
\begin{align}\label{equation/generalized dean--kawasaki equation 2}
\begin{split}
    \de\rho^{\epsilon}=\Delta\phi(\rho^\epsilon)\de t-\nabla\cdot\nu(\rho^\epsilon)-\sqrt{\epsilon}\,\nabla\!\cdot\!\left(\sigma(\rho^\epsilon) \,\de\xi^{\epsilon}\right)
    +\frac{\epsilon}{2}\sum_{k=1}^\infty\nabla\cdot\Big(f_k^\epsilon\Dot{\sigma}(\rho^\epsilon)\nabla\big(f_k^\epsilon\sigma(\rho^\epsilon)\big)\Big)\de t,
\end{split}
\end{align}
which can be written in the form
\begin{align}\label{equation/generalized dean--kawasaki equation 3}
\begin{split}
\de\rho^{\epsilon}=
\Delta\phi(\rho^\epsilon)\de t-\nabla\cdot\nu(\rho^\epsilon)-\sqrt{\epsilon}\,\nabla\!\cdot\!\left(\sigma(\rho^\epsilon) \,\de\xi^{\epsilon}\right)
+\frac{\epsilon}{2}\nabla\!\cdot\!\left(F_1^\epsilon \left(\Dot{\sigma}(\rho^\epsilon)\right)^2\nabla\rho^\epsilon+\dot{\sigma}(\rho^\epsilon)\sigma(\rho^\epsilon)F_2^\epsilon\right)\de t,
\end{split}
\end{align}
for $F_1^\epsilon:\T^d\to\R$ and $F_2^\epsilon:\T^d\to\R^d$ defined by
\begin{equation}
F_1^\epsilon(x)=\sum_{k=1}^\infty (f_k^\epsilon)^2(x)\;\;\textrm{and}\;\;F_2^\epsilon(x)=\sum_{k=1}^\infty f_k^\epsilon(x)\nabla f_k^\epsilon(x).
\end{equation}
We make the following assumptions on the noise.

\begin{assumption}[Assumptions on the noise]\label{assumption/assumption N1}
For each $\epsilon>0$, we assume that the sums $\{F_i^\epsilon\}_{i=1,2,3}$ defined by
\begin{equation}\label{equation/definition of f1 f2 f3}
F_1^\epsilon=\sum_{k=1}^\infty (f_k^\epsilon)^2\;\;\textrm{and}\;\;F_2^\epsilon=\frac{1}{2}\sum_{k=1}^\infty \nabla (f^\epsilon_k)^2\;\;\textrm{and}\;\;F_3^{\epsilon}=\sum_{k=1}^\infty\left|\nabla f_k^\epsilon\right|^2
\end{equation}
are continuous on $\T^d$ -- where the finiteness of $F_1^\epsilon$ and $F_3^\epsilon$ implies the absolute convergence of $F_2^\epsilon$ -- and assume that the divergence of $F_2^\epsilon$ vanishes:
\begin{equation}\label{equation/div f2 vanishes}
\nabla\cdot F_2^\epsilon=\frac{1}{2}\Delta F_1^\epsilon=0.
\end{equation}
\end{assumption}

\begin{remark}
Condition \eqref{equation/div f2 vanishes} is equivalent to the noise being probabilistically stationary in the sense that it has the same law at every point in space, a property satisfied by space-time white noise and all of its standard approximations, such as those presented in Example \ref{remark/ on assumption N2} below.
\end{remark}

As mentioned above, in our analysis of the small noise behaviour of \eqref{equation/generalized dean--kawasaki equation 1} we let the noise terms $\xi^\epsilon$ converge to some noise $\xi$ as $\epsilon\to0$. 
Precisely, we make the following assumption.

\begin{assumption}[Assumptions on the noise sequence]\label{assumption/assumption N2}
For each $\epsilon>0$ the family $F^{\epsilon}$ satisfies Assumption \ref{assumption/assumption N1} and there exists a family of functions $F=F^0=(f_m)_{m\in\N}\subseteq L^2(\T^d)$ such that, for some $C>0$,
\begin{equation}
    \label{equation/ banach steinhaus on the noise}
    \sum_{m=1}^{\infty}(f_m^\epsilon,u)_{L^2(\T^d)}^2\leq C\,\|u\|_{L^2(\T^d)}^2\quad\forall u\in L^2(\T^d)\quad \forall \epsilon\geq0,
\end{equation}
and such that
\begin{equation}\label{equation/ convergence of the noise sequence on C1 functions}
    \lim_{\epsilon\to0}\sum_{m=1}^{\infty}(f_m-f_m^\epsilon,u)_{L^2(\T^d)}^2=0\quad\forall u\in L^2(\T^d).
\end{equation}
\end{assumption}

\begin{example}[Applicability of the noise assumptions]\label{remark/ on assumption N2}
Assumption \ref{assumption/assumption N2} serves to deal with cylindrical Wiener processes for which the space $L^2(\T^d)$ is too small to live in (see e.g. \cite[Chapter~4]{da_prato_zabczyk_1992}) and for which the Dean--Kawasaki equation \eqref{equation/generalized dean--kawasaki equation 1} might be ill-posed.
For example, this is the case for a space-time white noise $\xi$ on $L^2(\T^d)$, which admits the spectral representation
\begin{equation}
    \label{equation/space time white noise series representation}
    \xi=\sum_{m=1}^\infty f_m(x) B^m_t,
\end{equation}
for any orthonormal basis $(f_m)_{m\in\N}$ of $L^2(\T^d)$.
In this setting, important examples of the approximating sequence of noise terms include spatial convolutions $\xi^\epsilon=\xi*\varphi^\epsilon$, that is $f_m^\epsilon=f_m*\varphi^\epsilon$ for a suitable mollifier $\varphi$; 
ultraviolet cut-offs like 
\begin{equation}\label{equation/ ultraviolet cutoff}
    \xi^\epsilon=\sum_{\substack{k\in\Z^d,\, |k|\leq M_{\epsilon}}}e^{i2\pi k\cdot x}\, B^k_t,
\end{equation}
for any arbitrary sequence $(M_{\epsilon})_{\epsilon>0}$ increasing to infinity as $\epsilon\to0$, where for simplicity we used the orthonormal basis of complex exponentials, and we have $f_k^\epsilon=e^{i2\pi k\cdot x}$ if $|k|\leq M_{\epsilon}$ and $f_k^{\epsilon}=0$ otherwise;
and weighted expansions like
\begin{equation}
\xi^{a^\epsilon} = \sum_{k\in\Z^d} a_k^\epsilon\, e^{i2\pi k\cdot x}\, B_t^k,
\end{equation}
for coefficients $a=(a_k^\epsilon)_{k\in \Z^d}$ satisfying $\sum_{k\in\Z^d}|k|^2|a_k^\epsilon|^2<\infty$ and $\lim_{\epsilon\to0}a_k^\epsilon=1$.
Indeed, an explicit computation gives, for the mollification,
\begin{equation}
    F_1^\epsilon=\frac{1}{\epsilon^d}\int_{\T^d}\!|\varphi(y)|^2\,dy,\,\,F_2^\epsilon=0,\,\, F_3^\epsilon=\frac{1}{\epsilon^{d+2}}\int_{\T^d}\left|\nabla\varphi(y)\right|^2\,dy\,\,\textrm{and}\,\sum_{|m|=0}^{\infty}\!(f_m-f_m^\epsilon,u)_{L^2}^2=\|u-u*\varphi^\epsilon\|_{L^2};
\end{equation}
for the ultraviolet cut-off,
{\small
\begin{equation}
    F_1^\epsilon(x)\!=\!\!\sum_{|k|=0}^{M_\epsilon}\!f_k(x)^2,\,\, F_2^\epsilon(x)\!=\!\!\frac{1}{2}\!\!\sum_{|k|=0}^{M_\epsilon}\!\!\nabla \left(f_k\right)^2,\,\, F_3^\epsilon(x)\!=\!\!\sum_{|k|=0}^{M_\epsilon}\!\!|\nabla f_k(x)|^2\,\,\textrm{and}\,\sum_{|k|=0}^{\infty}\!\!(f_k-f_k^\epsilon,u)_{L^2}^2\!=\!\!\sum_{|k|>M_\epsilon}^{\infty}\!(f_k,u)_{L^2}^2;
\end{equation}
}
and for the weighted expansion,
{\small
\begin{equation}   F_1^{a^\epsilon}=\sum_{m\in\Z^d}|a_m^\epsilon|^2,\,\,\,F_2^{a^\epsilon}=0,\quad F_3^{a^\epsilon}=\sum_{m\in\Z^d}|m|^2|a_m^\epsilon|^2\,\,\textrm{and}\,\,\sum_{m\in\Z^d}(f_m-f_m^\epsilon,u)_{L^2}^2=\sum_{m\in\Z^d}|a_m^{\epsilon}-1|^2(f_m,u)_{L^2}^2.
\end{equation}
}
\end{example}

Finally we collect our assumptions on the coefficients $\phi$, $\nu$, $\sigma$, and discuss their meaning and applicability.
The following set of assumptions coincides exactly with \cite[Assumption~4.1 and~5.2]{fehrman-gess-Well-posedness-of-the-Dean-Kawasaki-and-the-nonlinear-Dawson-Watanabe-equation-with-correlated-noise}.
Together with Assumption \ref{assumption/assumption I1} and \ref{assumption/assumption N1} on the initial data and the noise, it guarantees the well-posedness of equation \eqref{equation/generalized dean--kawasaki equation 2} in the sense of stochastic kinetic solutions (see Theorem \ref{theorem/existence and uniquess of kinetic solutions to dean-kawasaki}).

\begin{assumption}[Assumptions on the coefficients for the well-posedness of the equation]\label{assumption/assumption C1}
Let $\phi,\sigma\in C([0,\infty))\cap C^{1,1}_{\textrm{loc}}((0,\infty))$ and $\nu\in C([0,\infty))^d\cap C^{1,1}_{\textrm{loc}}((0,\infty))^d$, let $p\in[2,\infty)$ and $m\in[1,\infty)$, and for every $q\geq 2$ let $\Theta_{\phi,q}\in C([0,\infty))\cap C^1 ((0,\infty))$ be the unique function satisfying
\begin{equation}
    \label{equation/thetha phi p}
    \Theta_{\phi,q}(0)=0\quad\text{and}\quad\dot{\Theta}_{\phi,q}(z)=z^{\frac{q-2}{2}}\left(\dot{\phi}(z)\right)^{\frac{1}{2}}.
\end{equation}
Assume that the following eight conditions are satisfied.
\begin{itemize}
    \item [(i)] We have $\phi(0)=\sigma(0)=0$ and $\dot{\phi}>0$ on $(0,\infty)$.
    \item [(ii)] There exists $c\in(0,\infty)$ such that
        \[\phi(z)\leq c\,(1+z^m)\;\;\text{for every}\;\;z\in[0,\infty).\]
    \item[(iii)] There exists $c\in(0,\infty)$ such that 
        \[\limsup_{z\to0^+}\frac{\sigma^2(z)}{z}\leq c,\]
    which in particular implies that $\sigma(0)=0$.
    \item[(iv)] There exists $c\in(0,\infty)$ such that
        \begin{equation}\label{equation/assumption c1 v} \sup_{z'\in[0,z]}\sigma^2(z')\leq c\,(1+z+\sigma^2(z))\;\;\text{for every}\;\;z\in[0,\infty).
        \end{equation}
    \item[(v)] There exists $c\in(0,\infty)$ such that
        \begin{equation}\label{equation/assumption c1 vi}
        \textcolor{black}{
        \sup_{z'\in[0,z]}\nu^2(z')\leq c\,(1+z+\nu^2(z))\;\;\text{for every}\;\;z\in[0,\infty).}
        \end{equation}
    \item[(vi)] For $\Theta_{\phi,p}$ defined in \eqref{equation/thetha phi p}, either there exists $c\in(0,\infty)$ and $\gamma\in[0,\nicefrac{1}{2}]$ such that
        \begin{equation}\label{equation/assumption c1 vii A}
        \left(\dot{\Theta}_{\phi,p}(z)\right)^{-1}\leq c\, z^\gamma\;\;\text{for every}\;\;z\in(0,\infty),
        \end{equation}
    or there exists $c\in(0,\infty)$ and $q\in[1,\infty)$ such that, for every $z,z'\in[0,\infty)$,
        \begin{equation}\label{equation/assumption c1 vii B}
        \left|\,z-z'\,\right|^{\,q}\leq c\left|\Theta_{\phi,p}(z)-\Theta_{\phi,p}(z')\right|^2.
        \end{equation}
    \item[(vii)] For $\Theta_{\phi,2}$ and $\Theta_{\phi,p}$ defined in \eqref{equation/thetha phi p}, there exists $c\in(0,\infty)$ such that, for every $z\in[0,\infty)$,
        \begin{equation}\label{aa_5050}
        \sigma^2(z)\leq c\,\left(1+z+\Theta^2_{\phi,2}(z)\right)\;\;\text{and}\;\;z^{p-2}\sigma^2(z)\leq c\left(1+z+\Theta^2_{\phi,p}(z)\right).
        \end{equation}
    \item[(viii)] For every $\delta\in(0,1)$ there exists $c_\delta\in(0,\infty)$ such that, for every $z\in(\delta,\infty)$,
        \begin{equation}\label{aa_121212}
        \frac{[\dot{\sigma}(z)]^4}{\dot{\phi}(z)}+(\sigma(z)\dot{\sigma}(z))^2+\textcolor{black}{|\nu(z)|}+\dot{\phi}(z)\leq c_\delta\,\left(1+z+\Theta^2_{\phi,p}(z)\right).
        \end{equation}
\end{itemize}
\end{assumption}

\begin{remark}
With regards to Assumption \ref{assumption/assumption C1}, in Condition (i) the assumption $\phi(0)=0$ is just a normalization, whereas the assumption $\sigma(0)=0$ is crucial to prevent the noise from dragging solution $\rho^\epsilon$ towards negative values.
Condition (ii) amounts to polynomial growth of $\phi$ at infinity.
Condition (iii) requires that $\sigma(z)$ goes to zero as $z\to0$ at least as fast as $\sqrt{z}$.
Condition (iv) amounts to an assumption on the magnitude of the oscillations of $\sigma$ at infinity, regardless of their frequency (see Example \ref{example/applicability of coefficient assumptions} below).
For example, it is satisfied if $\sigma^2$ is monotone or if $\sigma^2$ grows linearly at infinity or if the oscillations of $\sigma$ grow linearly at infinity. 
Identical considerations holds for the analogous condition (v) on $\nu$.
Condition (vi), both in the form \eqref{equation/assumption c1 vii A} or \eqref{equation/assumption c1 vii B}, corresponds to a regularity assumption on $\phi$: specifically, H\"older continuity of the inverse of the resulting function $\Theta_{\phi,p}$.
Condition (vii) amounts to a growth condition at infinity on $\sigma$, with the aim that one of the convective terms, or better one of the resulting It\^o correction convective terms, is somehow dominated by the diffusion, namely by $\Theta_{\phi,p}$.
Similarly, Condition (viii) is a hypothesis on the growth away from zero of the convective terms, either the deterministic term $\nu$ or other It\^o correction terms involving $\sigma$.
\end{remark}

Assumption \ref{assumption/assumption C1} on the coefficients, together with Assumption \ref{assumption/assumption I1} and \ref{assumption/assumption N1} on the initial data and the noise, is sufficient to guarantee existence and uniqueness for equation \eqref{equation/generalized dean--kawasaki equation 3}.
To establish the CLT in probability \eqref{equation/clt in probability 1} we just need some more control on the coefficients at infinity, regardless of their behaviour near the irregularity at zero.

\begin{assumption}[Assumptions on the coefficients for the CLT in probability]\label{assumption/assumption C2 weak}
We assume that $\phi$, $\nu$, $\sigma$ satisfy the following, for $p\in[2,\infty)$ given in Assumption \ref{assumption/assumption C1}.

\begin{itemize}
    \item[(i)] For some $k\in[0,0\vee\frac{p-4}{4}]$, for every $\delta>0$ there exists $c_{\delta}\in(0,\infty)$ such that
    \begin{equation}
        |\sigma(z)|\leq c_{\delta}\,(1+z^{k+1})\quad\text{and}\quad
        |\dot{\sigma}(z)|\leq c_{\delta}\,(1+z^k) \quad \forall\,\,z\in(\delta,\infty).
    \end{equation}
    \item[(ii)] We have $\phi,\nu\in C^2_{\textrm{loc}}((0,\infty))$ and there exists $g\in[0,0\vee \frac{p}{2(k+1)}-2]$ such that, for every $\delta>0$ there exists $c_{\delta}\in(0,\infty)$ such that
    \begin{equation}
        |\ddot{\phi}(z)|+|\ddot{\nu}(z)|\leq c_{\delta}\,(1+z^g) \quad \forall\,\,z\in(\delta,\infty).
    \end{equation}
\end{itemize}
\end{assumption}

To establish the stronger version \eqref{equation/rate of convergence 1} of the CLT, we will replace Assumption \ref{assumption/assumption C2 weak} with the following stronger version.
It is essentially the requirement that Assumption \ref{assumption/assumption C2 weak} holds uniformly on $(0,\infty)$ up to the irregularity at zero.
We stress that \emph{this assumption is NOT needed} for the general CLT in probability \eqref{equation/clt in probability 1}.

\begin{assumption}[Assumptions on the coefficients for the CLT in $L^2(\Omega)$]\label{assumption/assumption C2}
We assume that $\phi$, $\nu$, $\sigma$ satisfy the following, for some constant $c\geq0$, for $p\in[2,\infty)$ given in Assumption \ref{assumption/assumption C1}.

\begin{itemize}
    \item[(i)] For some $\theta\in(0,\frac{1}{2})$ and some $k\in[0,0\vee\frac{p-4}{4}]$, for all $z\in(0,\infty)$,
    \begin{equation}
        |\sigma(z)|\leq c\,(1+z^{k+1}),\quad
        |\dot{\sigma}(z)|\leq c\,(1+z^{-\theta}+z^k)
        \quad\text{and}\quad
        |\sigma(z)\dot{\sigma}(z)|\leq c\,(1+z^{2k+1}).
    \end{equation}
    \item[(ii)] We have $\phi,\nu\in C^2((0,\infty))$ and there exists $g\in[0,0\vee\frac{p}{2(k+1)}-2]$ such that, for all $z\in(0,\infty)$,
    \begin{equation}
        |\ddot{\phi}(z)|+|\ddot{\nu}(z)|\leq c\,(1+z^g).
    \end{equation}
\end{itemize}
\end{assumption}

\begin{remark}
Assumption \ref{assumption/assumption C2 weak}(i) or \ref{assumption/assumption C2}(i) serves to control the convective term $\nabla\!\cdot\!(\sigma(\rho^\epsilon)\circ\xi^\epsilon)$, so that it is dominated by the diffusion, and indeed it might be replaced by the somewhat more explicit conditions $\left|\sigma(z)\right|\leq c\,(1+\phi^{\nicefrac{1}{2}}(z))$ and $\left|\dot{\sigma}(z)\right|\leq c\,(1+(\phi^{\nicefrac{1}{2}})^\prime(z))$ for all $z\in(0,\infty)$.
Assumption \ref{assumption/assumption C2 weak}(ii) or \ref{assumption/assumption C2}(ii) corresponds to polynomial growth at infinity, and also regularity near zero for \ref{assumption/assumption C2}(ii), of the diffusion nonlinearity $\phi$ and of the deterministic convective term $\nu$.
It is needed to recast the diffusive term $\Delta\big(\phi(\rho^\epsilon)-\phi(\Bar{\rho})\big)$, or the convective term respectively, of the nonequilibrium fluctuation $v^\epsilon$ in \eqref{equation/fluctuation equation 1} in terms of the fluctuation $v^\epsilon$ itself.
\end{remark}

\begin{example}[Applicability of the coefficient assumptions] \label{example/applicability of coefficient assumptions}
In the porous media case $\phi(z)=z^{m_0}$ all the conditions involving $\phi$ in Assumption \ref{assumption/assumption C1} and \ref{assumption/assumption C2 weak} are verified in the full regime $m_0\in(0,\infty)$.
In this case, the function $\Theta_{\phi,p}$ defined in \eqref{equation/thetha phi p} is given for a constant $c_{p,m_0}\in(0,\infty)$ by
\[\Theta_{\phi,p}(z)=c_{p,m_0}\,z^{\frac{m_0+p-1}{2}}.\]
The conditions are satisfied by taking $m=1$, $p=2$, $g=k=0$ and choosing option \eqref{equation/assumption c1 vii A} with $\gamma=\frac{1-m_0}{2}$ when $m_0\in(0,1)$, and by taking $m=m_0$, any $p\geq 4\vee m_0^2$, $k=0\vee\frac{m_0-2}{2}$, $g=0\vee m_0-2$ and choosing option \eqref{equation/assumption c1 vii B} with $q=m_0+p-1$ when $m_0\in[1,\infty)$.

In the important model case $\phi(z)=z^{m_0}$, $\sigma(z)=\phi^{\nicefrac{1}{2}}(z)=z^{\frac{m_0}{2}}$, with $\nu=0$ for simplicity, all the conditions in Assumption \ref{assumption/assumption C1} and \ref{assumption/assumption C2 weak} are verified for any $m_0\in[1,\infty)$ with the choices of $m$, $p$, $k$, $g$ given above.

In fact, Assumption \ref{assumption/assumption C2 weak} and all the conditions in Assumption \ref{assumption/assumption C1} except for condition (iii) are verified also for $m_0\in(0,1)$, with the corresponding choices of $m$, $p$, $k$, $g$, $\gamma$.
Only condition (iii) in Assumption \ref{assumption/assumption C1} breaks down.
Indeed, already to prove well-posedness of \eqref{equation/generalized dean--kawasaki equation 1} we need the noise coefficient $\sigma$ to be $\frac{1}{2}$-H\"older near the irregularity point zero.

We also comment on Assumption \ref{assumption/assumption C1}(iv) on the oscillations of $\sigma$.
An interesting example satisfying all the conditions in Assumption \ref{assumption/assumption C1} and \ref{assumption/assumption C2 weak}, is given by $\sigma^2(z)=z^{m_0}+z \sin\big(z^h\big)$ for every $m_0,h\in[1,\infty)$.
That is, condition (iv) imposes a condition on the growth of the magnitude of the oscillations of $\sigma^2$ at infinity, but not on the growth of the frequency of the oscillations.

Finally the stronger Assumption \ref{assumption/assumption C2} for the CLT in $L^2(\Omega)$, together with Assumption \ref{assumption/assumption C1}, applies to the model case $\phi(z)=z^{m_0}$, $\sigma(z)=z^{\frac{m_0}{2}}$ in the regime $m_0\in[2,\infty)$, or to the linear case $\phi(z)=z$ provided we consider a different noise coefficient $\sigma$ satisfying the assumptions, for example $\sigma(z)=z^{\frac{1}{s}}$ for some $s\in(0,2)$.
\end{example}

\subsection{The zero noise limit and the generalized Ornstein--Uhlenbeck process}
\label{section/the hydrodynamic limit and the generalized ornstein–-uhlenbeck process}

In this section we collect some well-established results on the well-posedness of the limiting equations \eqref{equation/hydrodynamic limit equation 1} and \eqref{equation/OU equation 1}.
As regards the deterministic limit \eqref{equation/hydrodynamic limit equation 1} we have the following classical theory.
We refer the reader to standard monographs like \cite{vasquez_porous_media} or \cite{Lieberman_parabolic}.

\begin{definition}
\label{definition/hydrodynamic limit weak solution}
A weak solution of \eqref{equation/hydrodynamic limit equation 1} is a function $\bar{\rho}\in L^1([0,T];L^1(\T^d))$ such that $\phi(\bar{\rho}),\nu(\bar{\rho})\in L^1([0,T];L^1(\T^d))$ and such that, for any $t\in[0,T]$ and any $\psi\in C^{\infty}(\T^d)$, 
\begin{align}
\begin{split}
    \int_{\T^d}\!\!\!\bar{\rho}(x,t)\,\psi(x)\,dx=\int_{\T^d}\!\!\!\rho_0(x)\,\psi(x)\,dx
    +\int_0^t\!\!\int_{\T^d}\!\!\!\phi(\bar{\rho})(x,s)\,\Delta\psi(x)\,dx\,ds
    +\textcolor{black}{\int_0^t\!\!\int_{\T^d}\!\!\!\nu(\bar{\rho})(x,s)\,\nabla\psi(x)\,dx\,ds.}
\end{split}
\end{align}
\end{definition}

\begin{theorem}\label{theorem/well-posedness of the hydrodynamic limit}
Under Assumption \ref{assumption/assumption C1}, for any nonnegative $\rho_0\in L^2(\T^d)$ bounded away from zero, there exists a unique weak solution $\Bar{\rho}$ of \eqref{equation/hydrodynamic limit equation 1} in the sense of Definition \ref{definition/hydrodynamic limit weak solution}.
Furthermore the solution satisfies $\bar{\rho}\in L^2([0,T];L^2(\T^d))$ and $\phi(\bar{\rho})\in L^2([0,T];H^1(\T^d))$.
\end{theorem}

\begin{remark}
Except for the initial data $\rho_0$, which can be random, equation \eqref{equation/hydrodynamic limit equation 1} is purely deterministic.
Furthermore, one immediately see that for $\rho_0\in(0,\infty)$ a (possibly random) constant, the solution is simply the constant $\Bar{\rho}(x,t,\omega)\equiv\rho_0(\omega)$.
\end{remark}

We now consider the stochastic equation
\begin{equation}\label{equation/OU equation 2}
    \partial_tv=\Delta\left(a\, v\right)-\nabla\cdot\left(b\,v+c\,\dot{\xi}\right)\,\,\,\text{in } \mathbb{T}^d\times(0,T),\quad v(\cdot,0)=0\,\,\,\text{in } \mathbb{T}^d\times\{0\},
\end{equation}
for some $\F_0$-measurable random variables $a,b\in L^0(\Omega;\R)$ with $a\geq r>0$ a.e. for some $r\in(0,\infty)$, and $c\in L^2(\Omega;\R)$, and for the noise $\xi=\lim_{\epsilon\to0}\xi^\epsilon$ specified in Assumption \ref{assumption/assumption N2}.
First of all, we make precise our notion of solution, here $\alpha\geq0$ depends on the regularity of the noise $\xi$, which satisfies Assumption \ref{assumption/assumption N2}.

\begin{definition}
\label{definition/ OU equation weak solution}
A weak solution in $H^{-\alpha}(\T^d)$ of equation \eqref{equation/OU equation 2} is a predictable process $v\in L^2(\Omega\times(0,T);H^{-\alpha}(\T^d))$ such that, for every $t\in[0,T]$ and every $w\in H^{\alpha+2}(\T^d)$, $\p$-almost surely we have
\begin{equation}
    \langle v(s),w(s)\rangle\,\Big|_{s=0}^{s=t}=\int_0^t\langle v(r),a\,\Delta w(r)+ b\nabla w(r)\rangle\, dr +  \int_0^t\int_{\T^d}c\,\nabla w(x,r)\, \de\xi.
\end{equation}
\end{definition}

We have the following well-posedness result.
This can be proved for example via semigroup methods (see e.g. \cite[Chapter~5]{da_prato_zabczyk_1992}) or Fourier analysis (see e.g. \cite{dirr-fehrman-gess-conservative-stochastic-pde-and-fluctuations-of-the-symmetric}).

\begin{theorem}
\label{theorem/ well-posedness of OU equation}
Let $\xi$ satisfies Assumption \ref{assumption/assumption N2}.
There exists a unique weak solution $v\in L^2(\Omega\times(0,T);H^{-\alpha}(\T^d))$ of equation \eqref{equation/OU equation 2} for any $\alpha>\frac{d}{2}$. 
Furthermore, if $a_n\to a$ and $b_n\to b$ almost surely, with $a_n\geq r>0$ for every $n\in\N$, and $c_n\to c$ in $L^2(\Omega;\R)$, as $n\to\infty$, and we denote by $v_n$ and $v$ the corresponding solutions of \eqref{equation/OU equation 2}, then we have $v_n\to v$ in $L^2(\Omega\times(0,T);H^{-\alpha}(\T^d))$ for any $\alpha>\frac{d}{2}$.
\end{theorem}


\subsection{The generalized Dean--Kawasaki equation}
\label{section/the generalized dean--kawasaki equation}

In this section we summarize the solution theory to equation \eqref{equation/generalized dean--kawasaki equation 1}, interpreted in its It\^o formulation \eqref{equation/generalized dean--kawasaki equation 3}, as put forward in \cite{fehrman-gess-Well-posedness-of-the-Dean-Kawasaki-and-the-nonlinear-Dawson-Watanabe-equation-with-correlated-noise}.
In particular, we introduce the notions of weak solution and of stochastic kinetic solution and we collect some of their properties.

The notion of weak solution is classical.

\begin{definition}
\label{definition/weak solution to generalized dean kawasaki}
Let $\rho_0\in L^1(\Omega;L^1(\T^d))$ be nonnegative and $\F_0$-measurable.
A weak solution to \eqref{equation/generalized dean--kawasaki equation 3} with initial data $\rho_0$ is an $\F_t$-adapted nonnegative continuous $L^1(\T^d)$-valued process $\rho\in L^1(\Omega\times[0,T];L^1(\T^d))$ such that almost surely
\begin{equation}
    \rho,\phi^{\nicefrac{1}{2}}(\rho),\sigma(\rho)\in L^2([0,T];H^1(\T^d)),\quad \dot{\sigma}(\rho)\in L^2([0,T];L^2(\T^d))\quad\text{and}\quad \nu(\rho)\in L^1([0,T];L^1(\T^d)),
\end{equation}
and such that almost surely, for every $\psi\in C^\infty(\T^d)$ and every $t\in[0,T]$,
\begin{align} \label{equation/definition of weak solution to dean-kawasaki}
\begin{split}
    \int_{\T^d}\rho(x,t)\psi(x)\,dx-\int_{\T^d}\rho_0(x)\,\psi(x)\,dx=&
    \int_0^t\int_{\T^d}\Delta\psi(x)\,\phi(\rho)(x,s)\,dx\,ds
    \\
    &+
    \int_0^t\int_{\T^d}\nabla\psi(x)\,\nu(\rho)(x,s)\,dx\,ds
    \\
    &+\sqrt{\epsilon}\int_0^t\int_{\T^d}\nabla\psi(x)\,\sigma(\rho)(x,s)\,d\xi^\epsilon
    \\
    &-\frac{\epsilon}{2}\!\int_0^t\!\int_{\T^d}\!\!\nabla\psi(x)\!\left(F_1^\epsilon \Dot{\sigma}(\rho)\nabla\sigma(\rho)+\dot{\sigma}(\rho)\sigma(\rho)F_2^\epsilon\right)\,dx\,ds.
\end{split}
\end{align}
\end{definition}

When the coefficients of the equation are sufficiently smooth, weak solutions exist and are unique.

\begin{theorem}[Theorem 5.20 in \cite{fehrman-gess-Well-posedness-of-the-Dean-Kawasaki-and-the-nonlinear-Dawson-Watanabe-equation-with-correlated-noise}]
\label{theorem/existence and uniqueness of weak solution to dean_kawasaki}
Let $\rho_0$ satisfy Assumption \ref{assumption/assumption I1}(i), let $\xi^\epsilon$ satisfy Assumption \ref{assumption/assumption N1} and let $\phi,\nu$ and $\sigma$ satisfy Assumption \ref{assumption/assumption C1}.
Suppose furthermore that $\phi,\sigma$ satisfy the additional assumptions:
\begin{itemize}
    \item[(i)] $\dot{\phi}(z)\geq c>0$ for all $z\in(0,\infty)$, for some $c\in(0,\infty)$;
    \item[(ii)] $\sigma\in C([0,\infty))\cap C^{\infty}((0,\infty))$ with $\dot{\sigma}\in C_c^{\infty}([0,\infty))$.
\end{itemize}
There exists a unique weak solution to equation \eqref{equation/generalized dean--kawasaki equation 3} in the sense of Definition \ref{definition/weak solution to generalized dean kawasaki}.
\end{theorem}

Despite being quite natural, the notion of weak solution is often doomed to fail when $\sigma$ is not smooth enough.
Indeed, as discussed in Section \ref{section/comments on the literature}, already in the prototypical Dean--Kawasaki case $\sigma(z)=\sqrt{z}$, it is not clear how to interpret some of the terms featuring in \eqref{equation/definition of weak solution to dean-kawasaki}.
To handle rougher coefficients a wider solution theory is needed.

Namely, given a nonnegative function $\rho:\T^d\times[0,T]\to\R$, one introduces an additional variable $z\in\R$ and consider the \emph{kinetic function} $\chi:\R\times\T^d\times[0,T]\to\{0,1\}$ of $\rho$ defined by
\begin{equation}\eqref{equation/derivation of the kinetic formulation 4}
    \chi(z,x,t)=\Bar{\chi}(z,\rho(x,t))=\mathbf{1}_{\{0<z<\rho(x,t)\}},
\end{equation}
for $\Bar{\chi}(z,v)\coloneqq\mathbf{1}_{\{0<z<v\}}-\mathbf{1}_{\{v<z<0\}}$.
The \emph{velocity} variable $z$ essentially corresponds to the magnitude of the solution $\rho$.
A formal computation presented in \cite{fehrman-gess-Well-posedness-of-the-Dean-Kawasaki-and-the-nonlinear-Dawson-Watanabe-equation-with-correlated-noise}, based on the fundamental identity
\begin{equation}\eqref{equation/derivation of the kinetic formulation 5}
    S(\rho(x,t))=\int_\R\dot{S}(z)\chi(z,x,t)\,dz,
\end{equation}
for every smooth function $S:\R\to\R$ with $S(0)=0$, motivates the notion of stochastic kinetic solution.

\begin{definition}\label{definition/stochastic kinetic solution}  Let $\rho_0\in L^1(\Omega;L^1(\Omega))$ be nonnegative and $\F_0$-measurable.  A \emph{stochastic kinetic solution} of \eqref{equation/generalized dean--kawasaki equation 3} is a nonnegative, a.s. continuous $L^1(\T^d)$-valued $\F_t$-predictable function $\rho\in L^1(\Omega\times[0,T];L^1(\T^d))$ that satisfies the following three properties.
\begin{enumerate}
\item[(i)]\emph{Preservation of mass}:  a.s.\ for every $t\in[0,T]$,
\begin{equation}\label{def_3535}\|\rho(\cdot,t)\|_{L^1(\T^d)}=\|\rho_0\|_{L^1(\T^d)}.\end{equation}
\item [(ii)]\emph{Integrability of the flux}:  we have that
\[\sigma(\rho)\in L^2(\Omega;L^2(\T^d\times[0,T]))\;\;\textcolor{black}{\textrm{and}\;\;\nu(\rho)\in L^1(\Omega;L^1(\T^d\times[0,T];\R^d)).}\]
\item [(iii)]\emph{Local regularity}:  for every $K\in\N$,
\begin{equation}\label{def_2500000} 
[(\rho\wedge K)\vee\nicefrac{1}{K}]\in L^2(\Omega;L^2([0,T];H^1(\T^d))).
\end{equation}
\end{enumerate}
Furthermore, there exists a \emph{kinetic measure} $q$, that is a map $q$ from $\Omega$ to the space of nonnegative, locally finite measures on $\T^d\times(0,\infty)\times[0,T]$ such that, for every $\psi\in\C^\infty(\T^d\times(0,\infty))$, the process
\[\Omega\times[0,T]\ni(\omega,t)\mapsto \int_0^t\int_\R\int_{\T^d}\psi(x,\xi)\,\de q(\omega)\;\;\textrm{is $\F_t$-predictable,}\]
that satisfies the following three properties.
\begin{enumerate}
\item[(iv)] \emph{Regularity}: almost surely as nonnegative measures,
\begin{equation}\label{2_500}
\delta_0(z-\rho)\dot{\phi}(z)|\nabla\rho|^2\leq q\;\;\textrm{on}\;\;\T^d\times(0,\infty)\times[0,T].
\end{equation}
\item[(v)] \emph{Vanishing at infinity}:  we have that
\begin{equation}\label{equation/kinetic measure vanishes at infinity}
\lim_{M\rightarrow\infty}\E\left[q(\T^d\times[M,M+1]\times[0,T])\right]=0.
\end{equation}
\item [(vi)]\emph{The equation}: almost surely, for every $\psi\in \C^\infty_c(\T^d\times(0,\infty))$ and every $t\in[0,T]$,
%
\begin{align}\label{equation/kinetic formulation of dean-kawasaki}
& \int_\R\int_{\T^d}\chi(x,z,t)\psi(x,z) = \int_\R\int_{\T^d}\overline{\chi}(\rho_0)\psi(x,z)-\int_0^t\int_{\T^d}\dot{\phi}(\rho)\nabla\rho\cdot(\nabla\psi)(x,\rho)
\\ \nonumber &  -\frac{\epsilon}{2}\int_0^t\int_{\T^d}F_1^\epsilon[\dot{\sigma}(\rho)]^2\nabla\rho\cdot(\nabla\psi)(x,\rho)-\frac{\epsilon}{2}\int_0^t\int_{\T^d}\sigma(\rho)\dot{\sigma}(\rho)F_2^\epsilon\cdot(\nabla\psi)(x,\rho)
\\ \nonumber & -\int_0^t\int_\R\int_{\T^d}\partial_z\psi(x,z)\,\de q+\frac{\epsilon}{2}\int_0^t\int_{\T^d}\left(\sigma(\rho)\dot{\sigma}(\rho)\nabla\rho\cdot F_2^\epsilon+F_3^\epsilon\,\sigma^2(\rho)\right)(\partial_z\psi)(x,\rho)
\\ \nonumber &  -\int_0^t\int_{\T^d}\psi(x,\rho)\nabla\cdot\nu(\rho)\de t
- \sqrt{\epsilon}\int_0^t\int_{\T^d}\psi(x,\rho)\nabla\cdot\left(\sigma(\rho)\,\de\xi^\epsilon\right).
\end{align}
\end{enumerate}
\end{definition}

\begin{remark}\label{remark/derivatives in kinetic formulation and interpretation of stochastic integral}
\begin{itemize}
    \item[(i)]
    We write $(\nabla\psi)(x,\rho(x,t))=\left.\nabla\psi(x,z)\right|_{z=\rho(x,t)}$
    to mean the gradient of $\nabla_x\psi(x,z)$ evaluated at the point $(x,\rho(x,t))$ as opposed to the full gradient of the composition $\psi(x,\rho(x,t))$.
    \item[(ii)] For $\F_t$-adapted processes $g_t\in L^2(\Omega\times[0,T];L^2(\T^d))$ and $h_t\in L^2(\Omega\times[0,T];H^1(\T^d))$ and for $t\in[0,T]$, we write
    \begin{align*}
    \int_0^t\int_{\T^d}g_s\nabla\cdot\left(h_s\,\de\xi^\epsilon\right)=\sum_{k=1}^\infty\left(\int_0^t\int_{\T^d}g_sf_k^\epsilon\nabla h_s\cdot \de B^k_s+\int_0^t\int_{\T^d}g_sh_s\nabla f_k^\epsilon\cdot \de B^k_s\right),
    \end{align*}
    where the integrals are interpreted in the It\^o sense.
\end{itemize}
\end{remark}

The upshot of the kinetic formulation of \eqref{equation/generalized dean--kawasaki equation 3} is of course that stochastic kinetic solutions exist and are unique for a much wider class of coefficients $\phi$ and $\sigma$.

\begin{theorem}[Theorem~5.29 in \cite{fehrman-gess-Well-posedness-of-the-Dean-Kawasaki-and-the-nonlinear-Dawson-Watanabe-equation-with-correlated-noise}]
\label{theorem/existence and uniquess of kinetic solutions to dean-kawasaki}
Let $\rho_0$ satisfy Assumption \ref{assumption/assumption I1}(i), let $\xi^\epsilon$ satisfy Assumption \ref{assumption/assumption N1}, and let $\phi$, $\nu$ and $\sigma$ satisfy Assumption \ref{assumption/assumption C1}.
There exists a unique stochastic kinetic solution $\rho$ to \eqref{equation/generalized dean--kawasaki equation 3} in the sense of Definition \ref{definition/stochastic kinetic solution}.
Furthermore, for $p\geq2$ and $m\geq1$ given in Assumption \ref{assumption/assumption C1}, for $n=2$ and $n=p$, it satisfies
    \begin{align}\label{equation/estimate 1 for kinetic solutions}
        \sup_{t\in[0,T]}&\E\left[\int_{\T^d}|\rho(x,t)|^n\,dx\right]+\inf_z\dot{\phi}(z)\,\E\left[\int_0^T\int_{\T^d}|\rho|^{n-2}|\nabla\rho|^2\,dx\,dt\right]
        \\
        &\leq C\,(1+\epsilon\,\|F_3^\epsilon\|_{\infty})^{\frac{d}{2}(p+m)}\left(1+\E\left[\|\rho_0\|_{L^1(\T^d)}^{m+n-1}+\|\rho_0\|_{L^n(\T^d)}^n\right]\right),
\end{align}
for a constant $C=C(T,\phi,\nu,\sigma)$ depending on $\phi$, $\nu$ and $\sigma$ only through the constants $c$ appearing in Assumption \ref{assumption/assumption C1}.
\end{theorem}

Weak and stochastic kinetic solutions are obviously related.
First of all, when they exist, weak solutions are stochastic kinetic solutions.

\begin{proposition}[Proposition~5.21 in \cite{fehrman-gess-Well-posedness-of-the-Dean-Kawasaki-and-the-nonlinear-Dawson-Watanabe-equation-with-correlated-noise}]
\label{proposition/weak solution implies kinetic solution to dean-kawasaki}
Let $\rho_0$ satisfy Assumption \ref{assumption/assumption I1}(i), let $\xi^\epsilon$ satisfy Assumption \ref{assumption/assumption N1} and let $\phi$, $\nu$, $\sigma$ satisfy Assumption \ref{assumption/assumption C1} and the additional assumptions (i)-(ii) in Theorem \ref{theorem/existence and uniqueness of weak solution to dean_kawasaki}.
Let $\rho$ be the unique weak solution of \eqref{equation/generalized dean--kawasaki equation 3} in the sense of Definition \ref{definition/weak solution to generalized dean kawasaki}, then $\rho$ is also a stochastic kinetic solution in the sense of Definition \ref{definition/stochastic kinetic solution} and its kinetic measure is given by
$q=\delta_0(z-\rho)\dot{\phi}(z)|\nabla\rho|^2dx\,dt$.
\end{proposition}

The following statement asserts that stochastic kinetic solutions, and thus also weak solutions thanks to the previous proposition, depend continuously on the coefficients of the equation, and it will be crucial in our arguments.
The proof is a straightforward adaptation of the final part of the proof of \cite[Theorem~5.29]{fehrman-gess-Well-posedness-of-the-Dean-Kawasaki-and-the-nonlinear-Dawson-Watanabe-equation-with-correlated-noise}.
\begin{proposition}
\label{proposition/kinetic solutions of dean-kawasaki depend continuously on the coefficients}
Let $\rho_0$ satisfy Assumption \ref{assumption/assumption I1}(i), let $\xi^\epsilon$ satisfy Assumption \ref{assumption/assumption N1} and let $\phi,\nu,\sigma$ satisfy Assumption \ref{assumption/assumption C1}. 
For $\eta\in(0,1)$, let $\phi^\eta,\nu^\eta,\sigma^\eta$ be a sequence of coefficients such that:
\begin{itemize}
    \item [(i)]$\phi^\eta,\nu^\eta,\sigma^\eta$ satisfy Assumption \ref{assumption/assumption C1} uniformly in $\eta\in(0,1)$;
    \item [(ii)]$\phi^\eta\to\phi$, $\nu^\eta\to\nu$ and $\sigma^\eta\to\sigma$ in $C_{\text{loc}}([0,\infty))\cap C^1_{\text{loc}}((0,\infty))$ as $\eta\to0$.
\end{itemize}
For each $\eta\in(0,1)$, let $\rho^{\epsilon,\eta}$ be the unique stochastic kinetic solution of
\begin{align}[left ={\empheqlbrace}]\label{equation/dean--kawasaki with smoothed coefficients 2}
\de\rho^{\epsilon,\eta}=&\Delta\phi^\eta(\rho^{\epsilon,\eta})\de t-\nabla\cdot\nu^{\eta}(\rho^{\epsilon,\eta})\,\de t-\sqrt{\epsilon}\,\nabla\!\cdot\!\left(\sigma^\eta(\rho^{\epsilon,\eta}) \,\de\xi^{\epsilon}\right)
\\
&+\frac{\epsilon}{2}\nabla\cdot\left(F_1^\epsilon \left(\Dot{\sigma}^\eta(\rho^{\epsilon,\eta})\right)^2\nabla\rho^{\epsilon,\eta}+\dot{\sigma}^\eta(\rho^{\epsilon,\eta})\sigma^\eta(\rho^{\epsilon,\eta})F_2^\epsilon\right)\de t,
\\
\rho^{\epsilon,\eta}(\cdot,0)&=\rho_0,
\end{align}
and let $\rho^\epsilon$ be the stochastic kinetic solution of \eqref{equation/generalized dean--kawasaki equation 3} with initial data $\rho_0$, in the sense of Definition \ref{definition/stochastic kinetic solution}.
Then we have
\begin{equation}
    \label{equation/kinetic solutions of dean-kawasaki depend continuously on the coefficients}
    \lim_{\eta\to0}\left\|\rho^{\epsilon,\eta}-\rho^\epsilon\right\|_{L^1([0,T];L^1(\T^d))}=0\quad\text{in probability.}
\end{equation}
\end{proposition}

\medskip

For the sake of completeness, we end this section by establishing that, if the stochastic convective term $\sigma$ is more regular near zero and somewhat dominated by the diffusion, we have the converse of Proposition \ref{proposition/weak solution implies kinetic solution to dean-kawasaki}: stochastic kinetic solutions are weak solutions.
Since we will not need this fact in the following, we only sketch the main details.
The proof is based on the following proposition, which establishes the $H^1$-regularity of the nonlinear term $\phi^{\nicefrac{1}{2}}(\rho)$ imposed by the diffusion.

\begin{proposition}[Corollary 5.31 in \cite{fehrman-gess-Well-posedness-of-the-Dean-Kawasaki-and-the-nonlinear-Dawson-Watanabe-equation-with-correlated-noise}]
\label{proposition/ dissipation estimate}
Let  $\rho_0$ satisfy Assumption \ref{assumption/assumption I1}, let $\xi^\epsilon$ satisfy Assumption \ref{assumption/assumption N1} and let $\phi,\nu$ and $\sigma$ satisfy Assumption \ref{assumption/assumption C1}.
Let $\rho$ be the unique stochastic kinetic solution to equation \eqref{equation/generalized dean--kawasaki equation 3}.
Assume furthermore the following:
\begin{itemize}
\item[(i)] there exists $c\in(0,\infty)$ such that $|\sigma(z)| \leq c\,\phi^{\nicefrac{1}{2}}(z)$ for every $z\in [0,\infty)$;
\item[(ii)] there exists $c\in(0,\infty)$ such that
\begin{equation}
    |\nu(z)|+\dot{\phi}(z)\leq c\,\left(1+z+\phi(z)\right)\,\,\text{for every}\,\,z\in[0,\infty);
\end{equation}
\item [(iii)] the map $z\mapsto\log(\phi(z))$ is locally integrable on $[0,\infty)$.
\end{itemize}
Then for the unique function $\Psi_\phi\in C([0,\infty);\R)\cap C^1((0,\infty);\R)$ satisfying $\Psi_\phi(0)=0$ and $\dot{\Psi}_\Phi(z)=\log(\phi(z))$, for a constant $C=C(T,\phi,\nu,\sigma)$ depending on $\phi$, $\nu$ and $\sigma$ only through the constants $c$ appearing in Assumption \ref{assumption/assumption C1} and in the further hypotheses (i)-(iii) above, we have
\begin{align}\label{equation/entropy dissipation estimate}
\begin{split}
\E\bigg[&\sup_{t\in[0,T]}\int_{\T^d}\Psi_\phi(\rho(x,t))\bigg]+\E\left[\int_0^T\int_{\T^d}\left|\nabla\phi^{\nicefrac{1}{2}}(\rho)\right|^2\right]+\inf_{z>0}\dot{\phi}(z)\,\E\left[\int_0^T\int_{\T^d}\mathbf{1}_{\{\rho>0\}}\,\frac{\phi'(\rho)}{\phi(\rho)}\left|\nabla\rho\right|^2\right]
\\ 
& \leq C \,\bigg(1+\|\epsilon\,F_1^\epsilon\|_{\infty}+\|\epsilon\,F_3^\epsilon\|_{\infty}^{\frac{d}{2}(m+1)}\bigg)\left(1+\E\left[\|\rho_0\|^m_{L^1(\T^d)}+\int_{\T^d}\Psi_\phi(\rho_0)\right]\right).
\end{split}
\end{align}
\end{proposition}

We can now prove the following.

\begin{proposition}
\label{proposition/kinetic solutions implies weak solution to deak-kawasaki}
In the setting of Proposition \ref{proposition/ dissipation estimate}, assume furthermore that, for some $c>0$,
\begin{equation}\label{equation/sigma' leq (phi^1/2)'} 
    |\dot{\sigma}(z)|\leq c\,\dot{\left(\phi^{\nicefrac{1}{2}}\right)}(z)\quad\text{and}\quad|\dot{\sigma}(z)|^2\leq c\,\left(1+z+\phi(z)\right)\quad\forall\,z\in(0,\infty).
\end{equation}
Then the stochastic kinetic solution $\rho$ of \eqref{equation/generalized dean--kawasaki equation 3} satisfies the weak formulation \eqref{equation/definition of weak solution to dean-kawasaki} of the equation.
\end{proposition}

\begin{proof}
The idea is to test the kinetic formulation \eqref{equation/kinetic formulation of dean-kawasaki} of the equation against functions of the form $\psi(x,z)=\varphi(x)K_{\delta}(z)$ for arbitrary $\varphi\in C^{\infty}(\T^d)$ and for a smooth cut-off $K_\delta:\R\to[0,1]$ such that
\begin{equation}
    K_\delta(z)=\empheqlbrace
    \begin{aligned}
        &1\quad\text{if $z\in\left[2\delta,\nicefrac{1}{\delta}\right]$},\\
        &0\quad\text{if $z\in[0,\delta]\cup\left[\nicefrac{2}{\delta},\infty\right)$},
    \end{aligned}
\end{equation}
and then let $\delta\to0$ and show that, along a suitable subsequence $\delta_n\to0$, the resulting expression coincides with the weak formulation \eqref{equation/definition of weak solution to dean-kawasaki}.

Except for the term involving the kinetic measure, each of the integrals in \eqref{equation/kinetic formulation of dean-kawasaki} is shown to converge to the corresponding quantity in \eqref{equation/definition of weak solution to dean-kawasaki} or to zero by dominated convergence.
This is justified using Assumption \ref{assumption/assumption C1}, the regularity estimates \eqref{equation/estimate 1 for kinetic solutions} and \eqref{equation/entropy dissipation estimate}, the further hypotheses (i)-(iii) in Proposition \ref{proposition/ dissipation estimate}, and the hypothesis \eqref{equation/sigma' leq (phi^1/2)'}, which together with \eqref{equation/entropy dissipation estimate} implies that $\dot{\sigma}(\rho)\in L^2([0,T];L^2(\T^d))$ and $\sigma(\rho)\in L^2([0,T];H^1(\T^d))$ with
\begin{equation}
    \left|\nabla\sigma(\rho)\right|\leq C  \left|\nabla\phi^{\nicefrac{1}{2}}(\rho)\right|.
\end{equation}
The term involving the kinetic measure is shown to converge to zero along a subsequence $\delta_n\to0$ using property \eqref{equation/kinetic measure vanishes at infinity} and the following property, which is proved in \cite[Proposition 4.6]{fehrman-gess-Well-posedness-of-the-Dean-Kawasaki-and-the-nonlinear-Dawson-Watanabe-equation-with-correlated-noise},
\begin{equation}
    \liminf_{\delta\rightarrow 0}\left(\delta^{-1}q(\T^d\times [\nicefrac{\delta}{2},\delta]\times[0,T])\right)= 0.
\end{equation}
\end{proof}

\section{The central limit theorem for fluctuations}
\label{section/The central limit theorem for fluctuations}

\subsection{The central limit theorem in $L^2(\Omega)$ for more regular coefficients}

This section is devoted to the proof of our central limit theorems: Theorem \ref{theorem/clt for fluctuations 1} and \ref{theorem/clt for fluctuations 3}.
In the first part we obtain the stronger version \eqref{equation/rate of convergence 1} of the CLT, for coefficients satisfying Assumption \ref{assumption/assumption C1} and the stronger Assumption \ref{assumption/assumption C2}.
In the second part we extend the result to rougher coefficients satisfying Assumption \ref{assumption/assumption C1} and Assumption \ref{assumption/assumption C2 weak} only.

The nonlinearity of the diffusion $\phi$ makes it difficult to apply standard Fourier analysis and we have to accompany this with some moment estimates.
Our ansatz is that $v^\epsilon$ converge to $v$, solving \eqref{equation/OU equation 1}, which can be distribution valued.
Therefore we expect the norm of $v^\epsilon$ in any function space to blow up as $\epsilon\to0$.
After an approximation lemma, the following proposition quantifies this explosion.

\begin{lemma}\label{lemma/smooth approximations phi^eta sigma^eta}
Let $\phi,\nu,\sigma\in W^{1,1}_{\text{loc}}([0,\infty))\cap C^1((0,\infty))$ with $\phi(0)=\sigma(0)=0$ and $\dot{\phi}(z)>0$ for every $z\in(0,\infty)$.
There exists a sequence of approximations $(\phi^\eta,\nu^\eta,\sigma^\eta)_{\eta\in(0,1)}$ such that
\begin{itemize}
    \item [(i)] $\phi^\eta\to\phi$, $\nu^\eta\to\nu$ and $\sigma^\eta\to\sigma$ in $C^1_{\text{loc}}(0,\infty)\cap C_{\text{loc}}([0,\infty))$ as $\eta\to0$;
    \item[(ii)] $\phi^\eta,\nu^\eta,\sigma^\eta$ satisfy Assumption \ref{assumption/assumption C1}, \ref{assumption/assumption C2 weak}, \ref{assumption/assumption C2} and Assumption (i)-(ii) in Theorem \ref{theorem/existence and uniqueness of weak solution to dean_kawasaki}.
\end{itemize}
Furthermore, if $\phi,\,\nu$ and $\sigma$ satisfy Assumption \ref{assumption/assumption C1}, \ref{assumption/assumption C2 weak}, \ref{assumption/assumption C2} or Assumption (i) or (ii) in Theorem \ref{theorem/existence and uniqueness of weak solution to dean_kawasaki}, then $\phi^\eta,\nu^\eta, \sigma^\eta$ satisfy Assumption \ref{assumption/assumption C1}, \ref{assumption/assumption C2 weak}, \ref{assumption/assumption C2} or Assumption (i) or (ii) in Theorem \ref{theorem/existence and uniqueness of weak solution to dean_kawasaki}, respectively, uniformly in $\eta\in(0,1)$.
\end{lemma}
\begin{proof}
A standard smoothing procedure by convolution and cut-off.
\end{proof}

\begin{proposition}\label{proposition/lp estimates}
Let $\rho_0$ satisfy Assumption \ref{assumption/assumption I1}(ii), let $\xi^\epsilon$ satisfy Assumption \ref{assumption/assumption N1} and let $\phi,\nu,\sigma$ satisfy Assumption \ref{assumption/assumption C1} and \ref{assumption/assumption C2}, for some $p\geq2$ and $m\geq1$.
Let $\rho^\epsilon$ be the stochastic kinetic solution to \eqref{equation/generalized dean--kawasaki equation 3} with initial data $\rho_0$, let $\Bar{\rho}\equiv\rho_0$ be the solution of \eqref{equation/hydrodynamic limit equation 1} and define $v^\epsilon=\epsilon^{-\nicefrac{1}{2}}(\rho^\epsilon-\Bar{\rho})$.
Then, for any $T\in[0,\infty)$ and any $h\in\big[1,\frac{p}{(k+1)}\big]$, we have
{\small
\begin{equation}\label{equation/lp estimates}
    \E\left[\|v^\epsilon\|_{L^h([0,T];\T^d)}^h\right]\leq C\,\|F_3^\epsilon\|_{\infty}^{\nicefrac{h}{2}}(1+\epsilon\,\|F_3^\epsilon\|_{\infty})^{\frac{d}{2} (p+m)}\left(1+\E\left[\|\rho_0\|_{L^1(\T^d)}^{m+p-1}+\|\rho_0\|_{L^{p}(\T^d)}^{p}\right]\right),
\end{equation}
}
for a constant $C=C(T,\phi,\nu,\sigma,p)$ depending on $\phi$, $\nu$ and $\sigma$ only through the constants $c$ featuring in Assumption \ref{assumption/assumption C1} and \ref{assumption/assumption C2}.
\end{proposition}

\begin{proof}
Consider the sequence $(\phi^\eta,\nu^\eta,\sigma^\eta)_{\eta\in(0,1)}$ of approximations given by Lemma \ref{lemma/smooth approximations phi^eta sigma^eta}.
Let $\rho^{\epsilon,\eta}$ be the corresponding weak solution of \eqref{equation/dean--kawasaki with smoothed coefficients 2}, whose existence is guaranteed by Theorem \ref{theorem/existence and uniqueness of weak solution to dean_kawasaki}, and let $\Bar{\rho}\equiv\rho_0$ be the unique solution of $\partial_t\Bar{\rho}=\Delta\phi^{\eta}(\Bar{\rho})-\nabla\!\cdot\!\nu^\eta(\bar{\rho})$ with initial data $\rho_0$.
It follows that the fluctuations $v^{\epsilon,\eta}=\epsilon^{-\nicefrac{1}{2}}(\rho^{\epsilon,\eta}-\Bar{\rho})$ are a weak solution, in the sense of Definition \ref{definition/weak solution to generalized dean kawasaki}, to
\begin{align}[left ={\empheqlbrace}]\label{proposition/lp estimates/proof 1}
\de v^{\epsilon,\eta}=&\Delta\left(\epsilon^{-\nicefrac{1}{2}}\big(\phi^\eta(\rho^{\epsilon,\eta})-\phi^\eta(\Bar{\rho})\big)\right)\de t-\nabla\cdot\left(\epsilon^{-\nicefrac{1}{2}}\big(\nu^\eta(\rho^{\epsilon,\eta})-\nu^\eta(\bar{\rho})\big)\right)\de t-\nabla\!\cdot\!\left(\sigma^\eta(\rho^{\epsilon,\eta}) \,\de\xi^{\epsilon}\right)
\\
&+\frac{\sqrt{\epsilon}}{2}\nabla\cdot\left(F_1^\epsilon \Dot{\sigma}^\eta(\rho^{\epsilon,\eta})\nabla\sigma^\eta(\rho^{\epsilon,\eta})+\dot{\sigma}^\eta(\rho^{\epsilon,\eta})\sigma^\eta(\rho^{\epsilon,\eta})F_2^\epsilon\right)\de t,
\\
v^{\epsilon,\eta}(\cdot,0)&=\,\,0.
\end{align}
Applying It\^o's formula, in the version proved in Krylov \cite{krylov_ito_formula}, yields:
\begin{align}\label{proposition/lp estimates/proof 2}
    \begin{split}
        \int_{\T^d}\left|v^{\epsilon,\eta}(x,t)\right|^h dx
        =
        &-\int_0^t\int_{\T^d}h(h-1)\left|v^{\epsilon,\eta}(x,t)\right|^{h-2}\nabla v^{\epsilon,\eta}\nabla\left(\epsilon^{-\nicefrac{1}{2}}\left(\phi^\eta(\rho^{\epsilon,\eta})-\phi^\eta(\Bar{\rho})\right)\right)\,dx\,ds
        \\
        &+\int_0^t\int_{\T^d}h(h-1)\left|v^{\epsilon,\eta}(x,t)\right|^{h-2}\nabla v^{\epsilon,\eta}\epsilon^{-\nicefrac{1}{2}}\big(\nu^\eta(\rho^{\epsilon,\eta})-\nu^\eta(\Bar{\rho})\big)\,dx\,ds
        \\
        &-\int_0^t\int_{\T^d}h\left|v^{\epsilon,\eta}(x,t)\right|^{h-2} v^{\epsilon,\eta}\nabla\cdot\left(\sigma^\eta(\rho^{\epsilon,\eta})\de \xi^\epsilon\right)
        \\
        &-\frac{\sqrt{\epsilon}}{2}\int_0^t\int_{\T^d}h(h-1)\left|v^{\epsilon,\eta}(x,t)\right|^{h-2}\nabla v^{\epsilon,\eta}
        \\
        &\qquad\qquad\cdot\left(F_1^\epsilon \Dot{\sigma}^\eta(\rho^{\epsilon,\eta})\nabla\sigma^\eta(\rho^{\epsilon,\eta})+\dot{\sigma}^\eta(\rho^{\epsilon,\eta})\sigma^\eta(\rho^{\epsilon,\eta})F_2^\epsilon\right)\,dx\,ds
        \\
        &+\frac{1}{2}\int_0^t\int_{\T^d}h(h-1)\left|v^{\epsilon,\eta}(x,t)\right|^{h-2}
        \\
        &\qquad\qquad\cdot\left(\left|\nabla\sigma^\eta(\rho^{\epsilon,\eta})\right|^2F_1^\epsilon+2\sigma^\eta(\rho^{\epsilon,\eta})\nabla\sigma^\eta(\rho^{\epsilon,\eta}) F_2^\epsilon+\left(\sigma^\eta(\rho^{\epsilon,\eta})\right)^2F_3^\epsilon\right)\,dx\,ds.
    \end{split}
\end{align}
Rearranging the terms, and using that $\nabla v^{\epsilon,\eta}=\epsilon^{-\nicefrac{1}{2}}\nabla\rho^{\epsilon,\eta}$ since $\Bar{\rho}\equiv\text{constant}$, the smoothness of $\phi^\eta$, $\nu^\eta$, $\sigma^\eta$, and the regularity properties of $\rho^{\epsilon,\eta}$ guaranteed by Theorem \ref{theorem/existence and uniqueness of weak solution to dean_kawasaki}, we obtain
\begin{align}\label{proposition/lp estimates/proof 2.bis}
    \begin{split}
        \int_{\T^d}&\left|v^{\epsilon,\eta}(x,t)\right|^h dx
        +\int_0^t\int_{\T^d}h(h-1)\left|v^{\epsilon,\eta}(x,s)\right|^{h-2}\nabla v^{\epsilon,\eta}\nabla\left(\epsilon^{-\nicefrac{1}{2}}\left(\phi^\eta(\rho^{\epsilon,\eta})-\phi^\eta(\Bar{\rho})\right)\right)\,dx\,ds
        \\
        =&\int_0^t\int_{\T^d}h(h-1)\left|v^{\epsilon,\eta}(x,s)\right|^{h-2}\nabla v^{\epsilon,\eta}\epsilon^{-\frac{1}{2}}\big(\nu^\eta(\rho^{\epsilon,\eta})-\nu^\eta(\Bar{\rho})\big)\,dx\,ds
        \\
        &-\int_0^t\int_{\T^d}h\left|v^{\epsilon,\eta}(x,s)\right|^{h-2} v^{\epsilon,\eta}\nabla\cdot\left(\sigma^\eta(\rho^{\epsilon,\eta})\de \xi^\epsilon\right)
        \\
        &+\frac{1}{2}\int_0^t\int_{\T^d}h(h-1)\left|v^{\epsilon,\eta}(x,s)\right|^{h-2} \sigma^\eta(\rho^{\epsilon,\eta})\nabla\sigma^\eta(\rho^{\epsilon,\eta})\cdot F_2^\epsilon\,dx\,ds
        \\
        &+\frac{1}{2}\int_0^t\int_{\T^d}h(h-1)\left|v^{\epsilon,\eta}(x,s)\right|^{h-2} \left(\sigma^\eta(\rho^{\epsilon,\eta})\right)^2F_3^\epsilon\,dx\,ds.
    \end{split}
\end{align}
The first term on the right-hand side of \eqref{proposition/lp estimates/proof 2.bis} is identically zero because of integration by parts and the rewriting
\begin{align}\label{proposition/lp estimates/ proof 3}
\begin{split}
    \int_{\T^d}&h(h-1)\left|v^{\epsilon,\eta}(x,s)\right|^{h-2}\nabla v^{\epsilon,\eta}\,\epsilon^{-\nicefrac{1}{2}}\big(\nu^\eta(\rho^{\epsilon,\eta})-\nu^\eta(\Bar{\rho})\big)\,dx
    \\
    &=
    \int_{\T^d}h(h-1)\nabla\cdot\left(\int_0^{\rho^{\epsilon,\eta}}\left|\frac{z-\bar{\rho}}{\sqrt{\epsilon}}\right|^{h-2}\epsilon^{-\nicefrac{1}{2}}\big(\nu^\eta(z)-\nu^\eta(\bar{\rho})\big)\,dz\right)\,dx
    =
    0.
\end{split}
\end{align}
Similarly, the third term on the right-hand side of \eqref{proposition/lp estimates/proof 2.bis} vanishes because of the assumption $\nabla\cdot F_2^\epsilon=0$ and the integration by parts
\begin{align}\label{proposition/lp estimates/proof 4}
\begin{split}
    \int_{\T^d}\left|v^{\epsilon,\eta}(x,s)\right|^{h-2}\sigma^\eta(\rho^{\epsilon,\eta})\nabla\sigma^\eta(\rho^{\epsilon,\eta})\cdot F_2^\epsilon\,dx
    =
    -\int_{\T^d}\left(\int_0^{\rho^{\epsilon,\eta}}\left|\frac{z-\Bar{\rho}}{\sqrt{\epsilon}}\right|^{h-2}\sigma^\eta(z)\,\dot{\sigma}^\eta(z)\,dz\right)\nabla\cdot F_2^\epsilon\,dx.
\end{split}
\end{align}
Furthermore, the second term on the right-hand side of \eqref{proposition/lp estimates/proof 2.bis} is a true martingale vanishing at time zero, thanks to Remark \ref{remark/derivatives in kinetic formulation and interpretation of stochastic integral}(ii), Assumption \ref{assumption/assumption N1} and Assumption (i)-(ii) in Theorem \ref{theorem/existence and uniqueness of weak solution to dean_kawasaki}, and the regularity of $\rho^{\epsilon,\eta}$.
Therefore, applying the expectation to \eqref{proposition/lp estimates/proof 2.bis} and further integrating over $t\in[0,T]$ yields:
\begin{align}\label{proposition/lp estimates/proof 3}
    \begin{split}
        \E\left[\left\|v^{\epsilon,\eta}\right\|^h_{L^h([0,T];\T^d)}\right]
        &+
        \E\left[\int_0^T\int_{\T^d}h(h-1)\left|v^{\epsilon,\eta}(x,t)\right|^{h-2}\left|\nabla v^{\epsilon,\eta}\right|^2\dot{\phi}^\eta(\rho^{\epsilon,\eta})\,(T-s)\,dx\,ds\right]
        \\
        =&
        \,\,\frac{1}{2}\E\left[\int_0^T\int_{\T^d}h(h-1)\left|v^{\epsilon,\eta}(x,t)\right|^{h-2}\left(\sigma^\eta(\rho^{\epsilon,\eta})\right)^2F_3^\epsilon\,(T-s)\,dx\,ds\right].
    \end{split}
\end{align}

We now estimate the right-hand side of \eqref{proposition/lp estimates/proof 3} as follows.
Assumption \ref{assumption/assumption C2}(i) on $\sigma$ and Lemma \ref{lemma/smooth approximations phi^eta sigma^eta} ensure that
\begin{equation}\label{proposition/lp estimates/proof 5}
    |\sigma^\eta(z)|\leq C\left(1+|z|^{k+1}\right),
\end{equation}
for a constant $C=C(\sigma)$ independent of $\eta\in(0,1)$.
Thus, applying H\"older's inequality, Young's inequality and formula \eqref{proposition/lp estimates/proof 5}, noticing that $h(k+1)\leq p$, yields:
\begin{align}\label{proposition/lp estimates/proof 6}
\begin{split}
    \frac{1}{2}\E&\left[\int_0^T\int_{\T^d}h(h-1)\left|v^{\epsilon,\eta}(x,s)\right|^{h-2}\left(\sigma^\eta(\rho^{\epsilon,\eta})\right)^2F_3^\epsilon\,(T-s)\,dx\,ds\right]
    \\
    &\leq\, C\,\left\|F_3^\epsilon\right\|_{\infty}
    \E\left[\int_0^T\int_{\T^d}\left|v^{\epsilon,\eta}\right|^{h}dx\,ds\right]^{\frac{h-2}{h}}
    \E\left[\int_0^T\int_{\T^d}\left(\sigma^\eta(\rho^{\epsilon,\eta})\right)^hdx\,ds\right]^{\frac{2}{h}}
    \\
    &\leq
    \frac{1}{2}\E\left[\left\|v^{\epsilon,\eta}\right\|^h_{L^h([0,T];\T^d)}\right]
    +
    C\,\left\|F_3^\epsilon\right\|_{\infty}^{\nicefrac{h}{2}}
    \left(1+\E\left[\int_0^T\int_{\T^d}(\rho^{\epsilon,\eta})^{p}dx\,ds\right]\right),
\end{split}
\end{align}
for a constant $C=C(\sigma,p,T)$ independent of $\eta\in(0,1)$.
Going back to \eqref{proposition/lp estimates/proof 3}, we use \eqref{proposition/lp estimates/proof 6} and we absorb the first term on the right hand side of \eqref{proposition/lp estimates/proof 6} into the corresponding term on the left hand side of \eqref{proposition/lp estimates/proof 3}, thanks to the factor $\frac{1}{2}$ in front, to obtain
\begin{equation}\label{proposition/lp estimates/proof 7}
    \E\left[\left\|v^{\epsilon,\eta}\right\|_{L^h([0,T];\T^d)}^h\right]
    \leq 
    C\,\|F_3^\epsilon\|_{\infty}^{\nicefrac{h}{2}}\left(1+\E\left[\int_0^T\int_{\T^d}(\rho^{\epsilon,\eta})^{p}dx\,ds\right]\right),
\end{equation}
for a constant $C=C(\sigma,p,T)$ independent of $\eta$.

Now we recall that Lemma \ref{lemma/smooth approximations phi^eta sigma^eta} ensures that $\phi^\eta,\nu^\eta,\sigma^\eta$ satisfy Assumption \ref{assumption/assumption C1} and \ref{assumption/assumption C2} uniformly in $\eta\in(0,1)$, since these are satisfied by $\phi,\nu,\sigma$.
Therefore estimate \eqref{equation/estimate 1 for kinetic solutions} holds for a constant $C=C(\phi,\nu,\sigma,p,T)$ independent of $\eta\in(0,1)$.
Using estimate \eqref{equation/estimate 1 for kinetic solutions} with $n=p$ in \eqref{proposition/lp estimates/proof 7}, we obtain
\begin{align}\label{proposition/lp estimates/proof 8}
\begin{aligned}
    \E\left[\left\|v^{\epsilon,\eta}\right\|_{L^h([0,T];\T^d)}^h\right]
    \leq
    C\,\|F_3^\epsilon\|_{\infty}^{\nicefrac{h}{2}}
    \left(1+\epsilon\,\|F_3^\epsilon\|_{\infty}\right)^{\frac{d}{2}(p+m)}
    \left(1+\E\left[\|\rho_0\|_{L^1(\T^d)}^{m+p-1}+\|\rho_0\|_{L^{p}(\T^d)}^{p}\right]\right),
\end{aligned}
\end{align}
for a constant $C=C(\sigma,\nu,\phi,p,T)$ independent of $\eta$.
Finally, Proposition \ref{proposition/kinetic solutions of dean-kawasaki depend continuously on the coefficients} ensures that 
\[\lim_{\eta\to0}\left\|\rho^{\epsilon,\eta}-\rho^\epsilon\right\|_{L^1([0,T];L^1(\T^d))}=0\quad\text{in probability.}\]
Therefore, upon passing to a subsequence $\eta_n\to0$, we have $\rho^{\epsilon,\eta_n}\to\rho^\epsilon$ for a.e. $(x,t,\omega)\in\T^d\times[0,T]\times\Omega$.
Letting $\eta_n\to0$ in \eqref{proposition/lp estimates/proof 8} and applying Fatou's Lemma eventually yield formula \eqref{equation/lp estimates}.
\end{proof}

We are now ready to estimate the difference between the actual fluctuations $v^\epsilon$ and their asymptotic description $v$, and obtain the CLT in $L^2(\Omega)$.

\begin{theorem}[Central limit theorem in $L^2(\Omega)$]\label{theorem/clt for fluctuations 1}
Let $\rho_0$ satisfy Assumption \ref{assumption/assumption I1}(ii) and let $\phi,\nu,\sigma$ satisfy Assumption \ref{assumption/assumption C1} and \ref{assumption/assumption C2}, for some $p\ge2$ and $m\geq1$.
Let $\xi=\lim_{\epsilon\to0}\xi^\epsilon$, where $(\xi^\epsilon)_{\epsilon>0}$ satisfy Assumption \ref{assumption/assumption N2}, and let $v$ be the corresponding solution of the Langevin equation \eqref{equation/OU equation 1}.
For any $\epsilon>0$, let $\rho^\epsilon$ be the stochastic kinetic solution to the generalized Dean--Kawasaki equation \eqref{equation/generalized dean--kawasaki equation 3} with initial data $\rho_0$, let $\Bar{\rho}\equiv\rho_0$ be the solution of the zero noise limit\eqref{equation/hydrodynamic limit equation 1} and let $v^\epsilon=\epsilon^{-\nicefrac{1}{2}}(\rho^\epsilon-\Bar{\rho})$.
Then, for any $T\in[0,\infty)$, for $\tau=2$ or $\tau=\infty$, for any $\beta>\frac{d}{2}$ or $\beta>\frac{d}{2}+1$ respectively, we have
{\small
\begin{align}
\begin{split}
    \label{equation/rate of convergence 2}
    &\E\left[\left\|v^{\epsilon}(t)\!-\!v(t)\right\|^2_{L^{\tau}([0,T];H^{-\beta}(\T^d))}\right]
    \\
    &\leq
    C
    \left(\epsilon\left(|F_1^\epsilon|_{\infty}+|F_3^\epsilon|_{\infty}\right)^2+\epsilon^{\nicefrac{1}{2}}|F_3^\epsilon|_{\infty}^{\nicefrac{1}{2}}\right)
    \left(1+\epsilon\,|F_3^\epsilon|_\infty\right)^{g+\frac{d}{2}(p+m)}
    \left(1+\E\left[\|\rho_0\|_{L^1(\T^d)}^{m+p-1}+\|\rho_0\|_{L^p(\T^d)}^p\right]\right)
    \\
    &\quad
    +C\sum_{n\in\Z^d}n^{(2-\frac{4}{\tau})-2\beta}  \sum_{k} \int_0^T\left|\int_{\T^d} e^{i2\pi nx}\,(f_k^\epsilon-f_k)\,dx\right|^2 ds,
\end{split}
\end{align}
}
for a constant $C=C(T,\bar\rho,\phi,\nu,\sigma,p,\beta)$ depending on $\phi$, $\nu$ and $\sigma$ only through the constants $c$ featuring in Assumption \ref{assumption/assumption C1} and \ref{assumption/assumption C2}.

In particular, along a scaling regime where $\epsilon\to0$ and $\xi^\epsilon\to\xi$ such that
\begin{equation}
    \label{equation/scaling regime for noise}
    \lim_{\epsilon\to0}\,\,\sqrt{\epsilon}\,
    \Big(\!\left\|F_1^\epsilon\right\|_{\infty}\!\!+\left\|F_2^\epsilon\right\|_{\infty}+\|F_3^\epsilon\|_{\infty}\Big)=0,
\end{equation}
the nonequilibrium fluctuations $v^\epsilon$ converge to $v$ in $L^2(\Omega;L^\tau_{\text{loc}}([0,\infty);H^{-\beta}(\T^d)))$  with rate of convergence \eqref{equation/rate of convergence 2}.
\end{theorem}

\begin{proof}
As in the proof of Proposition \ref{proposition/lp estimates}, consider the sequence of approximations $(\phi^\eta,\nu^\eta,\sigma^\eta)_{\eta\in(0,1)}$ given in Lemma \ref{lemma/smooth approximations phi^eta sigma^eta}, let $\rho^{\epsilon,\eta}$ be the solution of \eqref{equation/dean--kawasaki with smoothed coefficients 2}, let $\Bar{\rho}\equiv\rho_0$ be the solution of \eqref{equation/hydrodynamic limit equation 1} and let $v^{\epsilon,\eta}=\epsilon^{-\nicefrac{1}{2}}(\rho^{\epsilon,\eta}-\Bar{\rho})$ be the corresponding fluctuations.
Finally let $v^{\eta}$ be the solution of
\begin{equation}\label{theorem/clt for fluctuations 1/proof 1}
    \de v^\eta=\Delta\big(\dot{\phi}^\eta(\Bar{\rho})\, v^\eta\big)\,\de t-\nabla\cdot\left(\dot{\nu}(\bar\rho)\,v^\eta\right)\,\de t-\nabla\cdot(\sigma^\eta(\Bar{\rho})\,\de \xi),\quad v^\eta(\cdot,0)=0,
\end{equation}
as given in Theorem \ref{theorem/ well-posedness of OU equation}.
The difference $v^{\epsilon,\eta}-v^\eta$ is a weak solution of
{\small
\begin{align}[left ={\empheqlbrace}]\label{theorem/clt for fluctuations 1/proof 2}
\begin{split}
\de (v^{\epsilon,\eta}-v^\eta)=&\Delta\left(\epsilon^{-\nicefrac{1}{2}}\big(\phi^\eta(\rho^{\epsilon,\eta})-\phi^\eta(\Bar{\rho})\big)-\dot{\phi}^\eta(\Bar{\rho})v^\eta\right)\de t,
\\
&
-\nabla\cdot\left(\epsilon^{-\nicefrac{1}{2}}\big(\nu^\eta(\rho^{\epsilon,\eta})-\nu^\eta(\Bar{\rho})\big)-\dot{\nu}^\eta(\Bar{\rho})v^\eta\right)\de t
-\nabla\!\cdot\!\big(\sigma^\eta(\rho^{\epsilon,\eta}) \,\de\xi^{\epsilon}-\sigma^\eta(\Bar{\rho})\,\de\xi\big)
\\
&+\frac{\epsilon^{\nicefrac{1}{2}}}{2}\nabla\cdot\left(F_1^\epsilon \left(\Dot{\sigma}^\eta(\rho^{\epsilon,\eta})\right)^2\nabla\rho^{\epsilon,\eta}+\dot{\sigma}^\eta(\rho^{\epsilon,\eta})\sigma^\eta(\rho^{\epsilon,\eta})F_2^\epsilon\right)\de t,
\\
v^{\epsilon,\eta}(\cdot,0)=&\,\,0.
\end{split}
\end{align}}
Using the fundamental theorem of calculus, this is rewritten as 
{\small
\begin{align}\label{theorem/clt for fluctuations 1/proof 3}
\begin{split}
\de (v^{\epsilon,\eta}-v^\eta)
=
&
\Delta\left(\dot{\phi}^\eta(\Bar{\rho})(v^{\epsilon,\eta}-v^{\eta})
+\int_0^1\int_0^1\ddot{\phi}^{\eta}(\bar\rho+\lambda\mu(\rho^{\epsilon,\eta}-\bar\rho))\,\lambda\,d\lambda d\mu\,(v^{\epsilon,\eta})^2\,\epsilon^{\nicefrac{1}{2}}\right)\de t
\\
&
-\nabla\cdot\left(\dot{\nu}^\eta(\Bar{\rho})(v^{\epsilon,\eta}-v^{\eta})
+\int_0^1\int_0^1\ddot{\nu}^{\eta}(\bar\rho+\lambda\mu(\rho^{\epsilon,\eta}-\bar\rho))\,\lambda\,d\lambda d\mu\,(v^{\epsilon,\eta})^2\,\epsilon^{\nicefrac{1}{2}}\right)\de t
\\
&
+\frac{\epsilon^{\nicefrac{1}{2}}}{2}\nabla\cdot\left(F_1^\epsilon\,\, \nabla\!\int_0^{\rho^{\epsilon,\eta}}\!\!\!(\dot{\sigma}^\eta(z))^2\,dz\,+\dot{\sigma}^\eta(\rho^{\epsilon,\eta})\sigma^\eta(\rho^{\epsilon,\eta})F_2^\epsilon\right)\de t
\\
&
-\nabla\!\cdot\!\big((\sigma^\eta(\rho^{\epsilon,\eta})-\sigma^\eta(\Bar{\rho}))\,\de\xi^{\epsilon} + \sigma^\eta(\Bar{\rho})\,\de(\xi^\epsilon-\xi)\big).
\end{split}
\end{align}}

Using \eqref{theorem/clt for fluctuations 1/proof 3}, recalling that $\nabla F_1^\epsilon=2 F_2^\epsilon$ and integrating by parts several times, we compute the Fourier coefficients of $v^{\epsilon,\eta}-v^\eta$.
Namely, for each $n\in \Z^d$, for $e_n(x):=e^{i2\pi n\cdot x}$ and $\hat{v}^{\epsilon,\eta}_n-\hat{v}^\eta_n:=\langle v^{\epsilon,\eta}-v^{\eta},e_n\rangle$,
{\small
\begin{align}\label{theorem/clt for fluctuations 1/proof 4}
\begin{split}
\de (\hat{v}^{\epsilon,\eta}_n-\hat{v}^\eta_n)
=
&
\int_{\T^d}\Delta e_n(x)\left(\dot{\phi}^\eta(\Bar{\rho})(v^{\epsilon,\eta}-v^{\eta})
+\int_0^1\int_0^1\ddot{\phi}^{\eta}(\bar\rho+\lambda\mu(\rho^{\epsilon,\eta}-\bar\rho))\,\lambda\,d\lambda d\mu\,(v^{\epsilon,\eta})^2\,\epsilon^{\nicefrac{1}{2}}\right)d x\,\,\de t
\\
&
+\int_{\T^d}\nabla e_n(x)\,\left(\dot{\nu}^\eta(\Bar{\rho})(v^{\epsilon,\eta}-v^{\eta})
+\int_0^1\int_0^1\ddot{\nu}^{\eta}(\bar\rho+\lambda\mu(\rho^{\epsilon,\eta}-\bar\rho))\,\lambda\,d\lambda d\mu\,(v^{\epsilon,\eta})^2\,\epsilon^{\nicefrac{1}{2}}\right)dx\,\,\de t
\\
&
+\frac{\epsilon^{\nicefrac{1}{2}}}{2}
\int_{\T^d}\left(\Delta e_n F_1^\epsilon+\nabla e_n2 F_2^\epsilon\right)\int_0^{\rho^{\epsilon,\eta}}\!\!\!(\dot{\sigma}^\eta(z))^2\,dz\,
-
\nabla e_n F_2^\epsilon \dot{\sigma}^\eta(\rho^{\epsilon,\eta})\sigma^\eta(\rho^{\epsilon,\eta})\,dx\,\,\de t
\\
&
+\int_{\T^d}\nabla e_n(x)\big((\sigma^\eta(\rho^{\epsilon,\eta})-\sigma^\eta(\Bar{\rho}))\,\de\xi^{\epsilon} + \sigma^\eta(\Bar{\rho})\,\de(\xi^\epsilon-\xi)\big).
\end{split}
\end{align}}
This is an SDE for $\hat{v}^{\epsilon,\eta}_n-\hat{v}^\eta_n$, which is readily solved by variation of constants:
{\small
\begin{align}\label{theorem/clt for fluctuations 1/proof 5}
\begin{split}
&\hat{v}^{\epsilon,\eta}_n(t)-\hat{v}^\eta_n(t)
\\
&=
\epsilon^{\nicefrac{1}{2}}\int_0^te^{(-4\pi^2n^2\dot{\phi}^{\eta}(\bar\rho)+i2\pi n\dot{\nu}^{\eta}(\bar\rho))(t-s)}
\\
&\qquad\qquad
\cdot\left(
\int_{\T^d}\!\!\Delta e_n \int_0^1\int_0^1\big(\ddot{\phi}^{\eta}(\bar\rho+\lambda\mu(\rho^{\epsilon,\eta}-\bar\rho))+\ddot{\nu}^{\eta}(\bar\rho+\lambda\mu(\rho^{\epsilon,\eta}-\bar\rho))\big)\lambda\,d\lambda d\mu\,(v^{\epsilon,\eta}(x,s))^2dx\right)\de s
\\
&\,\,
+\frac{\epsilon^{\nicefrac{1}{2}}}{2}\!\int_0^t\!\!e^{(-4\pi^2n^2\dot{\phi}^{\eta}(\bar\rho)+i2\pi n\dot{\nu}^{\eta}(\bar\rho))(t-s)}
\\
&\qquad\qquad
\cdot\left(\!\int_{\T^d}\!\!\!\left(\Delta e_n F_1^\epsilon+\nabla e_n2 F_2^\epsilon\right)\int_0^{\rho^{\epsilon,\eta}}\!\!\!\!\!(\dot{\sigma}^\eta(z))^2\,dz\,
\!-\!
\nabla e_n F_2^\epsilon \dot{\sigma}^\eta(\rho^{\epsilon,\eta})\sigma^\eta(\rho^{\epsilon,\eta})\,dx\right)\de s
\\
&\,\,
+\int_0^te^{(-4\pi^2n^2\dot{\phi}^{\eta}(\bar\rho)+i2\pi n\dot{\nu}^{\eta}(\bar\rho))(t-s)}\int_{\T^d}\nabla e_n(x)\big((\sigma^\eta(\rho^{\epsilon,\eta})-\sigma^\eta(\Bar{\rho}))\,\de\xi^{\epsilon} + \sigma^\eta(\Bar{\rho})\,\de(\xi^\epsilon-\xi)\big).
\end{split}
\end{align}}

We now estimate each term on the right hand side of \eqref{theorem/clt for fluctuations 1/proof 5} separately.
We first consider the case $\tau=2$.
For the first term we compute, for a constant $C=C(T,p,\phi,\sigma,\nu,\bar\rho)$,
{\small
\begin{align}\label{theorem/clt for fluctuations 1/proof 6}
\begin{split}
&E\left[\int_0^T\left|
\int_0^te^{(-4\pi^2n^2\dot{\phi}^{\eta}(\bar\rho)+i2\pi n\dot{\nu}^{\eta}(\bar\rho))(t-s)}\,\epsilon^{\nicefrac{1}{2}}\left(
\int_{\T^d}\!\!\Delta e_n \int_0^1\!\!\int_0^1\big(\ddot{\phi}^{\eta}+\ddot{\nu}^{\eta}\big)\lambda\,d\lambda d\mu\,(v^{\epsilon,\eta}(x,s))^2dx\right)\de s\right|^2\!\!dt\right]
\\
&\leq
\epsilon\, \E\left[\int_0^T\!\!\left(
\int_0^te^{(-4\pi^2n^2\dot{\phi}^{\eta}(\bar\rho)+i2\pi n\dot{\nu}^{\eta}(\bar\rho))(t-s)}ds\right)\right.
\\
&\qquad\quad\left.\cdot\left(\int_0^te^{(-4\pi^2n^2\dot{\phi}^{\eta}(\bar\rho)+i2\pi n\dot{\nu}^{\eta}(\bar\rho))(t-s)}\left(
\int_{\T^d}\!\!\Delta e_n \int_0^1\int_0^1\big(\ddot{\phi}^{\eta}+\ddot{\nu}^{\eta}\big)\lambda\,d\lambda d\mu\,(v^{\epsilon,\eta}(x,s))^2dx\right)^2\!\!ds\right)\!dt\right]
\\
&\leq
C\,\frac{\epsilon}{n^2}\E\left[\int_0^T\left(
\int_s^Te^{(-4\pi^2n^2\dot{\phi}^{\eta}(\bar\rho)+i2\pi n\dot{\nu}^{\eta}(\bar\rho))(t-s)}dt\right)\right.
\\
&\qquad\qquad\qquad\quad\cdot\left.\left(
\int_{\T^d}\!\!\Delta e_n \int_0^1\int_0^1\big(\ddot{\phi}^{\eta}+\ddot{\nu}^{\eta}\big)\lambda\,d\lambda d\mu\,(v^{\epsilon,\eta}(x,s))^2dx\right)^2\de s\right]
\\
&\leq
C\,\frac{\epsilon}{n^4}\E\left[\int_0^T
\int_{\T^d}\!\!|\Delta e_n|^2 \int_0^1\int_0^1\big(\ddot{\phi}^{\eta}(\bar\rho+\lambda\mu(\rho^{\epsilon,\eta}-\bar\rho))+\ddot{\nu}^{\eta}(\bar\rho+\lambda\mu(\rho^{\epsilon,\eta}-\bar\rho))\big)^2\lambda^2\,d\lambda d\mu\,(v^{\epsilon,\eta})^4\de x \de s\right]
\\
&\leq
C\,\epsilon\, \E\left[\int_0^T
\int_{\T^d}\!\!\left(1+|\bar\rho|+\epsilon^{\frac{1}{2}}|v^{\epsilon,\eta}|\right)^{2g}(v^{\epsilon,\eta})^4\de x \de s\right]
\\
&\leq
C\,\epsilon\, |F_3^\epsilon|_{\infty}^2\left(1+\epsilon|F_3^\epsilon|_{\infty}\right)^{g+\frac{d}{2}(p+m)}
\left(1+\E\left[\|\rho_0\|_{L^1(\T^d)}^{m+p-1}+\|\rho_0\|_{L^p(\T^d)}^p\right]\right).
\end{split}
\end{align}}
In the first passage we used H\"older's inequality and in the second passage we estimated the first time integral and used Fubini's theorem to swap remaining time integrals.
In the third passage we estimated again the time integral of the exponential and then applied H\"older's inequality several times.
The fourth passage follows from $\nabla e_n=i2\pi n e_n$ and Assumption \ref{assumption/assumption C2}(ii).
The last passage follows from H\"older's inequality, Proposition \ref{proposition/lp estimates} and $2g+4\leq \frac{p}{k+1}$ in Assumption \ref{assumption/assumption C2}(ii).

For the second term on the right hand side of \eqref{theorem/clt for fluctuations 1/proof 5} we compute, for a constant $C=C(T,p,\phi,\sigma,\nu,\bar\rho)$,
{\small
\begin{align}\label{theorem/clt for fluctuations 1/proof 7}
\begin{split}
\E&\left[\int_0^T\left(
\int_0^t\!\!e^{(-4\pi^2n^2\dot{\phi}^{\eta}(\bar\rho)+i2\pi n\dot{\nu}^{\eta}(\bar\rho))(t-s)}\right.\right.
\\
&\qquad\qquad\left.\left.\cdot\frac{\epsilon^{\nicefrac{1}{2}}}{2}\left(\int_{\T^d}\!\!\left(\Delta e_n F_1^\epsilon+\nabla e_n2 F_2^\epsilon\right)\int_0^{\rho^{\epsilon,\eta}}\!\!\!\!\!(\dot{\sigma}^\eta(z))^2\,dz\,
\!-\!
\nabla e_n F_2^\epsilon \dot{\sigma}^\eta(\rho^{\epsilon,\eta})\sigma^\eta(\rho^{\epsilon,\eta})\,dx\right)\de s\right)^2dt\right]
\\
&\leq
C\,\frac{\epsilon}{n^4}\E\left[\int_0^T
\int_{\T^d}\!\!\left(|\Delta e_n|+|\nabla e_n|\right)^2\left(|F_1^\epsilon|_\infty+|F_2^\epsilon|_\infty\right)^2
\left(\int_0^{\rho^{\epsilon,\eta}}\!\!\!\!\!(\dot{\sigma}^\eta(z))^2\,dz+|\dot{\sigma}^\eta(\rho^{\epsilon,\eta})\sigma^\eta(\rho^{\epsilon,\eta})|\right)^2dx\,ds\right]
\\
&\leq
C\,\epsilon\,\left(|F_1^\epsilon|_\infty+|F_2^\epsilon|_\infty\right)^2\E\left[\int_0^T
\int_{\T^d}\!\!1+|\rho^{\epsilon,\eta}|^{2-4\theta}+|\rho^{\epsilon,\eta}|^{4k+2}dx\,ds\right]
\\
&\leq
C\,\epsilon\,\left(|F_1^\epsilon|_\infty+|F_2^\epsilon|_\infty\right)^2\left(1+\epsilon|F_3^\epsilon|_{\infty}\right)^{\frac{d}{2}(p+m)}
\left(1+\E\left[\|\rho_0\|_{L^1(\T^d)}^{m+p-1}+\|\rho_0\|_{L^p(\T^d)}^p\right]\right).
\end{split}
\end{align}}
In the first passage we used H\"older's inequality and Fubini's theorem for the time integrals, and then estimated the integrals involving the exponential.
The second passage follows from $\nabla e_n=i2\pi n e_n$ and Assumption \ref{assumption/assumption C2}(i) on $\sigma$.
The last passage follows from H\"older's inequality, estimate \eqref{equation/estimate 1 for kinetic solutions}, and $\theta\in[0,\nicefrac{1}{2})$ and $4k+4\leq p$ in Assumption \ref{assumption/assumption C2}(i).

We finally consider the last term on the right hand side of \eqref{theorem/clt for fluctuations 1/proof 5}.
We first use It\^o isometry, then Fubini's theorem and finally estimate the time integral of the exponential and use Assumption \ref{assumption/assumption N2} to obtain, for a constant $C=C(T,\phi,\bar\rho)$,
{\small
\begin{align}\label{theorem/clt for fluctuations 1/proof 8}
\begin{split}
\E&\left[\int_0^T\left(
\int_0^te^{(-4\pi^2n^2\dot{\phi}^{\eta}(\bar\rho)+i2\pi n\dot{\nu}^{\eta}(\bar\rho))(t-s)}\int_{\T^d}\nabla e_n(x)\left((\sigma^\eta(\rho^{\epsilon,\eta})-\sigma^\eta(\Bar{\rho}))\,\de\xi^{\epsilon} + \sigma^\eta(\Bar{\rho})\,\de(\xi^\epsilon-\xi)\right)\right)^2dt\right]
\\
&=
\E\left[\int_0^T
\int_0^te^{(-8\pi^2n^2\dot{\phi}^{\eta}(\bar\rho)+i4\pi n\dot{\nu}^{\eta}(\bar\rho))(t-s)}\right.
\\
&\qquad\qquad\qquad
\left.\cdot\sum_{m}\left(\left|\int_{\T^d}\nabla e_n (\sigma^\eta(\rho^{\epsilon,\eta})-\sigma^\eta(\Bar{\rho})) f_m^\epsilon \,dx\,\right|^2 + \left|\int_{\T^d}\nabla e_n\sigma^\eta(\Bar{\rho})(f_m^\epsilon-f_m)\,dx\,\right|^2\right)ds\,dt\right]
\\
&=
\E\left[\int_0^T
\left(\int_s^Te^{(-8\pi^2n^2\dot{\phi}^{\eta}(\bar\rho)+i4\pi n\dot{\nu}^{\eta}(\bar\rho))(t-s)}\,\,4\pi^2n^2\,\,dt\right)
\right.
\\
&\qquad\qquad\qquad
\left.
\cdot\sum_{m}\left(\left|\int_{\T^d} e_n (\sigma^\eta(\rho^{\epsilon,\eta})-\sigma^\eta(\Bar{\rho})) f_m^\epsilon \,dx\,\right|^2 + \left|\int_{\T^d} e_n\sigma^\eta(\Bar{\rho})(f_m^\epsilon-f_m)\,dx\,\right|^2\right)ds\right]
\\
&\leq
C\,\E\left[\int_0^T
\int_{\T^d} |\sigma^\eta(\rho^{\epsilon,\eta})-\sigma^\eta(\Bar{\rho})|^2 dx + \sum_{m} \,\left|\int_{\T^d} e_n\sigma^\eta(\Bar{\rho})(f_m^\epsilon-f_m)\,dx\,\right|^2 ds\right].
\end{split}
\end{align}}
In turn, for the first term on the right hand side of \eqref{theorem/clt for fluctuations 1/proof 8}, for a constant $C=C(T,p,\phi,\sigma,\nu,\bar\rho)$,
{\small
\begin{align}\label{theorem/clt for fluctuations 1/proof 9}
\begin{split}
\E&\left[\int_0^T
\int_{\T^d} |\sigma^\eta(\rho^{\epsilon,\eta})-\sigma^\eta(\Bar{\rho})|^2 dx ds\right]
\\
&\leq
C
\E\left[\int_0^T
\int_{\T^d} |\rho^{\epsilon,\eta}-\bar\rho|^2dx\,ds\right]^{\nicefrac{1}{2}}
\E\left[ \int_0^T\int_{\T^d}\left(\int_{\bar\rho}^{\rho^{\epsilon,\eta}}(\dot{\sigma}^\eta(z))^2dz\right)^2dx ds\right]^{\nicefrac{1}{2}}
\\
&\leq
C\,\epsilon^{\nicefrac{1}{2}}
\E\left[\|v^{\epsilon,\eta}\|_{L^2_{t,x}}^2\right]^{\nicefrac{1}{2}}
\left(1+\E\left[ \int_0^T\int_{\T^d}|\rho^{\epsilon,\eta}|^{(2k+1)2}dx\,ds\right]\right)^{\nicefrac{1}{2}}
\\
&\leq
C\,\epsilon^{\nicefrac{1}{2}}|F_3^\epsilon|_{\infty}^{\nicefrac{1}{2}}\left(1+\epsilon|F_3^\epsilon|_\infty\right)^{\frac{d}{2}(p+m)}\left(1+\E\left[\|\rho_0\|_{L^1(\T^d)}^{m+p-1}+\|\rho_0\|_{L^p(\T^d)}^p\right]\right).
\end{split}
\end{align}}
The first passage follows from H\"older's inequality and the fundamental theorem of calculus.
In the second passage we used the definition of $v^{\epsilon,\eta}$, Assumption \ref{assumption/assumption C2}(i) on $\dot{\sigma}$ and Young's inequality.
In the last passage we used Proposition \ref{proposition/lp estimates}, H\"older's inequality with $2(2k+1)\leq p$ from Assumption \ref{assumption/assumption C2} and the $L^p$-estimate \eqref{equation/estimate 1 for kinetic solutions}.

In conclusion, combining \eqref{theorem/clt for fluctuations 1/proof 5} with the estimates \eqref{theorem/clt for fluctuations 1/proof 6}, \eqref{theorem/clt for fluctuations 1/proof 7}, \eqref{theorem/clt for fluctuations 1/proof 8} and \eqref{theorem/clt for fluctuations 1/proof 9}, we obtain, for a constant $C=C(T,\bar\rho,p,\phi,\sigma,\nu)$,
{\small
\begin{align}\label{theorem/clt for fluctuations 1/proof 10}
\begin{split}
\E\left[\|{v}^{\epsilon,\eta}_n(t)-{v}^\eta_n(t)\|_{L^2([0,T];H^{-\beta}(\T^d))}^2\right]
&=
\E\left[\int_0^T\sum_{n\in\Z^d}\left|\hat{v}^{\epsilon,\eta}_n(t)-\hat{v}^\eta_n(t)\right|n^{-2\beta}dt\right]
\\
&\leq\,\,
C\sum_{n\in\Z^d}n^{-2\beta}  \sum_{m} \int_0^T\left|\int_{\T^d} e_n\,(f_m^\epsilon-f_m)\,dx\right|^2 ds
\\
&\,\,\,\,+
C \left(\sum_{n\in\Z^d}n^{-2\beta}\right)
\left(\epsilon\left(|F_1^\epsilon|_{\infty}+|F_3^\epsilon|_{\infty}\right)^2+\epsilon^{\nicefrac{1}{2}}|F_3^\epsilon|_{\infty}^{\nicefrac{1}{2}}\right)
\\
&\qquad\quad\cdot
\left(1+\epsilon\,|F_3^\epsilon|_\infty\right)^{g+\frac{d}{2}(p+m)}
\left(1+\E\left[\|\rho_0\|_{L^1(\T^d)}^{m+p-1}+\|\rho_0\|_{L^p(\T^d)}^p\right]\right). 
\end{split}
\end{align}}
We now let $\eta\to0$, keeping $\epsilon\in(0,1)$ fixed, and use Fatou's Lemma in \eqref{theorem/clt for fluctuations 1/proof 10} to obtain formula \eqref{equation/rate of convergence 2} in the case $\tau=2$.

The estimate for $\tau=\infty$ is obtained almost identically, by estimating {\small$\E\left[\sup_{t\in[0,T]}\left|\hat{v}^{\epsilon,\eta}_n(t)-\hat{v}^\eta_n(t)\right|n^{-2\beta}\right]$} for each Fourier mode via the expression \eqref{theorem/clt for fluctuations 1/proof 5} and computations completely analogous to \eqref{theorem/clt for fluctuations 1/proof 6}-\eqref{theorem/clt for fluctuations 1/proof 8}, where we simply replace $\int_0^Tdt$ by $\sup_{t\in[0,T]}$.
The only differences are that in the first passage of \eqref{theorem/clt for fluctuations 1/proof 8} we use Doob's maximal inequality for stochastic convolutions (cf. \cite{da_prato_zabczyk_1992}) instead of the It\^o isometry, and that in the respective computations \eqref{theorem/clt for fluctuations 1/proof 6}-\eqref{theorem/clt for fluctuations 1/proof 8} we pick up a factor $n^2$ on the right hand side of each estimate since this is no more compensated by the time integral of the exponential. This forces us to require $\beta>\frac{d}{2}+1$ in this case. Then we conclude identically to \eqref{theorem/clt for fluctuations 1/proof 9}-\eqref{theorem/clt for fluctuations 1/proof 10}.

Finally we argue that $v^\epsilon\to v$ along the scaling regime \eqref{equation/scaling regime for noise}.
Ideed, since $\beta>\frac{d}{2}$, or $\beta>\frac{d}{2}+1$ respectively, the first term on the right hand side of \eqref{equation/rate of convergence 2} vanishes along the prescribed scaling regime.
The second term vanishes thanks to Assumption \ref{assumption/assumption N2} on the noise sequence and dominated convergence.
\end{proof}

\subsection{The central limit theorem in probability for rougher coefficients}

In this subsection we extend the CLT to rougher coefficients satisfying Assumption \ref{assumption/assumption C1} and Assumption \ref{assumption/assumption C2 weak} only.

In the following, for $\eta\in(0,1)$, we consider smoothed coefficients $\phi^\eta,\nu^\eta,\sigma^\eta$ satisfying further assumptions, but obtained by smoothing the original coefficients $\phi,\nu,\sigma$ \emph{only near zero}.
Namely, a standard smoothing procedure yields the following lemma.

\begin{lemma}
\label{lemma/ smooth approximations only near zero}
Consider coefficients $\phi,\nu,\sigma$ satisfying Assumption \ref{assumption/assumption C1} and \ref{assumption/assumption C2 weak}.
There exists a sequence of approximating coefficients $\{\phi^\eta,\nu^\eta,\sigma^\eta\}_{\eta\in(0,1)}$ that satisfies Assumption \ref{assumption/assumption C1} and \ref{assumption/assumption C2 weak} uniformly in $\eta\in(0,1)$, and Assumption \ref{assumption/assumption C2} and Assumption (i)-(ii) in Theorem \ref{theorem/existence and uniqueness of weak solution to dean_kawasaki}, not necessarily uniformly in $\eta$, and such that, defining
\begin{equation}
    \label{equation/delta_eta}
        \delta_\eta:=\inf\{\delta\geq0\mid \phi^\eta(z)=\phi(z),\,\nu^\eta(z)=\nu(z),\,\sigma^\eta(z)=\sigma(z)\,\forall z\in[\delta,\infty)\},
    \end{equation}
we have that $\delta_\eta$ decreases to zero as $\eta$ decreases to zero.
That is $\lim_{\eta\to0}\delta_{\eta}=0$.
\end{lemma}

Given the stochastic kinetic solution $\rho^{\epsilon,\eta}$ to equation \eqref{equation/dean--kawasaki with smoothed coefficients 2} with these smoothed coefficients$\phi^\eta$, $\nu^\eta$, $\sigma^\eta$, we denote
\begin{equation}
    \label{formula/ set omega epsilone eta}
    \Omega^{\epsilon,\eta}:=\Big\{\omega\in\Omega\,\big| \,\,\underset{\T^d\times[0,T]}{\text{ess-inf }}\rho^{\epsilon,\eta}>\delta_{\eta}\Big\}.
\end{equation}
Furthermore, given any initial data $\rho_0$ satisfying Assumption \ref{assumption/assumption I1}(ii), we denote
\begin{equation}
        \low:=\underset{\Omega\times\T^d}{\text{ess-inf }}\rho_0.
\end{equation}

The following two results establish stronger uniqueness both for the zero-noise limit and for the generalized Dean--Kawasaki equation. 
The first is an immediate application of the maximum principle and the uniqueness for the deterministic equation \eqref{equation/hydrodynamic limit equation 1}. 

\begin{lemma}
\label{lemma/hydrodinamic limit is unchanged by smoothing}
Let $\bar\rho$ and $\bar\rho^\eta$ be the solutions of equation \eqref{equation/hydrodynamic limit equation 1} and of its smoothed version $\partial_t\Bar{\rho}^\eta=\Delta\phi^{\eta}(\Bar{\rho}^\eta)-\nabla\!\cdot\!\nu^\eta(\bar{\rho}^\eta)$, both with initial data $\rho_0$ satisfying Assumption \ref{assumption/assumption I1}(ii), i.e. a random positive constant.
Let $\eta\in(0,1)$ be small enough so that $\delta_\eta<\low$, that is so that the smoothed coefficients match the true coefficients on $[\low,\infty)$.
Then we have $\bar\rho=\bar\rho^\eta$ for a.e. $(\omega,x,t)\in\Omega\times\T^d\times[0,\infty)$.
\end{lemma}

A straightforward adaptation of the uniqueness proof in \cite[Theorem 4.7]{fehrman-gess-Well-posedness-of-the-Dean-Kawasaki-and-the-nonlinear-Dawson-Watanabe-equation-with-correlated-noise}, which just amounts to restricting all the arguments to a smaller probability subset $\Omega_0\subseteq\Omega$, establishes the following.

\begin{lemma}[Enhanced pathwise uniqueness]
\label{proposition/enhanced pathwise uniqueness}
For $i=1,2$, let $\rho^i$ be the stochastic kinetic solution to the equation
\begin{equation}
    \label{formula/ generalized dean kawasaki equation couple}
    \partial_t\rho^i=\Delta\phi_i(\rho^i)-\nabla\cdot\nu_i(\rho^i)-\sqrt{\epsilon}\nabla\cdot\left(\sigma_i(\rho^i)\circ \dot{\xi}^{\epsilon}\right)\,\,\,\text{in } \mathbb{T}^d\times(0,T),\quad \rho^i(\cdot,0)=\rho_0^i\,\,\,\text{in } \mathbb{T}^d\times\{0\},
\end{equation}
with noise $\xi^{\epsilon}$ satisfying Assumption \ref{assumption/assumption N1}, coefficients $\phi_i$, $\nu_i$ and $\sigma_i$ satisfying Assumption \ref{assumption/assumption C1}, and initial data $\rho^i_0$ satisfying Assumption \ref{assumption/assumption I1}.
Let $\Omega_0\subseteq \Omega$ be a measurable subset such that
\begin{itemize}
    \item[i)] $\rho_0^1(x,\omega)=\rho_0^2(x,\omega)$ for every $x\in \T^d$, for a.e. $\omega\in\Omega_0$;
    
    \item[ii)] for some $\delta>0$, for a.e. $\omega\in\Omega_0$, 
        {\small
        \begin{equation}
            \phi_1(z)=\phi_2(z),\,\,\nu_1(z)=\nu_2(z),\,\,\sigma_1(z)=\sigma_2(z)\,\,\,\,\forall z\in\bigg(\essinf_{\T^d\times[0,T]\times\Omega_0}\!\!\rho^1-\delta,\,\,\underset{\T^d\times[0,T]\times\Omega_0}{\esssup}\!\!\rho^1+\delta\bigg).
        \end{equation}}
\end{itemize}
Then we have that
\begin{equation}\label{formula/ enhanced pathwise uniqueness}
    \sup_{t\in[0,T]}\|\rho^1(\cdot,t)-\rho^2(\cdot,t)\|_{L^1(\T^d)}\leq \|\rho_0^1-\rho_0^2\|_{L^1(\T^d)}\quad\text{almost surely on $\Omega_0$.}
\end{equation}
\end{lemma}

The two previous Lemmas, the convergence \eqref{equation/small noise limit} from \cite[Theorem 6.6]{fehrman-gess-large-deviations-for-conservative} and Theorem \ref{theorem/clt for fluctuations 1} from the previous section immediately yield the following.

\begin{corollary}
\label{corollary/clt for the smoothed solutions}
Let $\rho_0$, $\left(\xi^\epsilon\right)_{\epsilon>0}$ and $\phi,\nu,\sigma$ satisfy Assumption \ref{assumption/assumption I1}(ii), Assumption \ref{assumption/assumption N2}, and Assumption \ref{assumption/assumption C1} and \ref{assumption/assumption C2 weak} respectively.
Consider the smoothed coefficients from Lemma \ref{lemma/ smooth approximations only near zero} and let $\eta\in(0,1)$ be small enough so that $\delta_\eta<\low$.
Let $\rho^{\epsilon,\eta}$ be the stochastic kinetic solution to the smoothed equation \eqref{equation/dean--kawasaki with smoothed coefficients 2}.
We have, keeping $\eta$ fixed,
\begin{equation}
    \lim_{\epsilon\to0}\|\rho^{\epsilon,\eta}-\bar{\rho}\|_{L^1([0,T];L^1(\T^d))}=0\,\,\,\text{in probability}.
\end{equation}
Furthermore, keeping $\eta$ fixed and letting $\epsilon\to0$ along the scaling regime \eqref{equation/scaling regime for noise}, we have
\begin{equation}
   v^{\epsilon,\eta}=\epsilon^{-\nicefrac{1}{2}}(\rho^{\epsilon,\eta}-\Bar{\rho})\to v \quad \text{in } L^2(\Omega;L^\tau_{\text{loc}}([0,\infty);H^{-\beta}(\T^d))),
\end{equation}
for $\tau=2$ or $\tau=\infty$, for any $\beta>\frac{d}{2}$ or $\beta>1+\frac{d}{2}$ respectively, with rate of convergence given in formula \eqref{equation/rate of convergence 2}.
Finally, for every $\epsilon\in(0,1)$ and every $\eta\in(0,1)$ small enough, we have
\begin{equation}
    \label{equation/solutions coincide on Omega^(epsilon,eta)}
    \rho^\epsilon=\rho^{\epsilon,\eta}\quad \text{and}\quad v^{\epsilon}=v^{\epsilon,\eta} \quad\text{ for a.e. $(x,t)\in\T^d\times[0,T]$, for every $\omega\in\Omega^{\epsilon,\eta}$}.
\end{equation}
\end{corollary}

\medskip

We now present the main ingredient to establish the central limit theorem in probability.
Corollary \ref{corollary/clt for the smoothed solutions} implies that the sets $\Omega^{\epsilon,\eta}$ satisfy $\Omega^{\epsilon,\eta_1}\supseteq\Omega^{\epsilon,\eta_2}$ for every $\eta_1<\eta_2\in(0,1)$ and every $\epsilon\in(0,1)$. 
The following proposition essentially establishes that the measure of $\Omega^{\epsilon,\eta}\subseteq\Omega$ increases to $1$ as $\epsilon\to0$.

The proposition is adapted from \cite{dirr-fehrman-gess-conservative-stochastic-pde-and-fluctuations-of-the-symmetric} to the nonlinear diffusion case and we state it on its own, since the result is interesting per se and it holds for any stochastic kinetic solution of equation \eqref{equation/generalized dean--kawasaki equation 2} provided the coefficients satisfy Assumption \ref{assumption/assumption C1} and the diffusion is nondegenerate.

\begin{proposition}[Moser iteration argument]
\label{proposition/moser iteration}
Fix $\epsilon>0$.
Let $\rho_0$ satisfy Assumption \ref{assumption/assumption I1}, let $\xi^\epsilon$ satisfy Assumption \ref{assumption/assumption N1} and let $\phi,\nu,\sigma$ satisfy Assumption \ref{assumption/assumption C1}.
Furthermore, suppose that $\inf \dot{\phi}>0$.
Let $\rho^\epsilon$ be the stochastic kinetic solution to \eqref{equation/generalized dean--kawasaki equation 3} and let  $\low:=\underset{\Omega\times\T^d}{\text{ess-inf }}\rho_0$.
For every $T\geq0$, we have
\begin{align}\label{formula/ moser iteration bound}
\begin{aligned}
    \E\left[\norm{(\rho^\epsilon-\low)_-}_{L^\infty(\TT^d\times[0,T])}\right]\leq c^*
    \sum_{j=1}^{\infty}\frac{1}{j^2}
    \, 
    R_{\epsilon}^{\,\,\zeta\,(1+\nicefrac{2}{d})^{-j}},
\end{aligned}
\end{align}
where $c^*$ is a \emph{numeric} constant, and $R_{\epsilon}$ and the exponent $\zeta$ are given by
\begin{equation}
\label{formula/ reminder R epsilon}
    R_{\epsilon}:=C_0\,(\inf\dot{\phi})^{-2}\big(\epsilon^2\,\|F_1^{\epsilon}\|_{\infty}\|F_3^\epsilon\|_{\infty}+\epsilon\,\|F_3^\epsilon\|_{\infty}\big),\quad
    \zeta:=\begin{cases}\frac{d}{2}(1+\nicefrac{2}{d})^2\,\qquad\quad\,\,\,\,\text{if $R_\epsilon\geq1$,}\\
e^{-d(d+1)}(\nicefrac{1}{2}+\nicefrac{1}{d})\quad\text{if $R_\epsilon\in(0,1)$,}\end{cases}
\end{equation}
for a constant $C_0=C_0(T,d,\low,\phi,\sigma,\nu)$ depending on the coefficients $\phi,\nu,\sigma$ only through the constants $c$ featuring in Assumption \ref{assumption/assumption C1}.
\end{proposition}
\begin{proof}
For $\eta\in(0,1)$ consider smoothed coefficients $\phi^\eta,\nu^\eta,\sigma^\eta$ obtained from the original coefficients $\phi,\nu,\sigma$ thanks to Lemma \ref{lemma/smooth approximations phi^eta sigma^eta}.
They satisfy Assumption \ref{assumption/assumption C1} and also Assumption (i) in Theorem \ref{theorem/existence and uniqueness of weak solution to dean_kawasaki}, since $\phi$ is nondegenerate, uniformly in $\eta$.
In particular we require $\inf\dot{\phi}^\eta\geq\inf\dot\phi>0$.
Let $\roeta$ be the weak solution to equation \eqref{equation/dean--kawasaki with smoothed coefficients 2} with coefficients $\phi^\eta,\nu^\eta,\sigma^\eta$ and the same noise $\xi^\epsilon$ and initial data $\rho_0$. Finally let $\low:=\underset{\Omega\times\T^d}{\text{ess-inf }}\rho_0$.

For any $\alpha\in[1,\infty)$, consider the function $f(z):=(z-\low)^{\alpha+1}_-=(0\vee(\low-z))^{\alpha+1}$.
Applying It\^o's formula in the form proved in Krylov \cite{krylov_ito_formula} yields
\begin{align}\label{proof/ moser 1}
    \begin{aligned}
        \int_{\T^d}(\rho^{\epsilon,\eta}_r-\low)^{\alpha+1}_- dx\Big|_{r=0}^{r=t}
        =
        &-\int_0^t\int_{\T^d}\alpha(\alpha+1)(\rho^{\epsilon,\eta}_r-\low)^{\alpha-1}_-\dot{\phi}^\eta(\rho^{\epsilon,\eta})|\nabla\rho_r^{\epsilon,\eta}|^2\,dx\,dr
        \\
        &-\int_0^t\int_{\T^d}(\alpha+1)(\rho^{\epsilon,\eta}_r-\low)^{\alpha}_-\nabla\cdot\left(\nu^\eta(\rho^{\epsilon,\eta})\right)\,dx\,dr
        \\
        &-\epsilon^{\nicefrac{1}{2}}\int_0^t\int_{\T^d}(\alpha+1)(\rho^{\epsilon,\eta}_r-\low)^{\alpha}_-\,\nabla\cdot\left(\sigma^\eta(\rho^{\epsilon,\eta})\,\de \xi^\epsilon\right)
        \\
        &+\frac{\epsilon}{2}\int_0^t\int_{\T^d}\!\!\!\!\alpha(\alpha+1)(\rho^{\epsilon,\eta}_r-\low)^{\alpha-1}_-
        \sigma^\eta(\rho^{\epsilon,\eta})\nabla\sigma^\eta(\rho^{\epsilon,\eta}) F_2^\epsilon\,dx\,dr
        \\
        &+\frac{\epsilon}{2}\int_0^t\int_{\T^d}\!\!\!\!\alpha(\alpha+1)(\rho^{\epsilon,\eta}_r-\low)^{\alpha-1}_-
        \left(\sigma^\eta(\rho^{\epsilon,\eta})\right)^2F_3^\epsilon \,dx\,dr.
    \end{aligned}
\end{align}
We note that that the integral at time $r=0$ on the left hand side of \eqref{proof/ moser 1} vanishes because $(\rho_0-\low)_-\equiv0$ by definition of $\low$.
Furthermore, arguing as in \eqref{proposition/lp estimates/ proof 3}-\eqref{proposition/lp estimates/proof 4}, using integration by parts and the periodic boundary conditions or the assumption $\nabla\cdot F_2^\epsilon=0$, shows that the second and the fourth term on the right hand side of \eqref{proof/ moser 1} are identically zero.
We integrate by parts the third term on the right hand side of \eqref{proof/ moser 1}, then we use the identity
\[\nabla\,(\rho^{\epsilon,\eta})^{\frac{\alpha+1}{2}}={\small\frac{\alpha+1}{2}}\,(\rho^{\epsilon,\eta})^{\frac{\alpha-1}{2}}\,\nabla\rho^{\epsilon,\eta}\]
to rewrite the first and the third term on the right hand side, and finally we rearrange the terms to obtain
\begin{align}\label{proof/ moser 5}
    \begin{aligned}
        \int_{\T^d}(\rho^{\epsilon,\eta}_t&-\low)^{\alpha+1}_- dx
        +
        \int_0^t\int_{\T^d}{\small\frac{4\alpha}{\alpha+1}}\dot{\phi}^\eta(\rho^{\epsilon,\eta}_r)\left|\nabla(\rho_r^{\epsilon,\eta}-\low)_-^{\frac{\alpha+1}{2}}\right|^2\,dx\,dr
        \\
        =
        &
        \,\,\,\epsilon^{\nicefrac{1}{2}}\int_0^t\int_{\T^d}2\alpha\,(\rho^{\epsilon,\eta}_r-\low)^{\frac{\alpha-1}{2}}_-\,\nabla(\rho_r^{\epsilon,\eta}-\low)^{\frac{\alpha+1}{2}}_-\sigma^\eta(\rho^{\epsilon,\eta})\,\de \xi^\epsilon
        \\
        &+\frac{\epsilon}{2}\int_0^t\int_{\T^d}\!\!\!\!\alpha(\alpha+1)(\rho^{\epsilon,\eta}_r-\low)^{\alpha-1}_-
        \left(\sigma^\eta(\rho^{\epsilon,\eta})\right)^2F_3^\epsilon \,dx\,dr.
    \end{aligned}
\end{align}

Since $\rho^{\epsilon,\eta}\geq0$ always, we have $|(\rho^{\epsilon,\eta}-\low)_-|_{\infty}\leq \low$.
Moreover, the assumption $\sigma\in C_{\text{loc}}([0,\infty))$ implies that there exists $C=C(\low,\sigma)$ independent of $\eta\in(0,1)$ such that
\begin{equation}
\label{proof/ moser 6}
    |\sigma^{\eta}(\rho^{\epsilon,\eta})\,\mathbf{1}_{\{(\rho^{\epsilon,\eta}-\low)_-\neq0\}}|\leq C.
\end{equation}
Hence, for the second term on the right hand side of \eqref{proof/ moser 5}, we have, for a constant $C=C(\low,\sigma)$ independent of $\eta,\epsilon\in(0,1)$ and of $\alpha\geq1$, almost surely for every $t\in[0,T]$,
\begin{align}\label{proof/ moser 7}
    \begin{aligned}
        \frac{\epsilon}{2}\int_0^t\int_{\T^d}\!\!\!\!\alpha(\alpha+1)(\rho^{\epsilon,\eta}_r-\low)^{\alpha-1}_-
        \left(\sigma^\eta(\rho^{\epsilon,\eta})\right)^2F_3^\epsilon \,dx\,dr
        \leq
        C\,\alpha^2\,\epsilon\,\|F_3^\epsilon\|_{\infty}\int_0^t\int_{\T^d}(\rho^{\epsilon,\eta}_r-\low)^{\alpha-1}_-\,dx\,dr.
    \end{aligned}
\end{align}
Given an arbitrary bounded stopping time $\tau\leq T$, we now evaluate \eqref{proof/ moser 5} at $t=\tau$.
After taking the expectation, so that the martingale term vanishes, using estimate \eqref{proof/ moser 7} and multiplying both sides by $(\inf\dot\phi)^{-1}$, we obtain, for a constant $C=C(\low,\sigma)$ independent of $\eta,\epsilon\in(0,1)$ and $\alpha\geq1$,
\begin{align}\label{proof/ moser 9}
    \begin{aligned}
        \E\bigg[
        \int_0^{\tau}\int_{\T^d}\left|\nabla(\rho_r^{\epsilon,\eta}-\low)^{\frac{\alpha+1}{2}}\right|^2\,dx\,dr
        \bigg]
        \leq
        C\,\alpha^2\,(\inf\dot\phi)^{-1}\,\epsilon\,\|F_3^\epsilon\|_{\infty}\E\bigg[\int_0^{\tau}\int_{\T^d}(\rho^{\epsilon,\eta}_r-\low)^{\alpha-1}_-\,dx\,dr\bigg].
    \end{aligned}
\end{align}

For the martingale term in \eqref{proof/ moser 5}, given an arbitrary bounded stopping time $\tau\leq T$, we compute, for a constant $C=C(\low,T,\sigma)$ independent of $\eta,\epsilon\in(0,1)$ and $\alpha\geq1$, for any $\delta_{\alpha}\in(0,\low)$,
{\small
\begin{align}\label{proof/ moser 10}
    \begin{aligned}
        \E\bigg[&\sup_{t\in[0,\tau]}\bigg|\epsilon^{\nicefrac{1}{2}}\int_0^t\int_{\T^d}2\alpha\,(\rho^{\epsilon,\eta}_r-\low)^{\frac{\alpha-1}{2}}_-\,\nabla(\rho_r^{\epsilon,\eta}-\low)^{\frac{\alpha+1}{2}}_-\sigma^\eta(\rho^{\epsilon,\eta})\,\de \xi^\epsilon\bigg|\bigg]
        \\
        \leq&\,
        C\epsilon^{\nicefrac{1}{2}}\alpha\,
        \E\left[\bigg(\int_0^{\tau}\scalebox{0.85}{$\displaystyle\sum_{k=1}^\infty$}\Big(\int_{\T^d}2(\rho^{\epsilon,\eta}_r-\low)^{\frac{\alpha-1}{2}}_-\,\nabla(\rho_r^{\epsilon,\eta}-\low)^{\frac{\alpha+1}{2}}_-\sigma^\eta(\rho^{\epsilon,\eta})\,f_k(x)\,dx\Big)^2\,dr\bigg)^{\nicefrac{1}{2}}\right]
        \\
        \leq&\,
        C\epsilon^{\nicefrac{1}{2}}\alpha\,
        \E\left[\bigg(\int_0^{\tau}\Big(\int_{\T^d}(\rho^{\epsilon,\eta}_r-\low)^{\alpha-1}_-\sigma^\eta(\rho^{\epsilon,\eta})^2F_1^{\epsilon}\,dx\Big)\,\Big(\int_{\T^d}\Big|\nabla(\rho_r^{\epsilon,\eta}-\low)^{\frac{\alpha+1}{2}}_-\Big|^2dx\Big)\,dr\bigg)^{\nicefrac{1}{2}}\right]
        \\
        \leq&\,
        C\alpha\,\epsilon^{\nicefrac{1}{2}}\|F_1^{\epsilon}\|_{\infty}^{\nicefrac{1}{2}}
        \E\Bigg[
        \sup_{t\in[0,\tau]}\bigg(\int_{\T^d}(\rho^{\epsilon,\eta}_t-\low)^{\alpha-1}_-\mathbf{1}_{\{\rho^{\epsilon,\eta}>\low-\delta_{\alpha}\}}dx
        +
        \int_{\T^d}(\rho^{\epsilon,\eta}_t-\low)^{\alpha-1}_-\mathbf{1}_{\{\rho^{\epsilon,\eta}\leq\low-\delta_{\alpha}\}}dx\bigg)^{\nicefrac{1}{2}}
        \\
        &\qquad\qquad\qquad\qquad\qquad\qquad\qquad\qquad\qquad\qquad\qquad\qquad\cdot
        \bigg(\int_0^{\tau}\int_{\T^d}\Big|\nabla(\rho_r^{\epsilon,\eta}-\low)^{\frac{\alpha+1}{2}}_-\Big|^2dx\,dr\bigg)^{\nicefrac{1}{2}}\Bigg]
        \\
        \leq&\,
        (C\alpha^2\,\epsilon\|F_1^{\epsilon}\|_{\infty})^{\nicefrac{1}{2}}
        \bigg(\delta_{\alpha}^{\alpha-1}+\delta_{\alpha}^{-2}\,\E\Big[
        \sup_{t\in[0,\tau]}\int_{\T^d}(\rho^{\epsilon,\eta}_t-\low)^{\alpha+1}_-dx\Big]\bigg)^{\nicefrac{1}{2}}
        \E\bigg[\int_0^{\tau}\int_{\T^d}\Big|\nabla(\rho_r^{\epsilon,\eta}-\low)^{\frac{\alpha+1}{2}}_-\Big|^2dx\,dr\bigg]^{\nicefrac{1}{2}}
        \\
        \leq&\,
        (C\alpha^2\,\epsilon\|F_1^{\epsilon}\|_{\infty})^{\nicefrac{1}{2}}
        \delta_{\alpha}^{\frac{\alpha-1}{2}}
        \E\Big[\int_0^{\tau}\int_{\T^d}\Big|\nabla(\rho_r^{\epsilon,\eta}-\low)^{\frac{\alpha+1}{2}}_-\Big|^2dx\,dr\Big]^{\nicefrac{1}{2}}
        \\
        &+
        (C\alpha^2\,\epsilon\|F_1^{\epsilon}\|_{\infty})^{\nicefrac{1}{2}}
        \delta_{\alpha}^{-1}
        \E\Big[
        \sup_{t\in[0,\tau]}\int_{\T^d}(\rho^{\epsilon,\eta}_t-\low)^{\alpha+1}_-dx\Big]^{\nicefrac{1}{2}}
        \E\Big[\int_0^{\tau}\int_{\T^d}\Big|\nabla(\rho_r^{\epsilon,\eta}-\low)^{\frac{\alpha+1}{2}}_-\Big|^2dx\,dr\Big]^{\nicefrac{1}{2}}
        \\
        \leq&\,
        (C\alpha^2\,\epsilon\|F_1^{\epsilon}\|_{\infty})^{\nicefrac{1}{2}}
        \bigg(\frac{1}{2}
        \delta_{\alpha}^{\alpha-1}+\frac{1}{2}
        \E\Big[\int_0^{\tau}\int_{\T^d}\Big|\nabla(\rho_r^{\epsilon,\eta}-\low)^{\frac{\alpha+1}{2}}_-\Big|^2dx\,dr\Big]\bigg)
        \\
        &+
        \frac{1}{2}\E\Big[
        \sup_{t\in[0,\tau]}\int_{\T^d}(\rho^{\epsilon,\eta}_t-\low)^{\alpha+1}_-dx\Big]
        +
        \frac{1}{2}
        (C\alpha^2\,\epsilon\|F_1^{\epsilon}\|_{\infty})
        \delta_{\alpha}^{-2}\,
        \E\Big[\int_0^{\tau}\int_{\T^d}\Big|\nabla(\rho_r^{\epsilon,\eta}-\low)^{\frac{\alpha+1}{2}}_-\Big|^2dx\,dr\Big]
        \\
        \leq&\,
        \frac{1}{2}\E\bigg[
        \sup_{t\in[0,\tau]}\int_{\T^d}(\rho^{\epsilon,\eta}_t-\low)^{\alpha+1}_-dx\bigg]
        \\
        &+
        (C\alpha^2\epsilon\|F_1^{\epsilon}\|_{\infty})^{\nicefrac{1}{2}}\,
        \delta_{\alpha}^{\alpha-1}
        \\
        &+
        \Big(C\alpha^4\,\delta_{\alpha}^{-2}(\inf\dot{\phi})^{-1}\big(\epsilon^2\|F_1^{\epsilon}\|_{\infty}+\epsilon^{\nicefrac{3}{2}}\|F_1^{\epsilon}\|_{\infty}^{\nicefrac{1}{2}}\big)\|F_3^\epsilon\|_{\infty}\Big)
        \E\bigg[\int_0^{\tau}\int_{\T^d}(\rho^{\epsilon,\eta}_r-\low)^{\alpha-1}_-dx\,dr\bigg].
    \end{aligned}
\end{align}
}
In the first passage we used the Burkholder-Davis-Gundy inequality and the form of the noise term.
The second passage follows from H\"older's inequality.
In the third passage we used \eqref{proof/ moser 6} and we took the supremum in time and then split the first space integral according to the values of $\rho^{\epsilon,\eta}$.
The fourth passage follows again from H\"older's inequality and the splitting introduced.
The fifth and sixth passage follows from convexity and H\"older's inequality.
The last passage follows from \eqref{proof/ moser 9}.

Now, given an arbitrary bounded stopping time $\tau\leq T$, combining \eqref{proof/ moser 5} with \eqref{proof/ moser 7} and \eqref{proof/ moser 10} yields, for a constant $C=C(\low,T,\sigma)$ independent of $\eta,\epsilon\in(0,1)$ and $\alpha\geq1$, for any $\delta_{\alpha}\in(0,\low)$,
{\small
\begin{align}\label{proof/ moser 11}
    \begin{aligned}
        \E\bigg[&\sup_{t\in[0,\tau]}\int_{\T^d}(\rho^{\epsilon,\eta}_t-\low)^{\alpha+1}_- dx
        +
        \int_0^{\tau}\int_{\T^d}{\small\frac{4\alpha}{\alpha+1}}\dot{\phi}^\eta(\rho^{\epsilon,\eta}_r)\left|\nabla(\rho_r^{\epsilon,\eta}-\low)^{\frac{\alpha+1}{2}}\right|^2\,dx\,dr\bigg]
        \\
        &\leq
        \,\,
        \E\bigg[\sup_{t\in[0,\tau]}\bigg|\epsilon^{\nicefrac{1}{2}}\int_0^t\int_{\T^d}2\alpha\,(\rho^{\epsilon,\eta}_r-\low)^{\frac{\alpha-1}{2}}_-\,\nabla(\rho_r^{\epsilon,\eta}-\low)^{\frac{\alpha+1}{2}}_-\sigma^\eta(\rho^{\epsilon,\eta})\,\de \xi^\epsilon\bigg|\bigg]
        \\
        &\,\,\,+
        \E\bigg[\frac{\epsilon}{2}\int_0^{\tau}\int_{\T^d}\!\!\!\!\alpha(\alpha+1)(\rho^{\epsilon,\eta}_r-\low)^{\alpha-1}_-
        \left(\sigma^\eta(\rho^{\epsilon,\eta})\right)^2F_3^\epsilon \,dx\,dr\bigg]
        \\
        &\leq\,
        \frac{1}{2}\E\bigg[
        \sup_{t\in[0,\tau]}\int_{\T^d}(\rho^{\epsilon,\eta}_t-\low)^{\alpha+1}_-dx\bigg]
        \\
        &\,\,\,+
        \Big(C\alpha^4\,\delta_{\alpha}^{-2}(\inf\dot{\phi})^{-1}\big(\epsilon^2\,\|F_1^{\epsilon}\|_{\infty}\|F_3^\epsilon\|_{\infty}+\epsilon\,\|F_3^\epsilon\|_{\infty}\big)\Big)\bigg(\delta_{\alpha}^{\alpha-1}\,+\,
        \E\bigg[\int_0^{\tau}\int_{\T^d}(\rho^{\epsilon,\eta}_r-\low)^{\alpha-1}_-dx\,dr\bigg]\bigg).
    \end{aligned}
\end{align}
}
In turn, absorbing the first term on the right hand side of \eqref{proof/ moser 11} into the left hand side, thanks to the factor $\frac{1}{2}$ in front, and then multiplying both sides by $\max{\{\inf(\dot{\phi})^{-1},1\}}$ yields, for a \emph{fixed} constant $\bar{c}=\bar{c}(T,\sigma,\low)$ independent of $\eta,\epsilon\in(0,1)$ and $\alpha\geq1$,
{\small
\begin{align}\label{proof/ moser 12}
    \begin{aligned}
        \E\bigg[&\sup_{t\in[0,\tau]}\int_{\T^d}(\rho^{\epsilon,\eta}_t-\low)^{\alpha+1}_- dx
        +
        \int_0^{\tau}\int_{\T^d}\left|\nabla(\rho_r^{\epsilon,\eta}-\low)^{\frac{\alpha+1}{2}}\right|^2\,dx\,dr\bigg]
        \\
        &\leq
        \Big(\bar{c}\,\alpha^4\,\delta_{\alpha}^{-2}(\inf\dot{\phi})^{-2}\big(\epsilon^2\,\|F_1^{\epsilon}\|_{\infty}\|F_3^\epsilon\|_{\infty}+\epsilon\,\|F_3^\epsilon\|_{\infty}\big)\Big)\bigg(\delta_{\alpha}^{\alpha-1}\,+\,
        \E\bigg[\int_0^{\tau}\int_{\T^d}(\rho^{\epsilon,\eta}_r-\low)^{\alpha-1}_-dx\,dr\bigg]\bigg).
    \end{aligned}
\end{align}
}
In particular we notice that, upon possibly enlarging the constant $\bar{c}$ by a factor $(1+T)$, when $\alpha=1$ formula \eqref{proof/ moser 12} reduces to 
{\small
\begin{align}\label{proof/ moser 12bis}
    \begin{aligned}
        \E\bigg[\!\sup_{t\in[0,T]}\int_{\T^d}\!\!\!(\rho^{\epsilon,\eta}_t-\low)^{2}_- dx
        \!+\!
        \int_0^{\tau}\!\int_{\T^d}\!\!\!\left|\nabla(\rho_r^{\epsilon,\eta}-\low)\right|^2\,dx\,dr\bigg]
        \leq
        \Big(\bar{c}\,\alpha^4\,\delta_{\alpha}^{-2}(\inf\dot{\phi})^{-2}\big(\epsilon^2\,\|F_1^{\epsilon}\|_{\infty}\|F_3^\epsilon\|_{\infty}+\epsilon\,\|F_3^\epsilon\|_{\infty}\big)\Big).
    \end{aligned}
\end{align}
}

Since the stopping time $\tau$ in \eqref{proof/ moser 12} is arbitrary, it follows from \cite[Chapter IV, Proposition 4.7 and Exercise 4.30]{revuz_yor_book} that, for any $\alpha\geq 1$, for the same constant $\bar{c}=\bar{c}(T,\sigma,\low)$ introduced in \eqref{proof/ moser 12} and independent of $\eta,\epsilon\in(0,1)$ and of $\alpha\geq1$, for any $\delta_{\alpha}\in(0,1)$,
{\small
\begin{align}\label{proof/ moser 13}
    \begin{aligned}
        \E\Bigg[&\bigg(\sup_{t\in[0,T]}\int_{\T^d}(\rho^{\epsilon,\eta}_t-\low)^{\alpha+1}_- dx
        +
        \int_0^{T}\int_{\T^d}\left|\nabla(\rho_r^{\epsilon,\eta}-\low)^{\frac{\alpha+1}{2}}\right|^2\,dx\,dr\bigg)^{\frac{1}{\alpha+1}}\Bigg]
        \\
        &\leq
        \frac{(\alpha+1)^{\frac{1}{\alpha+1}}}{1-\frac{1}{\alpha+1}}\, \big(\alpha^4\delta_{\alpha}^{-2}\big)^{\frac{1}{\alpha+1}}
        \Big(\bar{c}\,\,(\inf\dot{\phi})^{-2}\big(\epsilon^2\,\|F_1^{\epsilon}\|_{\infty}\|F_3^\epsilon\|_{\infty}+\epsilon\,\|F_3^\epsilon\|_{\infty}\big)\Big)^{\frac{1}{\alpha+1}}
        \\
        &
        \qquad\qquad\qquad\qquad\qquad\qquad
        \cdot\E\Bigg[\bigg(\delta_{\alpha}^{\alpha-1}\,+\,
        \int_0^{T}\int_{\T^d}(\rho^{\epsilon,\eta}_r-\low)^{\alpha-1}_-dx\,dr\bigg)^{\frac{1}{\alpha+1}}\Bigg].
    \end{aligned}
\end{align}
}

We are now ready to conclude the proof using a Moser iteration argument.
To ease the notation, let us define $\psi:=(\rho^{\epsilon,\eta}-\low)_-$.
Using that $(x+y)^{\nicefrac{1}{h}}\leq x^{\nicefrac{1}{h}}+y^{\nicefrac{1}{h}}$, for $h\geq 1$ and $x,y\geq0$, and H\"older's inequality with exponent $\frac{\alpha+1}{\alpha-1}$ in \eqref{proof/ moser 13} gives, for the constant $\bar{c}$ introduced in \eqref{proof/ moser 12}, for any $\alpha\geq1$ and any $\delta_{\alpha}\in(0,1)$,
{\small
\begin{align}\label{proof/ moser 14}
    \begin{aligned}
        \E&\bigg[\Big(\|\psi\|_{L_t^{\infty}L_x^{\alpha+1}}^{\alpha+1}+\|\nabla\psi^{\frac{\alpha+1}{2}}\|_{L_t^2L_x^2}^2\Big)^{\frac{1}{\alpha+1}}\bigg]
        \\
        &\leq
        \frac{(\alpha+1)^{\frac{1}{\alpha+1}}}{1-\frac{1}{\alpha+1}}\, \big(\alpha^4\delta_{\alpha}^{-2}\big)^{\frac{1}{\alpha+1}}
        \Big(\bar{c}\,(\inf\dot{\phi})^{-2}\big(\epsilon^2\,\|F_1^{\epsilon}\|_{\infty}\|F_3^\epsilon\|_{\infty}+\epsilon\,\|F_3^\epsilon\|_{\infty}\big)\Big)^{\frac{1}{\alpha+1}}
        \left(\delta_{\alpha}^{\frac{\alpha-1}{\alpha+1}}\!+\!\E\big[\|\psi\|_{L_t^{\!\alpha-1}\!L_x^{\!\alpha-1}}\big]^{\frac{\alpha-1}{\alpha+1}}\right).
    \end{aligned}
\end{align}
}
We first use interpolation and Sobolev inequalities to deduce that for
\[\lambda=\frac{2}{2+d}\;\;\textrm{and}\;\;q = \frac{(2+d)(\a+1)}{d},\]
we have that, for a \emph{fixed} constant $\Tilde{c}=\Tilde{c}(T,d)$ independent of $\eta,\epsilon\in(0,1)$ and $\alpha\geq1$,
\begin{align}\label{proof/ moser 15}
\begin{aligned}
& \norm{\psi}_{L^q(\TT^d\times[0,T])} \leq \norm{\psi}^\lambda_{L^\infty([0,T];L^{\a+1}(\TT^d))}\norm{\psi}^{1-\lambda}_{L^{\a+1}([0,T];L^{\frac{2_*}{2}(\a+1)}(\TT^d))}
\\ & = \norm{\psi}^\lambda_{L^\infty([0,T];L^{\a+1}(\TT^d))}\norm{\psi^{\frac{\a+1}{2}}}^{\frac{2(1-\lambda)}{\a+1}}_{L^2([0,T];L^{2_*}(\TT^d))}.
\\ & \leq\norm{\psi}^\lambda_{L^\infty([0,T];L^{\a+1}(\TT^d))}\left(\Tilde{c}\left(\norm{\psi}_{L^\infty([0,T];L^{\a+1}(\TT^d))}^\frac{\a+1}{2}+\norm{\nabla \psi^{\frac{\a+1}{2}}}_{L^2([0,T];L^{2}(\TT^d))}\right)\right)^{\frac{2(1-\lambda)}{\a+1}}.
\end{aligned}
\end{align}
Now H\"older's inequality, the inequality $(x+y)^2\leq 2(x^2+y^2)$ for all $x,y\in[0,\infty)$, the fact that $\lambda\in(0,1)$, $\a\in[1,\infty)$, and \eqref{proof/ moser 14} prove that, for the constant $C_0:=\bar{c}\Tilde{c}$, for $\bar{c}$ and $\Tilde{c}$ introduced in \eqref{proof/ moser 12} and \eqref{proof/ moser 15}, independent of $\epsilon,\eta\in(0,1)$ and $\alpha\geq1$, 
{\small
\begin{align}
\begin{aligned}\label{proof/ moser 16}
\E\Big[& \norm{\psi}_{L^q(\TT^d\times[0,T])}\Big]
\\   \leq & \E\left[\norm{\psi}_{L^\infty([0,T];L^{\a+1}(\TT^d))}\right]^\lambda\E\left[\left(\tilde{c}\norm{\psi}^{\a+1}_{L^\infty([0,T];L^{\a+1}(\TT^d))}+\Tilde{c}\norm{\nabla \psi^{\frac{\a+1}{2}}}^2_{L^2([0,T];L^{2}(\TT^d))}\right)^\frac{1}{\a+1}\right]^{1-\lambda}
\\  \leq & \E\left[\left(\tilde{c}\norm{\psi}^{\a+1}_{L^\infty([0,T];L^{\a+1}(\TT^d))}+\tilde{c}\norm{\nabla \psi^{\frac{\a+1}{2}}}^2_{L^2([0,T];L^{2}(\TT^d))}\right)^\frac{1}{\a+1}\right]
\\ 
\leq&
\frac{(\alpha+1)^{\frac{1}{\alpha+1}}}{1-\frac{1}{\alpha+1}}\, \big(\alpha^4\delta_{\alpha}^{-2}\big)^{\frac{1}{\alpha+1}}
        \Big(C_0\,(\inf\dot{\phi})^{-2}\big(\epsilon^2\,\|F_1^{\epsilon}\|_{\infty}\|F_3^\epsilon\|_{\infty}+\epsilon\,\|F_3^\epsilon\|_{\infty}\big)\Big)^{\frac{1}{\alpha+1}}
        \left(\delta_{\alpha}^{\frac{\alpha-1}{\alpha+1}}\!+\!\E\big[\|\psi\|_{L_t^{\!\alpha-1}\!L_x^{\!\alpha-1}}\big]^{\frac{\alpha-1}{\alpha+1}}\right).
\end{aligned}
\end{align}
}
Furthermore, thanks to \eqref{proof/ moser 12bis}, when $\alpha=1$ we can replace the last factor on the right hand side of \eqref{proof/ moser 16} simply with $1$.

We now iterate inequality \eqref{proof/ moser 16} along a suitable sequence of exponents $\alpha_k$ tending to infinity and small parameters $\delta_{k}=\delta_{\alpha_k}$ for $k\in\N$.
Namely, mimicking the relation between the exponents $q$, $\alpha$ and $\alpha-1$ in formula \eqref{proof/ moser 16}, which holds for any $\alpha\in[1,\infty)$, we define
\begin{align}\label{proof/ moser 17.1}
    \begin{aligned}
        \alpha_k=\frac{2+d}{d}(\alpha_{k-1}+2)\quad\text{and}\quad\beta_{k}=\alpha_{k-1}+1\quad\forall\,k\in\N^*,\quad\alpha_0=0.
    \end{aligned}
\end{align}
This gives
\begin{align}\label{proof/ moser 17.2}
    \begin{aligned}
        \alpha_k=(d+2)\big((1+\nicefrac{2}{d})^k-1\big)\quad\text{and}\quad\beta_{k}=(d+2)\big((1+\nicefrac{2}{d})^{k-1}-1\big)+1\quad\forall\,k\in\N^*.
    \end{aligned}
\end{align}
We also introduce the shorthand
\begin{align}\label{proof/ moser 17.3}
    \begin{aligned}
        u_{\beta_k}:=\frac{1}{\beta_k+1}\quad\forall\,k\in\N^*.
    \end{aligned}
\end{align}
Let us now define
\begin{align}\label{proof/ moser 17.4}
    \begin{aligned}
        \gamma:=\inf_{k\geq2}\prod_{i=2}^k\frac{\beta_i-1}{\beta_i+1}=\lim_{k\to\infty}\prod_{i=2}^k\frac{\beta_i-1}{\beta_i+1}\in(0,1).
    \end{aligned}
\end{align}
It is shown in \eqref{proof/ moser 21} below that indeed $\gamma>0$.
In turn we define
\begin{align}\label{proof/ moser 17.5}
    \begin{aligned}
        \delta_k:=k^{-\nicefrac{2}{\gamma}}\quad\forall\,k\in\N^*.
    \end{aligned}
\end{align}
Finally, let us also introduce the shorthand
\begin{align}\label{proof/ moser 17.6}
    \begin{aligned}
        R_\epsilon:=C_0\,(\inf\dot{\phi})^{-2}\big(\epsilon^2\,\|F_1^{\epsilon}\|_{\infty}\|F_3^\epsilon\|_{\infty}+\epsilon\,\|F_3^\epsilon\|_{\infty}\big).
    \end{aligned}
\end{align}

Iterating inequality \eqref{proof/ moser 16} over the sequence $\alpha_k$, using the inequality $(x+y)^{\nicefrac{1}{h}}\leq x^{\nicefrac{1}{h}}+y^{\nicefrac{1}{h}}$ for $h\geq 1$ and $x,y\geq0$ at each step, yields
{\small
\begin{align}\label{proof/ moser 18}
\begin{aligned}
    \E\Big[ \|\psi\|_{L^{\!\alpha_k}(\T^d\times[0,T])}\Big]
    \leq &\,
    \frac{u_{\beta_k}^{-u_\beta{k}}}{1-u_{\beta_k}}\beta_k^{4u_{\beta_k}}\delta_k^{-2u_{\beta_k}}R_\epsilon^{u_{\beta_k}}\left(\delta_k^{\frac{\beta_k-1}{\beta_k+1}}+\E\Big[\|\psi\|_{L^{\!\alpha_{k-1}}(\T^d\times[0,T])}\Big]^{\frac{\beta_k-1}{\beta_k+1}}\right)
    \\
    \leq &\,
    \frac{u_{\beta_k}^{-u_\beta{k}}}{1-u_{\beta_k}}\beta_k^{4u_{\beta_k}}\delta_k^{-2u_{\beta_k}}R_\epsilon^{u_{\beta_k}}\delta_k^{\frac{\beta_k-1}{\beta_k+1}}
    \\
    &+
    \frac{u_{\beta_k}^{-u_\beta{k}}}{1-u_{\beta_k}}\beta_k^{4u_{\beta_k}}\delta_k^{-2u_{\beta_k}}R_\epsilon^{u_{\beta_k}}
    \left(
    \frac{u_{\beta_{k-1}}^{-u_\beta{{k-1}}}}{1-u_{\beta_{k-1}}}\beta_{k-1}^{4u_{\beta_{k-1}}}\delta_{k-1}^{-2u_{\beta_{k-1}}}R_\epsilon^{u_{\beta_{k-1}}}
    \right)^{\frac{\beta_k-1}{\beta_k+1}}
    \\
    &\qquad\qquad\qquad\qquad\qquad\qquad\qquad\cdot
    \left(\delta_{k-1}^{\frac{\beta_{k-1}-1}{\beta_{k-1}+1}}+\E\Big[\|\psi\|_{L^{\!\alpha_{k-2}}(\T^d\times[0,T])}\Big]^{\frac{\beta_{k-1}-1}{\beta_{k-1}+1}}\right)^{\frac{\beta_k-1}{\beta_k+1}}
    \\
    \leq&
    \sum_{j=1}^k\,\,\,\delta_j^{\prod_{i=j}^k\frac{\beta_i-1}{\beta_i+1}}
    \times\,\,
    \prod_{l=j}^k\left(\frac{u_{\beta_l}^{-u_{\beta_l}}}{1-u_{\beta_l}}\,\,\beta_l^{4u_{\beta_l}}\,\,\delta_l^{-2u_{\beta_l}}\right)^{\prod_{i=l+1}^k\frac{\beta_i-1}{\beta_i+1}}
    \!\!\!\times\,
    R_\epsilon^{\sum_{l=j}^k\left(u_{\beta_l}\prod_{i=l+1}^k\frac{\beta_i-1}{\beta_i+1}\right)}.
\end{aligned}
\end{align}
}
We remark that in the last passage of the iteration, from $\alpha_1$ to $\alpha_0$, we do not pick up any integral of $\psi$ since by construction $\alpha_0=0$ (cf. formula \eqref{proof/ moser 12bis} and the comment after formula \eqref{proof/ moser 16}).
Finally, recalling \eqref{proof/ moser 17.4}-\eqref{proof/ moser 17.5}, we estimate
{\small
\begin{align}\label{proof/ moser 19}
\begin{aligned}
    \E\Big[& \|\psi\|_{L^{\alpha_k}(\TT^d\times[0,T])}\Big]
    \\
    \leq&
    \left(
    \sum_{j=1}^k\delta_j^{\gamma}\,\,
    R_\epsilon^{\sum_{l=j}^k\left(u_{\beta_l}\prod_{i=l+1}^k\frac{\beta_i-1}{\beta_i+1}\right)}
    \right)
    \times
    \sup_{k\in\N^*}
    \prod_{l=1}^k\left(\frac{u_{\beta_l}^{-u_{\beta_l}}}{1-u_{\beta_l}}\,\,\beta_l^{4u_{\beta_l}}\right)^{\prod_{i=l+1}^k\frac{\beta_i-1}{\beta_i+1}}
    \!\!\times
    \sup_{k\in\N^*}
    \prod_{l=1}^k\delta_l^{-2u_{\beta_l}\prod_{i=l+1}^k\frac{\beta_i-1}{\beta_i+1}}.
\end{aligned}
\end{align}
}

We now analyze the limit of the infinite sum and products in \eqref{proof/ moser 19}.
As anticipated in \eqref{proof/ moser 17.4} we have $\gamma\in(0,1)$.
Indeed, passing to logarithms, using that $\log(1-z)\geq-(1+d)z$ for $z\leq\frac{d}{1+d}$ and the expression \eqref{proof/ moser 17.2}, we find
\begin{equation}
    \label{proof/ moser 21}
    \log\left(\prod_{i=2}^k\frac{\beta_i-1}{\beta_i+1}\right)=\sum_{i=2}^k\log\left(1-\frac{2}{\beta_i+1}\right)\geq -2(1+d)\sum_{i=1}^\infty\frac{1}{\beta_i+1}>-\infty.
\end{equation}
In particular the series $\sum_{j=1}^{\infty}\delta_j^\gamma<\infty$ converges since $\delta_j=j^{-\nicefrac{2}{\gamma}}$.
Furthermore, for every $l\leq k\in\N^*$, we have 
\[
\gamma\, u_{\beta_l}\leq u_{\beta_l}\prod_{i=l+1}^k\frac{\beta_i-1}{\beta_i+1}\leq u_{\beta_l}.
\]
Therefore, by \eqref{proof/ moser 17.2} and \eqref{proof/ moser 21} and dominated convergence, for every $j\in\N^*$ the following limit exists and satisfies the inequality
{\small
\begin{equation}
    \label{proof/ moser 23}
    \gamma\sum_{l=j}^\infty u_{\beta_l}\leq\lim_{k\to\infty}\sum_{l=j}^k \left(u_{\beta_l}\prod_{i=l+1}^k\frac{\beta_i-1}{\beta_i+1}\right)
    =
    \sum_{l=j}^\infty \left(u_{\beta_l}\prod_{i=l+1}^\infty\frac{\beta_i-1}{\beta_i+1}\right)
    \leq
    \sum_{l=j}^\infty u_{\beta_l}
    <\infty.
\end{equation}
}
In particular we have the estimate, for every $j\in\N^*$,
\begin{equation}
    \label{proof/ moser 24}
    R_{\epsilon}^{\sum_{l=j}^k \left(u_{\beta_l}\prod_{i=l+1}^k\frac{\beta_i-1}{\beta_i+1}\right)}
    \leq
    \max\left\{1, R_{\epsilon}^{\sum_{l=1}^\infty u_{\beta_l}}\right\}.
\end{equation}
Furthermore, the first supremum on the right hand side of \eqref{proof/ moser 19} is finite since we compute, recalling the expression \eqref{proof/ moser 17.2},
{\small
\begin{align}\label{proof/ moser 26}
\begin{aligned}
    \log\left(
    \prod_{l=1}^k\left(\frac{u_{\beta_l}^{-u_{\beta_l}}}{1-u_{\beta_l}}\,\,\beta_l^{4u_{\beta_l}}\right)^{\prod_{i=l+1}^k\frac{\beta_i-1}{\beta_i+1}}\right)
    &=
    \sum_{l=1}^k\prod_{i=l+1}^k\frac{\beta_i-1}{\beta_i+1}\left(\frac{1}{\beta_l+1}\log(\beta_l+1)+\log(1+\nicefrac{1}{\beta_l})+\frac{4}{\beta_l+1}\log(\beta_l)\right)
    \\
    &\lesssim
    \sum_{l=1}^\infty\frac{1+\log(\beta_l)}{\beta_l}<\infty.
\end{aligned}
\end{align}
}
Similarly, the second supremum on the right hand side of \eqref{proof/ moser 19} is finite since we compute, recalling also \eqref{proof/ moser 17.4}-\eqref{proof/ moser 17.5},
{\small
\begin{align}\label{proof/ moser 28}
\begin{aligned}
    \log\left(
    \prod_{l=1}^k\delta_l^{-2u_{\beta_l}\prod_{i=l+1}^k\frac{\beta_i-1}{\beta_i+1}}
    \right)
    =
    \sum_{l=1}^k
    2u_{\beta_l}\prod_{i=l+1}^k\frac{\beta_i-1}{\beta_i+1}\log(\delta_l^{-1})
    \leq
    \sum_{l=1}^k\frac{4}{\gamma}\frac{\log(l)}{\beta_l+1}<\infty.
\end{aligned}
\end{align}
}

We now let $k\to\infty$ in \eqref{proof/ moser 19} and, by formulas \eqref{proof/ moser 21}-\eqref{proof/ moser 28} and dominated convergence, we obtain, for a \emph{numeric} constant $c^*\geq0$ only depending on the sequences chosen in \eqref{proof/ moser 17.1}-\eqref{proof/ moser 17.5},
{\small
\begin{align}\label{proof/ moser 29}
\begin{aligned}
    \E\Big[\|(\rho^{\epsilon,\eta}-\low)_-\|_{L^{\infty}(\T^d\times[0,T])}\Big]
    =
    \lim_{k\to\infty}\E\Big[ \|(\rho^{\epsilon,\eta}-\low)_-\|_{L^{\alpha_k}(\T^d\times[0,T])}\Big]
    \leq
    c^*\sum_{j=1}^\infty j^{-2}\,\,
    R_\epsilon^{\sum_{l=j}^\infty\left(u_{\beta_l}\prod_{i=l+1}^\infty\frac{\beta_i-1}{\beta_i+1}\right)}.
\end{aligned}
\end{align}
}
In particular, thanks to \eqref{proof/ moser 24}, if $R_{\epsilon}\to0$ then the right hand side of \eqref{proof/ moser 29} vanishes by dominated convergences.
In fact, formula \eqref{proof/ moser 17.1}-\eqref{proof/ moser 17.5} and tedious, but elementary, estimates with geometric series and logarithms allow to further quantify the exponent hitting $R_{\epsilon}$ in \eqref{proof/ moser 29}, and yield the estimate, for the same numeric constant $c^*$,
\begin{align}\label{proof/ moser 30}
\begin{aligned}
    \E\Big[\|(\rho^{\epsilon,\eta}-\low)_-\|_{L^{\infty}(\T^d\times[0,T])}\Big]
    \leq
    c^*\sum_{j=1}^\infty j^{-2}
    \begin{cases}
        R_{\epsilon}^{\nicefrac{d}{2}(1+\nicefrac{2}{d})^{2-j}}\,\,\,\,\,\qquad\quad\qquad\text{if $R_\epsilon\geq1$,}
        \\
        R_{\epsilon}^{\exp(-d(d+1))\nicefrac{1}{2}(1+\nicefrac{2}{d})^{1-j}}\quad\text{if $R_\epsilon\in(0,1]$.}
    \end{cases}
\end{aligned}
\end{align}

In conclusion, we note that the constant $C_0$ introduced in \eqref{proof/ moser 16} and featuring in the definition \eqref{proof/ moser 17.6} of $R_\epsilon$ is independent of $\eta\in(0,1)$.
Therefore we pass to the limit $\eta\to0$ and use Proposition \ref{proposition/kinetic solutions of dean-kawasaki depend continuously on the coefficients} and Fatou's Lemma in \eqref{proof/ moser 30} to obtain an identical estimate for the solution $\rho^{\epsilon}$ of the original equation.
This concludes the proof.
\end{proof}

We now go back to the the central limit theorem.
Applying Proposition \ref{proposition/moser iteration} to the solution $\rho^{\epsilon,\eta}$ of \eqref{equation/dean--kawasaki with smoothed coefficients 2} with the smoothed coefficients $\phi^\eta$, $\nu^\eta$, $\sigma^\eta$ given in Lemma \ref{lemma/ smooth approximations only near zero} and using Markov's inequality we readily obtain the following.
\begin{corollary}
\label{corollary/solutions are bounded by initial data with high probability}
In the setting of Corollary \ref{corollary/clt for the smoothed solutions}, for every $\epsilon,\eta\in(0,1)$ and every $\delta\in[0,\low)$, we have
    \begin{equation}
        \P\left(\underset{\T^d\times[0,T]}{\text{ess-inf }}\rho^{\epsilon,\eta}<\low\right)
        =
        \P\left(\underset{\T^d\times[0,T]}{\text{ess-sup }}\left(\rho^{\epsilon,\eta}-\low\right)_->\low-\delta\right)
        \leq
        \frac{c^*}{\low-\delta}\sum_{j=1}^{\infty}\frac{1}{j^2}\, 
        R_{\epsilon,\eta}^{\,\,\zeta\,(1+\nicefrac{2}{d})^{-j}},
    \end{equation}
where $c^*\in(0,\infty)$ is a \emph{numeric} constant and 
\begin{equation}
\label{formula/ reminder R epsilon 2}
    R_{\epsilon,\eta}:=C_0\,(\inf\dot{\phi}^\eta)^{-2}\big(\epsilon^2\,\|F_1^{\epsilon}\|_{\infty}\|F_3^\epsilon\|_{\infty}+\epsilon\,\|F_3^\epsilon\|_{\infty}\big),\quad
    \zeta:=\begin{cases}\frac{d}{2}(1+\nicefrac{2}{d})^2\,\qquad\quad\,\,\,\,\,\text{if $R_\epsilon\geq1$,}\\
e^{-d(d+1)}(\nicefrac{1}{2}+\nicefrac{1}{d})\quad\text{if $R_\epsilon\in(0,1)$,}\end{cases}
\end{equation}
for a constant $C_0=C_0(T,d,\phi,\sigma,\nu,\low)$ independent of $\epsilon,\eta\in(0,1)$.
\end{corollary}

\medskip

In turn, Corollary \ref{corollary/clt for the smoothed solutions} and the definition of $\Omega^{\epsilon,\eta}$ yield that, for every $\epsilon\in(0,1)$ and every $\eta\in(0,1)$ small enough,
{\small
\begin{align}
\begin{aligned}\label{equation/solution and smoothed solutions coincide with high probability}
        \P\bigg(\underset{\T^d\times[0,T]}{\text{ess-sup}}\left|v^\epsilon-v^{\epsilon,\eta}\right|\neq0\bigg)
        =
        \P\bigg(\underset{\T^d\times[0,T]}{\text{ess-sup}}\left|\rho^\epsilon-\rho^{\epsilon,\eta}\right|\neq0\bigg)
        \leq
        \P\Big(\left(\Omega^{\epsilon,\eta}\right)^c\Big)
        \leq
        \frac{c^*}{\low-\delta_{\eta}}\sum_{j=1}^{\infty}\frac{1}{j^2}\, 
        R_{\epsilon,\eta}^{\,\,\zeta\,(1+\nicefrac{2}{d})^{-j}}.
\end{aligned}
\end{align}
}
In particular we remark that, as we keep $\eta\in(0,1)$ fixed and we let $\epsilon\to0$ along the scaling regime \eqref{equation/scaling regime for noise}, the right hand side of \eqref{equation/solution and smoothed solutions coincide with high probability} vanishes by dominated convergence.

We finally establish our main result.

\begin{theorem}[Central limit theorem in probability]
\label{theorem/clt for fluctuations 3}
Let $\rho_0$ satisfy Assumption \ref{assumption/assumption I1}(ii), i.e. a random positive constant, let $(\xi^\epsilon)_{\epsilon>0}$ satisfy Assumption \ref{assumption/assumption N2}, and let $\phi,\nu,\sigma$ satisfy Assumption \ref{assumption/assumption C1} and \ref{assumption/assumption C2 weak}, for some $p\geq2$ and $m\geq1$.
For every $\epsilon>0$, let $\rho^\epsilon$ be the stochastic kinetic solution to the generalized Dean--Kawasaki equation \eqref{equation/generalized dean--kawasaki equation 3} with initial data $\rho_0$ and let $\Bar{\rho}\equiv\rho_0$ be the solution to the zero noise limit \eqref{equation/hydrodynamic limit equation 1}.
Let $v$ be the solution to the Langevin equation \eqref{equation/OU equation 1} with noise $\xi=\lim_{\epsilon\to0}\xi^\epsilon$.
Along a scaling regime where $\epsilon\to0$ and $\xi^\epsilon\to\xi$ such that
\begin{equation}
    \label{equation/scaling regime for noise weak clt}
    \lim_{\epsilon\to0}\,\,\sqrt{\epsilon}\,
    \Big(\!\left\|F_1^\epsilon\right\|_{\infty}\!\!+\left\|F_2^\epsilon\right\|_{\infty}+\|F_3^\epsilon\|_{\infty}\Big)=0,
\end{equation}
for the nonequilibrium fluctuations we have
\[v^\epsilon=\epsilon^{-\nicefrac{1}{2}}(\rho^\epsilon-\Bar{\rho})\to v\quad\text{ in $L^{\tau}_{\text{loc}}([0,\infty);H^{-\beta}(\T^d)))$ in probability,}\]
for $\tau=2$ or $\tau=\infty$,  for any $\beta>\frac{d}{2}$ or $\beta>1+\frac{d}{2}$ respectively, with explicit rate of convergence given in \eqref{formula/ rate of convergence clt in probability} below.
\end{theorem}

\begin{proof}
We fix $\eta_0\in(0,1)$ small enough.
For any $a>0$, using formula \eqref{equation/solution and smoothed solutions coincide with high probability} and Markov's inequality we estimate
{\small
\begin{align}\label{theorem/clt for fluctuations 3/proof 1}
    \begin{aligned}
        \P\Big(\|v^\epsilon-v\|_{L^{\tau}([0,T];H^{-\beta}(\T^d))}>a\Big)
        \leq&
        \P\left(\|v^\epsilon-v\|_{L^{\tau}([0,T];H^{-\beta}(\T^d))}>a,\,\underset{\T^d\times[0,T]}{\text{ess-sup}}\left|v^\epsilon-v^{\epsilon,\eta_0}\right|\neq0\right)
        \\
        &+
        \P\left(\|v^\epsilon-v\|_{L^{\tau}([0,T];H^{-\beta}(\T^d))}>a,\,\underset{\T^d\times[0,T]}{\text{ess-sup}}\left|v^\epsilon-v^{\epsilon,\eta_0}\right|=0\right)
        \\
        \leq&\,
        \P\left(\underset{\T^d\times[0,T]}{\text{ess-sup}}\left|v^\epsilon-v^{\epsilon,\eta_0}\right|\neq0\right)
        +
        \P\left(\|v^{\epsilon,\eta_0}-v\|_{L^{\tau}([0,T];H^{-\beta}(\T^d))}>a\right)
        \\
        \leq &\,
        \P\left(\underset{\T^d\times[0,T]}{\text{ess-sup}}\left|v^\epsilon-v^{\epsilon,\eta_0}\right|\neq0\right)
        +
        \frac{1}{a}\E\left[\|v^{\epsilon,\eta_0}-v\|_{L^{\tau}([0,T];H^{-\beta}(\T^d))}^2\right]^{\frac{1}{2}}.
    \end{aligned}
\end{align}
}
We keep $\eta_0$ fixed and we let $\epsilon\to0$.
The first term on the right hand side of \eqref{theorem/clt for fluctuations 3/proof 1} vanishes because of formula \eqref{equation/solution and smoothed solutions coincide with high probability}, the scaling regime chosen and dominated convergence.
The second term vanishes because of the scaling regime and Corollary \ref{corollary/clt for the smoothed solutions}.
Since $a>0$ is arbitrary, we conclude that $v^\epsilon\to v$ in probability.

To obtain an explicit convergence rate, we consider again the last line of formula \eqref{theorem/clt for fluctuations 3/proof 1}.
The constant $\delta_\eta$ defined in \eqref{equation/delta_eta} decreases to zero as $\eta\to0$, and therefore is eventually smaller than $\nicefrac{\low}{2}$ for any $\eta\in(0,\eta_0]$, for a suitable $\eta_0$ which only depends on $\low$, the original coefficients $\phi,\eta,\sigma$ and the approximation procedure employed in Lemma \ref{lemma/ smooth approximations only near zero}.
Fixed such an $\eta_0\in(0,1)$, we can explicitly compute the constant $C$ appearing in formula \eqref{equation/rate of convergence 2}.
Indeed, the constant depends on this fixed $\eta_0$ only through the original coefficients $\phi,\nu,\sigma$ and the approximation procedure used in Lemma \ref{lemma/ smooth approximations only near zero}.
Furthermore, the definition of $R_{\epsilon,\eta_0}$ in \eqref{formula/ reminder R epsilon 2} depends on this fixed $\eta_0$ only through $\inf\dot{\phi}^{\eta_0}$, and we can require this to be greater than some positive number, say $\nicefrac{\low}{3}$.
Applying this reasoning and plugging \eqref{equation/rate of convergence 2} and \eqref{equation/solution and smoothed solutions coincide with high probability}, with this $\eta_0$, into the last line of \eqref{theorem/clt for fluctuations 3/proof 1} we obtain, for arbitrary $a>0$, 
{\small
\begin{align}\label{formula/ rate of convergence clt in probability}
\begin{aligned}
    \P\Big(&\|v^\epsilon-v\|_{L^{\tau}([0,T];H^{-\beta}(\T^d))}>a\Big)
    \\
    \leq &
    \,C\sum_{j=1}^{\infty}\frac{1}{j^2}\, 
        \big(\epsilon^2\,\|F_1^{\epsilon}\|_{\infty}\|F_3^\epsilon\|_{\infty}+\epsilon\,\|F_3^\epsilon\|_{\infty}\big)^{\,\,e^{-d(d+1)}(1+\nicefrac{2}{d})^{1-j}}
        \\
    &
    +\frac{C}{a}\,
\left(\epsilon\left(|F_1^\epsilon|_{\infty}+|F_3^\epsilon|_{\infty}\right)^2+\epsilon^{\nicefrac{1}{2}}|F_3^\epsilon|_{\infty}^{\nicefrac{1}{2}}\right)
    \left(1+\epsilon\,|F_3^\epsilon|_\infty\right)^{g+\frac{d}{2}(p+m)}
    \left(1+\E\left[\|\rho_0\|_{L^1(\T^d)}^{m+p-1}+\|\rho_0\|_{L^p(\T^d)}^p\right]\right)
    \\
    &
    +\frac{C}{a}
    \sum_{n\in\Z^d}n^{(2-\frac{4}{\tau})-2\beta}  \sum_{k} \int_0^T\left|\int_{\T^d} e^{i2\pi nx}\,(f_k^\epsilon-f_k)\,dx\right|^2 ds,
\end{aligned}
\end{align}}
for a constant $C=C(\low,\phi,\nu,\sigma,T,d,p,\beta,\bar\rho)$ depending only on the initial data and the original coefficients, which can be computed explicitly in terms of $\low=\text{ess-inf}\rho_0$ and the constants $c$ featuring in Assumption \ref{assumption/assumption C1} and \ref{assumption/assumption C2 weak}.
\end{proof}


\section*{Acknowledgments}
The authors have no competing interests.
Andrea Clini is supported by the EPSRC Centre for Doctoral Training in Mathematics of Random Systems: Analysis, Modelling and Simulation (EP/S023925/1).
Benjamin Fehrman is supported by the EPSRC Early Career Fellowship (EP/V027824/1).
The authors would like to deeply thank J. A. Carrillo for comments and helpful discussions.




\begin{thebibliography}{CKNP22}

\bibitem[BD19]{BudDup2019}
Amarjit Budhiraja and Paul Dupuis.
\newblock {\em Analysis and approximation of rare events}, volume~94 of {\em
  Probability Theory and Stochastic Modelling}.
\newblock Springer, New York, 2019.
\newblock Representations and weak convergence methods.

\bibitem[BDM08]{budhiraja-dupuis-maroula-large-deviations}
Amarjit Budhiraja, Paul Dupuis, and Vasileios Maroulas.
\newblock {Large deviations for infinite dimensional stochastic dynamical
  systems}.
\newblock {\em The Annals of Probability}, 36(4):1390 -- 1420, 2008.

\bibitem[BDS19]{BudDupSal}
Amarjit Budhiraja, Paul Dupuis, and Michael Salins.
\newblock Uniform large deviation principles for {B}anach space valued
  stochastic evolution equations.
\newblock {\em Trans. Amer. Math. Soc.}, 372(12):8363--8421, 2019.

\bibitem[BKL95]{benois-kipnis-landim-large-deviations-from-the-hydrodynamical-limit-of-mean-zero-asymmetric}
O.~Benois, C.~Kipnis, and C.~Landim.
\newblock {Large deviations from the hydrodynamical limit of mean zero
  asymmetric zero range processes}.
\newblock {\em Stochastic Processes and their Applications}, 55(1):65--89,
  January 1995.

\bibitem[CFIR23]{cornalba_fischer_ingmanns_raithel_density_fluctuations_in_weakly_interacting}
Federico Cornalba, Julian Fischer, Jonas Ingmanns, and Claudia Raithel.
\newblock Density fluctuations in weakly interacting particle systems via the
  dean-kawasaki equation, 2023.

\bibitem[CKNP22]{Chen_nualart_pu_poincare_inequality_and_central_limit}
Le~Chen, Davar Khoshnevisan, David Nualart, and Fei Pu.
\newblock Central limit theorems for parabolic stochastic partial differential
  equations.
\newblock {\em Annales de l'Institut Henri Poincar{\'e}, Probabilit{\'e}s et
  Statistiques}, 2022.

\bibitem[Cli23]{clini-porous-media-equations}
Andrea Clini.
\newblock Porous media equations with nonlinear gradient noise and dirichlet
  boundary conditions.
\newblock {\em Stochastic Processes and their Applications}, 2023.

\bibitem[DE97]{DupEll1997}
Paul Dupuis and Richard Ellis.
\newblock {\em A weak convergence approach to the theory of large deviations}.
\newblock Wiley Series in Probability and Statistics: Probability and
  Statistics. John Wiley \& Sons, Inc., New York, 1997.

\bibitem[DFG20]{dirr-fehrman-gess-conservative-stochastic-pde-and-fluctuations-of-the-symmetric}
Nicolas Dirr, Benjamin Fehrman, and Benjamin Gess.
\newblock Conservative stochastic pde and fluctuations of the symmetric simple
  exclusion process, 2020.

\bibitem[DKP22]{djurdjevac-kremp-perkowski-weak-error-analysis}
Ana Djurdjevac, Helena Kremp, and Nicolas Perkowski.
\newblock Weak error analysis for a nonlinear spde approximation of the
  dean-kawasaki equation, 2022.

\bibitem[DPZ92]{da_prato_zabczyk_1992}
Guiseppe Da~Prato and Jerzy Zabczyk.
\newblock {\em Stochastic Equations in Infinite Dimensions}.
\newblock Encyclopedia of Mathematics and its Applications. Cambridge
  University Press, 1992.

\bibitem[DVNZ20]{vences-nualart-zheng-central-limit-theorem-for-the-stochastic-wave}
Francisco Delgado-Vences, David Nualart, and Guangqu Zheng.
\newblock {A Central Limit Theorem for the stochastic wave equation with
  fractional noise}.
\newblock {\em Annales de l'Institut Henri Poincaré, Probabilités et
  Statistiques}, 56(4):3020 -- 3042, 2020.

\bibitem[FG20]{fehrman-gess-large-deviations-for-conservative}
Benjamin Fehrman and Benjamin Gess.
\newblock Large deviations for conservative stochastic {PDE} and
  non-equilibrium fluctuations.
\newblock {\em pre-print}, 2020.

\bibitem[FG21]{fehrman-gess-Well-posedness-of-the-Dean-Kawasaki-and-the-nonlinear-Dawson-Watanabe-equation-with-correlated-noise}
Benjamin Fehrman and Benjamin Gess.
\newblock Well-posedness of the {Dean}--{Kawasaki} and the nonlinear
  {Dawson}--{Watanabe} equation with correlated noise, 2021.

\bibitem[FGG22]{fehrman_gess_gvalani_ergodicity}
Benjamin Fehrman, Benjamin Gess, and Rishabh~S. Gvalani.
\newblock Ergodicity and random dynamical systems for conservative spdes, 2022.

\bibitem[GLP99]{giacomin-lebowitz-presutti-Deterministic-and-stochastic-hydrodynamic-equations-arising-from-simple-microscopic-model-systems}
Giambattista Giacomin, Joel Lebowitz, and Errico Presutti.
\newblock Deterministic and stochastic hydrodynamic equations arising from
  simple microscopic model systems.
\newblock In {\em Mathematical Surveys and Monographs}, pages 107--152.
  American Mathematical Society, nov 1999.

\bibitem[Hai14]{Hairer-regulairity-structures}
M.~Hairer.
\newblock A theory of regularity structures.
\newblock {\em Inventiones mathematicae}, 198(2):269--504, mar 2014.

\bibitem[HLW18]{Hu-li-wang-central-limit}
Shulan Hu, Ruinan Li, and Xinyu Wang.
\newblock Central limit theorem and moderate deviations for a class of
  semilinear spdes.
\newblock {\em arXiv: Probability}, 2018.

\bibitem[HNV20]{huang-nualart-viitasaari-a-central-limit-theorem}
Jingyu Huang, David Nualart, and Lauri Viitasaari.
\newblock A central limit theorem for the stochastic heat equation.
\newblock {\em Stochastic Processes and their Applications},
  130(12):7170--7184, 2020.

\bibitem[HNVZ19]{huang-nualart-viitasaari-gaussian-fluctuations}
Jingyu Huang, David Nualart, Lauri Viitasaari, and Guangqu Zheng.
\newblock Gaussian fluctuations for the stochastic heat equation with colored
  noise.
\newblock {\em Stochastics and Partial Differential Equations: Analysis and
  Computations}, 8:402--421, 2019.

\bibitem[KL99]{kipnis-landim-scaling-limits}
C.~Kipnis and C.~Landim.
\newblock {\em {Scaling Limits of Interacting Particle Systems}}.
\newblock Springer, 1999.

\bibitem[KLvR19]{lehmann_konarovsky_vonrenesse_On_deankawasaki_dynamics_smooth_drift}
Vitalii Konarovskyi, Tobias Lehmann, and Max von Renesse.
\newblock On dean{\textendash}kawasaki dynamics with smooth drift potential.
\newblock {\em Journal of Statistical Physics}, 178(3):666--681, nov 2019.

\bibitem[Kry12]{krylov_ito_formula}
Nicolai Krylov.
\newblock A relatively short proof of it\^o's formula for spdes and its
  applications.
\newblock {\em Stochastic Partial Differential Equations: Analysis and
  Computations}, 1, 08 2012.

\bibitem[Lie96]{Lieberman_parabolic}
Gary~M. Lieberman.
\newblock {\em Second order parabolic differential equations}.
\newblock World Scientific, 1996.

\bibitem[RY99]{revuz_yor_book}
{Daniel} Revuz and {Marc} Yor.
\newblock {\em Continuous martingales and Brownian motion}.
\newblock Number 293 in Grundlehren der mathematischen Wissenschaften.
  Springer, Berlin [u.a.], 3. ed edition, 1999.

\bibitem[Sla21]{Slavik-large-and-moderate-deviations}
Jakub Slav'ik.
\newblock Large and moderate deviations principles and central limit theorem
  for the stochastic 3d primitive equations with gradient-dependent noise.
\newblock {\em Journal of Theoretical Probability}, 2021.

\bibitem[Spo12]{spohn-large-scale-dynamics}
H.~Spohn.
\newblock {\em {Large Scale Dynamics of Interacting Particles}}.
\newblock Springer, 2012.

\bibitem[Vaz07]{vasquez_porous_media}
Juan~Luis Vazquez.
\newblock {\em The Porous Medium Equation: Mathematical Theory}.
\newblock Oxford Science Publications, 2007.

\bibitem[WWZ22]{wang_wu_zhang_dean_kawasaki_equation_with_singular_nonlocal_interactions}
Likun Wang, Zhengyan Wu, and Rangrang Zhang.
\newblock Dean-kawasaki equation with singular non-local interactions and
  correlated noise, 2022.

\bibitem[ZZG19]{zhang-zhou-guo-stochastic-2d-primitive-equations-central-limit}
Rangrang Zhang, Guoli Zhou, and Boling Guo.
\newblock Stochastic 2d primitive equations: Central limit theorem and moderate
  deviation principle.
\newblock {\em Computers \& Mathematics with Applications}, 77(4):928--946,
  2019.

\end{thebibliography}
\end{document}